\documentclass[11pt]{article}

\usepackage[margin=1in]{geometry}

\usepackage[T1]{fontenc}
\usepackage[utf8]{inputenc}
\usepackage{dsfont}

\usepackage{amsmath,amssymb,amsfonts,amsthm}
\usepackage{mathtools}
\usepackage{bm} 
\usepackage{mathrsfs}
\usepackage[only,llbracket,rrbracket]{stmaryrd}
\usepackage{bbm}   

\usepackage{amsfonts}
\usepackage{graphicx}

\usepackage{booktabs}
\usepackage{multirow}
\usepackage{float}
\usepackage{tikz}
\usepackage{sidecap}
\usepackage{caption}

\usepackage{microtype}
\usepackage{xcolor}
\usepackage{nicefrac}
\usepackage[shortlabels]{enumitem}
\usepackage{comment}
\usepackage[normalem]{ulem}   

\usepackage[numbers,sort&compress]{natbib}

\usepackage{url}
\usepackage{hyperref}
\usepackage[capitalise,sort]{cleveref}

\makeatletter
\newcounter{subfig}
\@addtoreset{subfig}{figure}

\newcommand{\subfigcaption}[2]{%
  \refstepcounter{subfig}%
  \label{#1}%
  {\small\textbf{(\alph{subfig})}~#2}%
}
\makeatother

\usepackage{todonotes}

\crefname{enumi}{item}{items}
\crefname{figure}{Figure}{Figures}
\crefname{equation}{}{}

\theoremstyle{definition}

\newtheorem{assumption}{Assumption}

\crefname{case}{Case}{Cases}

\theoremstyle{plain}
\crefname{corollary}{Corollary}{Corollaries}
\newtheorem{theorem}{Theorem}
\newtheorem{proposition}{Proposition}

\newtheorem{remark}{Remark}
\newtheorem{definition}{Definition}
\newtheorem{corollary}{Corollary}

\DeclareMathOperator{\id}{Id}

\newcommand{\R}{\mathbb{R}}

\newcommand{\simcost}{\mathrm{sim}}

\newcommand{\lrbr}[1]{\llbracket #1 \rrbracket}
\newcommand{\bbE}{\mathbb{E}}

\usepackage{tikz}
\usetikzlibrary{matrix,chains,positioning,decorations.pathreplacing,arrows}
\usetikzlibrary{shapes,arrows}
\tikzset{
	font={\fontsize{9pt}{12}\selectfont}}
\usetikzlibrary{calc}

\ExplSyntaxOn

\bool_new:N \g_noteobserve

\NewDocumentCommand{\setnote}{}{
  \bool_gset_true:N \g_noteobserve
}

\NewDocumentCommand{\setobserve}{}{
  \bool_gset_false:N \g_noteobserve
}

\NewDocumentCommand{\nobs}{ o }{
  \IfValueT{#1}{
    \str_if_eq:noTF {note} {#1} {
      \bool_gset_true:N \g_noteobserve
    } {
      \str_if_eq:noTF {Note} {#1} {
        \bool_gset_true:N \g_noteobserve
      } {
        \bool_gset_false:N \g_noteobserve
      }
    }
  }
  \bool_if:nTF { \g_noteobserve } {
    \bool_gset_false:N \g_noteobserve
    note
  } {
    \bool_gset_true:N \g_noteobserve
    observe
  }
  \IfValueF{#1}{~}
}

\NewDocumentCommand{\Nobs}{ o }{
  \IfValueT{#1}{
    \str_if_eq:noTF {note} {#1} {
      \bool_gset_true:N \g_noteobserve
    } {
      \str_if_eq:noTF {Note} {#1} {
        \bool_gset_true:N \g_noteobserve
      } {
        \bool_gset_false:N \g_noteobserve
      }
    }
  }
  \bool_if:nTF { \g_noteobserve } {
    \bool_gset_false:N \g_noteobserve
    Note
  } {
    \bool_gset_true:N \g_noteobserve
    Observe
  }
  \IfValueF{#1}{~}
}

\int_new:N \g_furthermore

\NewDocumentCommand{\Moreover}{ o o }{
  \IfValueT{#1}{
    \str_case:nn {#1} {
      {Furthermore} {\int_set:Nn {\g_furthermore} {0}}
      {Moreover} {\int_set:Nn {\g_furthermore} {1}}
      {In~addition} {\int_set:Nn {\g_furthermore} {2}}
      {note} {\bool_gset_true:N \g_noteobserve}
      {observe} {\bool_gset_false:N \g_noteobserve}
    }
    \IfValueT{#2}{
      \str_case:nn {#2} {
        {Furthermore} {\int_set:Nn {\g_furthermore} {0}}
        {Moreover} {\int_set:Nn {\g_furthermore} {1}}
        {In~addition} {\int_set:Nn {\g_furthermore} {2}}
        {note} {\bool_gset_true:N \g_noteobserve}
        {observe} {\bool_gset_false:N \g_noteobserve}
      }
    }
  }
  \int_case:nn { \int_mod:nn {\g_furthermore} {3} } {
    { 0 } { Furthermore,~\nobs that}
    { 1 } { Moreover,~\nobs that}
    { 2 } { In~addition,~\nobs that}
  }
  \int_incr:N \g_furthermore
  \IfValueF{#1}{~}
}

\bool_new:N \g_hencetherefore

\NewDocumentCommand{\hence}{}{
  \bool_if:nTF { \g_hencetherefore } {
    \bool_gset_false:N \g_hencetherefore
    hence~
  } {
    \bool_gset_true:N \g_hencetherefore
    therefore~
  }
}

\NewDocumentCommand{\Hence}{}{
  \bool_if:nTF { \g_hencetherefore } {
    \bool_gset_false:N \g_hencetherefore
    Hence,~we~obtain~
  } {
    \bool_gset_true:N \g_hencetherefore
    Therefore,~we~obtain~
  }
}

\seq_new:N \g_cflist_loaded
\seq_new:N \g_cflist_pending

\NewDocumentCommand{\cfadd}{ m }
{
	\seq_if_in:NnF \g_cflist_loaded { #1 } {
		\seq_if_in:NnF \g_cflist_pending { #1 } {
			\seq_gput_right:Nn \g_cflist_pending { #1 }
		}
	}
}

\NewDocumentCommand{\cfconsiderloaded}{ m }{
	\seq_gput_right:Nn \g_cflist_loaded {#1}
}

\NewDocumentCommand{\cfremove}{ m }
{
	\seq_gremove_all:Nn \g_cflist_pending { #1 }
}

\NewDocumentCommand{\cfload}{ o }
{
	\seq_if_empty:NTF \g_cflist_pending {\unskip} {
		(cf.\ \cref{\seq_use:Nn \g_cflist_pending {,}})\IfValueTF{#1}{#1~}{\unskip}
		\seq_gconcat:NNN \g_cflist_loaded \g_cflist_loaded \g_cflist_pending
		\seq_gclear:N \g_cflist_pending
	}
}

\NewDocumentCommand{\cfclear} {} {
	\seq_gclear:N \g_cflist_loaded
	\seq_gclear:N \g_cflist_pending
}

\NewDocumentCommand{\cfout}{ o }
{
	\seq_if_empty:NTF \g_cflist_pending {\unskip} {
		(cf.\ \cref{\seq_use:Nn \g_cflist_pending {,}})\IfValueTF{#1}{#1~}{\unskip}
		\seq_gclear:N \g_cflist_pending
	}
}

\NewDocumentCommand{\ifnocf} { m } {
	\seq_if_empty:NT \g_cflist_pending { #1 }
}

\ExplSyntaxOff

\NewDocumentEnvironment{cproof}{m}
{\begin{proof}[Proof of \cref{#1}]}%
	{\noindent The proof of \cref{#1} is thus complete.
\end{proof}}

\NewDocumentEnvironment{cproof2}{m}
{\begin{proof}[Proof of \cref{#1}]}%
	{\noindent This completes the proof of \cref{#1}.
\end{proof}}

\usepackage{accents}

\renewcommand{\dot}[1]{\overset{\scalebox{1.25}{$\,.$}}{#1}}

\begin{document}

\title{Dynamics Over Landscape: The Emergence of Linear Separability via Spectral Alignment in Contrastive Learning}

\author{Jeff Calder\footnote{University of Minnesota, jwcalder@umn.edu}, \quad Wonjun Lee\footnote{The Ohio State University, lee.8222@osu.edu}}




\maketitle

\begin{abstract}
Contrastive learning effectively clusters data despite a loss landscape filled with poor solutions, a success that is heavily dependent on the choice of data augmentations. How optimization consistently finds meaningful patterns remains an open question. We show this success stems from training dynamics rather than the loss function alone. Crucially, under a highly specific structural assumption governing the connectivity and variance of the data augmentations, we prove that once a critical spectral alignment threshold is reached, data features inevitably and rapidly separate into distinct clusters. We establish this mechanism for both discrete datasets and the macroscopic continuum limit, modeling latent dynamics as a Wasserstein gradient flow to demonstrate that this separation persists as the number of data points approaches infinity. We hypothesize that natural training dynamics inherently drive the system toward this critical state. We extensively validate this empirically across four diverse domains (synthetic shapes, images, text, and PDEs). In every setting, a sharp increase in this spectral quantity consistently precedes clean data separation, revealing that contrastive learning's success is governed by a dynamically emerging trigger tightly coupled to the underlying augmentation structure.
\end{abstract}

\section{Introduction}

Unsupervised learning of effective representations for data is one of the most fundamental problems in machine learning, especially in the context of image data. The widely successful \emph{discriminative approach} to learning representations of data is most similar to fully supervised learning, where features are extracted by a backbone convolutional neural network, except that the fully supervised task is replaced by an unsupervised or \emph{self-supervised} task that can be completed without labeled training data. 

Many successful discriminative representation learning methods are based around the idea of finding a feature map that is \emph{invariant} to a set of transformations (i.e., data augmentations) that are expected to be present in the data. For image data, the transformations may include image scaling, rotation, cropping, color jitter, Gaussian blurring, and adding noise, though the question of which augmentations give the best features is not trivial \cite{tian2020makes}. Invariant feature learning methods include VICReg \cite{bardes2021vicreg}, Bootstrap Your Own Latent (BYOL) \cite{grill2020bootstrap}, Siamese neural networks \cite{chicco2021siamese}, and contrastive learning techniques such as SimCLR \cite{chen2020simple} (see also \cite{hadsell2006dimensionality, dosovitskiy2014discriminative, oord2018representation, bachman2019learning}).

In contrastive learning, the primary self-supervised task is to differentiate between positive and negative pairs of data instances. The goal is to find a feature map for which positive pairs have maximally similar features, while negative pairs have maximally different features. The positive and negative examples do not necessarily correspond to classes. In SimCLR, positive pairs are images that are the same up to a transformation, while all other pairs are negative pairs. Contrastive learning has also been successfully applied in supervised \cite{khosla2020supervised} and semi-supervised contexts \cite{li2021comatch,yang2022class,singh2021clda,zhang2022semi,lee2022contrastive,kim2021selfmatch, ji2023power}, and has been used for learning Lie Symmetries of partial differential equations \cite{mialon2023self} (for a survey see \cite{le2020contrastive}).

All invariance-based feature extraction techniques must address the fundamental problem of dimensional collapse, whereby a method learns the trivial constant map $f(x)=c$ (or a very low rank map), which is invariant to \emph{all} transformations, but not informative or descriptive. There are various ways to prevent dimensional collapse. In contrastive learning, the role of the negative pairs is to prevent collapse by creating repulsion terms in the latent space; however, full or partial collapse can still occur \cite{jing2021understanding,zhang2022does,shen2022connect,li2022understanding}. In BYOL, collapse is prevented by halting backpropagation in certain parts of the loss and incorporating temporal averaging. In VICReg, additional terms are added to the loss function to maintain variance in each latent dimension, as well as to decorrelate variables. 

\begin{figure}[!t]
\centering

\begin{minipage}[t]{0.32\linewidth}
\centering
\includegraphics[width=\linewidth]{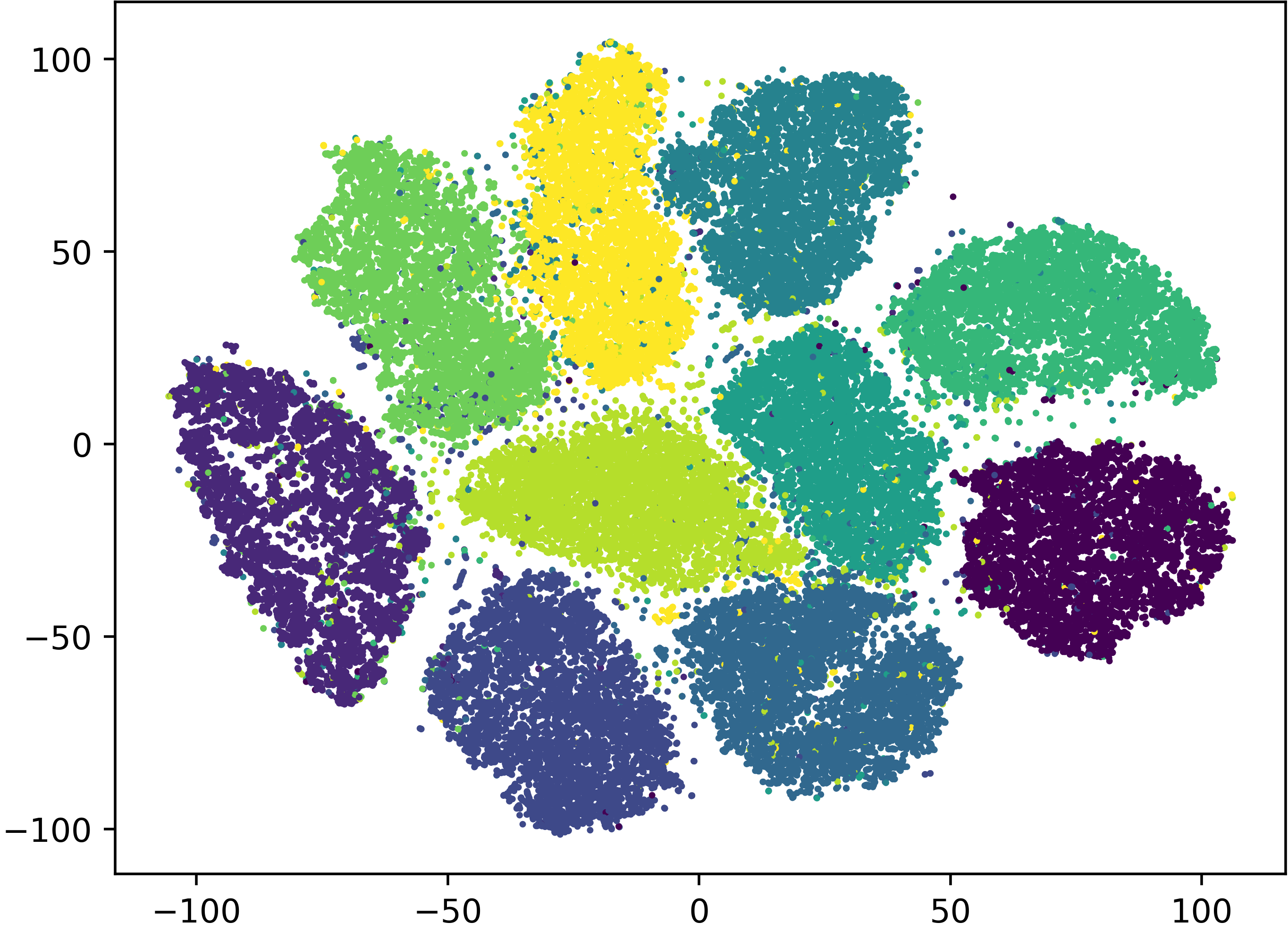}
\par\subfigcaption{fig:mnist_raw}{MNIST: pixel space}
\end{minipage}\hfill
\begin{minipage}[t]{0.32\linewidth}
\centering
\includegraphics[width=\linewidth]{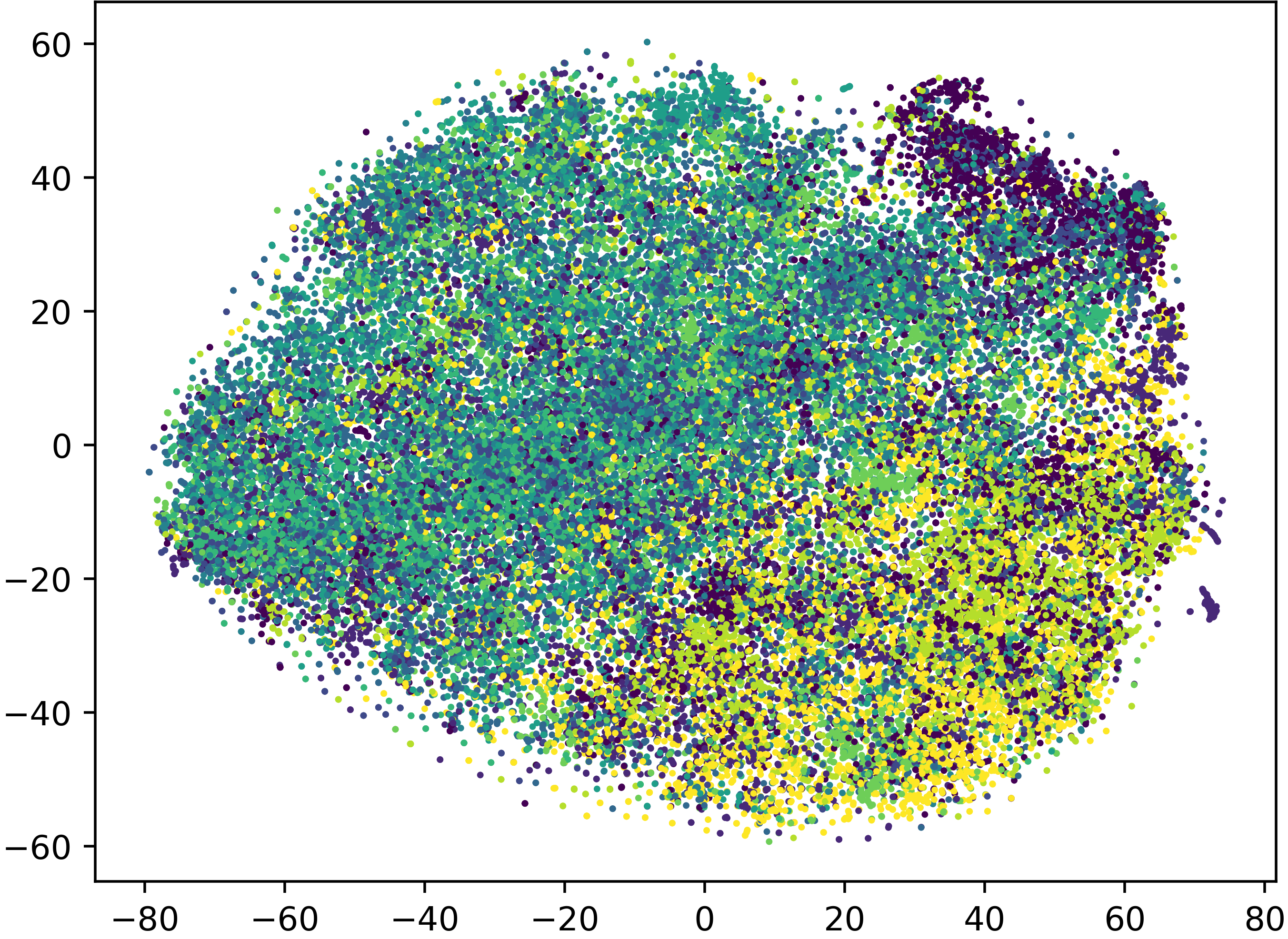}
\par\subfigcaption{fig:cifar_raw}{CIFAR-10: pixel space}
\end{minipage}\hfill
\begin{minipage}[t]{0.32\linewidth}
\centering
\includegraphics[width=\linewidth]{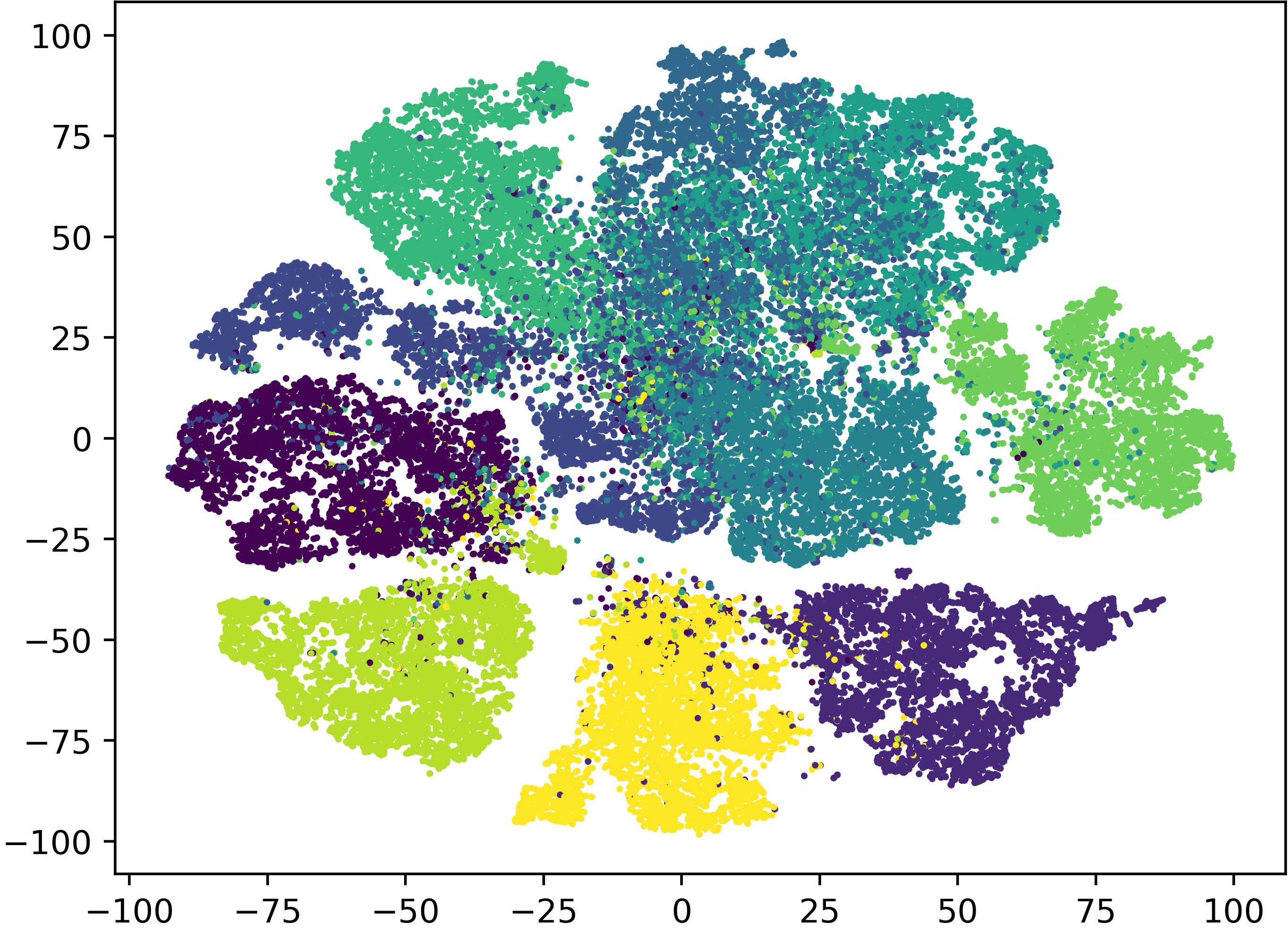}
\par\subfigcaption{fig:cifar_simclr}{CIFAR-10: SimCLR}
\end{minipage}

\caption{t-SNE visualizations of the MNIST and CIFAR-10 data sets. In (a) and (b) the images are represented by the raw pixels, while (c) gives a visualization of the SimCLR embedding. This illustrates how SimCLR is able to uncover clustering structure in data sets.}
\label{fig:demo}
\end{figure}

Provided dimensional collapse does not occur, a fundamental unresolved question surrounding many feature learning methods is: 

    \begin{quote}
        \emph{why do they work so well at producing embeddings that uncover key features and patterns in complex data sets? }
    \end{quote}

\noindent As a simple example, consider \cref{fig:demo}. In (a) and (b) we show t-SNE \cite{van2008visualizing} visualizations of the MNIST \cite{deng2012mnist} and CIFAR-10 \cite{krizhevsky2009learning} data sets, respectively, using their pixel representations. We can see that visual features are not required on MNIST, which is highly preprocessed, while for CIFAR-10 the pixel representations are largely uninformative, and feature representations are essential. In (c) we show a t-SNE visualization of the latent embedding of the SimCLR method applied to CIFAR-10, which indicates that SimCLR has uncovered a strong clustering structure in CIFAR-10 that was not present in the pixel representation.  

The goal of this paper is to provide a framework that can begin to address this question, and in particular, to explain \cref{fig:demo}. To do this, we assume the data follows a \emph{corruption} model, where the observed data is derived from some clean data with distribution $\mu$ that is highly structured or clustered in some way (e.g., follows the manifold assumption with a clustered density). The observed data is then obtained by applying transformations at random from a set of augmentations $\mathcal{T}$ to the clean data points (i.e., taking different views of the data), producing a corrupted distribution $\tilde \mu$. Crucially, our analysis relies on a highly specific structural assumption governing how these augmentations connect data points, which we formalize through an \emph{augmentation graph}. We assume that the transformations create bounded variance within local neighborhoods while preserving a necessary degree of spectral connectivity between related views. The main question that motivated our work is that of understanding what properties of the original clean data distribution $\mu$ can be uncovered by unsupervised contrastive feature learning techniques? That is, once an invariant feature map $f:\R^D \to \R^d$ is learned, is the latent distribution $f_\#\tilde \mu$ similar in any way to the clean distribution $\mu$, or can it be used to deduce any geometric or topological properties of $\mu$?   

To answer this, we explain how meaningful feature structure and linear separability emerge from neural network training under contrastive learning. The central object of interest in our analysis is the \emph{time-evolving neural kernel}, which governs how gradient information propagates across data points during training. Our goal is to understand how the dynamics induced by this kernel transform initially unstructured features into linearly separable representations. Rather than working in a lazy training regime, we analyze the fully nonlinear training dynamics in which the kernel itself evolves over time. 

The key idea is to decompose the argument into two conceptual steps. First, we show that, under mild regularity assumptions, the short-time evolution of features admits an exact representation in terms of a kernel operator that can be interpreted as a graph Laplacian acting on the data. Second, we identify a quantitative \emph{spectral dominance condition} on this kernel that guarantees the emergence of linear separability within finite time. The only substantive requirement to derive these kernel dynamics is a spectral discrepancy condition, which characterizes when the evolving kernel exhibits sufficient contrast between within-cluster and cross-cluster interactions. We demonstrate empirically that such a spectral structure naturally develops during training. Explaining \emph{why} nonlinear kernel dynamics tend to produce this spectral organization is an important question in its own right and lies beyond the scope of this work; our focus here is instead to isolate and characterize the mechanism that triggers linear separability once this spectral condition is in place.

The main contribution of our paper can be summarized as follows. 
\begin{center}
\fbox{
\begin{minipage}{0.95\textwidth}
\vspace{1mm}
{\bf Contribution:} We demonstrate that the success of contrastive learning cannot be explained by the loss function alone, which contains infinitely many undesirable local minimizers. Instead, through a rigorous analysis of the continuous-time gradient flow, we show that the parameterized optimization dynamics inherently drive the learned representation. Specifically, heavily contingent upon our structural assumption regarding the augmentation graph, we identify a precise spectral triggering condition, quantified by coercivity, that guarantees the onset of linear separability, ensuring that cluster structures in $\mu$ persist and sharpen in the learned features.
\vspace{1mm}
\end{minipage}
}
\end{center}

Our framework explains the success of invariance-based feature learning techniques like SimCLR (with extensions to other methods detailed in the appendix) and complements existing research on dimensional collapse \cite{jing2021understanding, zhang2022does, shen2022connect, li2022understanding} by characterizing the learning process even when collapse is avoided. Aligning with recent macroscopic modeling \cite{meng2024training}, we formulate our discrete finite-particle dynamics as a continuous Wasserstein gradient flow, proving that this geometric separation mechanism fundamentally persists as the dataset size approaches infinity. Furthermore, our approach builds upon foundational works \cite{haochen2021provable, balestriero2022contrastive} that provide guarantees for downstream tasks by linking intrinsic semantic clusters to an ``augmentation graph.'' While these prior works connect the optimal representations of contrastive losses to the \emph{static} spectral properties of this graph at the global minimum, we shift the focus entirely to continuous-time training dynamics. 
Modeling these dynamics naturally connects our framework to the rich literature on the Neural Tangent Kernel (NTK) \cite{jacot2018neural, arora2019exact, lee2019wide}. However, classical NTK analyses typically invoke infinite-width or lazy training assumptions, which constrain the kernel to remain fixed throughout the optimization process. In contrast, we analyze the \emph{time-dependent} NTK without relying on these restrictive regimes, explicitly studying its evolution over training time. Specifically, we introduce a \emph{dynamic} spectral property: an alignment threshold within this evolving kernel. We demonstrate that the structural properties of the augmentation graph interact directly with the time-dependent gradient flow to eventually cross this critical threshold, actively steering features away from sub-optimal local minima and triggering the onset of clean linear separation.

\normalcolor

\noindent{\bf Outline:} In \cref{sec:cont} we provide an overview of contrastive learning and introduce our data corruption model. In \cref{sec:opt} we derive and study the optimality conditions for the SimCLR loss and characterize its stationary points. In \cref{sec:opt-nn} we analyze the neural training dynamics of SimCLR and present our main results: identifying the strict coercivity condition under which the learned representations become linearly separable, and establishing that this separation phenomenon holds in both the discrete regime and the infinite-data continuum limit. Finally, in \cref{sec:num} we present extensive numerical experiments across diverse data modalities to empirically validate our theory, demonstrating that the rise in coercivity consistently precedes the onset of linear separation.

\section{Contrastive learning}\label{sec:cont}

We describe here our model for corrupted data in the setting of contrastive learning, and a reformulation of the SimCLR loss that is useful for our analysis.  Let $\mu \in \mathbb{P}(\mathbb{R}^D)$ be a data distribution in $\mathbb{R}^D$. Let $\mathcal{T}$ be a set of transformation functions $T: \mathbb{R}^D \rightarrow \mathbb{R}^D$ that is measurable such that, for a given $x \in \mathbb{R}^D$, $T(x) \in \mathbb{R}^D$ represents a perturbation of $x$, such as a data augmentation (e.g., cropping and image, etc.). Let $\tilde\mu \in \mathbb{P}(\mathbb{R}^D)$ denote the distribution obtained by perturbing $\mu$ with the perturbations defined in $\mathcal{T}$. That is, we choose a probability distribution $\nu\in \mathbb{P}(\mathcal{T})$ over the perturbations, and samples from $\tilde \mu$ are generated by sampling $x\sim \mu$ and $f\sim \nu$, and taking the composition $f(x)$.

We treat $\mu$ as the original clean data, which is not observable, while the perturbed distribution $\tilde \mu$ is how the data is presented. Our goal is to understand whether contrastive learning can recover information about the original data $\mu$, provided the distribution of augmentations $\nu$ is known.

Ostensibly, the objective of contrastive learning is to identify an embedding function $ f: \mathbb{R}^D \rightarrow \mathbb{R}^d $ that is invariant to the set of transformations $\mathcal{T}$. Provided such an invariant map is identified, $f$ pushes forwards both $\mu$ and $\tilde \mu$ to the same latent distributions, that is 
\[
    f_\#\tilde{\mu} = f_\#\mu.
\]
However, it is far from clear how $\mu$ and $f_\#\mu$ are related, and whether any interesting structures in $\mu$ (such as clusterability) are also present in $f_\#\mu$. 
 
For instance, if $\tilde{\mu}$ represents image data, contrastive learning aims to discover a feature distribution $f_\#\tilde{\mu}$ that remains invariant to transformations such as random translation, rotation, cropping, Gaussian blurring, and others. 
As a result, this feature distribution effectively captures the essential characteristics of the data without being influenced by these perturbations. These feature distributions are often leveraged in downstream tasks such as classification, clustering, object detection, and retrieval, where they achieve state-of-the-art performance \cite{le2020contrastive}.

To achieve this, a cost function is designed to bring similar points closer and push dissimilar points apart through the embedding map, using attraction and repulsion forces. A popular example is the Normalized Temperature-Scaled Cross-Entropy Loss (NT-Xent loss) introduced by \cite{chen2020simple}, which leads to the optimization problem
\begin{equation}\label{eq:cost-discrete}
\min_{f:\mathbb{R}^D\rightarrow\mathbb{R}^d} \mathop{\mathbb{E}}_{x\sim \mu, T,T' \sim \nu}  \log \left( 1 + \frac{\sum_{h\in\{T,T'\}} \mathbb{E}_{y\sim\mu} \left[\mathds{1}_{x \neq y} \exp\left(\frac{\simcost_f(T(x),h(y))}{\tau}\right) \right]}{\exp\left(\frac{{\simcost}_f(T(x),T'(x))}{\tau}\right)} \right),
\end{equation}
where $\nu \in \mathbb{P}(\mathcal{T})$ is a probability distribution on $\mathcal{T}$, which is assumed to be a measurable space,  $\tau$ is a given parameter, and $\simcost_f: \mathbb{R}^D \times \mathbb{R}^D \rightarrow \mathbb{R}$ is a function measuring the similarity between two embedded points with $f$ in $\mathbb{R}^d$ defined as $\simcost_f(x,y) = \frac{f(x) \cdot f(y)}{\|f(x)\| \|f(y)\|}.$
The denominator inside the log function acts as an attraction force between perturbed points from the same sample $x$, while minimizing the numerator acts as a repulsion force between points from different samples $x$ and $y$. Thus, the minimizer $f$ of the cost is expected to exhibit invariance under the set of transformations  $\mathcal{T}$:
{\begin{equation}\label{eq:inv}
    f(T(x)) = f(x), \quad 
    \quad \forall \ x \in \mathbb{R}^D, \ \forall \ T \in \mathcal{T}.
\end{equation}}
The repulsion force prevents dimensional collapse, where the map sends every sample to a constant: $f(x)=c$ for all $x\in \mathbb R^D$.

An important observation is that the NT-Xent loss becomes independent of the data distribution once the feature map $f$ is invariant, meaning that the latent distribution corresponding to an invariant minimizer may bear little resemblance to the input data. Interestingly, a similar effect is observed in other unsupervised learning models such as VICReg \cite{bardes2021vicreg} and BYOL \cite{grill2020bootstrap} (see the appendix for further discussion). Despite the fact that the NT-Xent loss has minimizers completely independent of the original data distribution $\mu$, in practice the clustering structure of $\mu$ often emerges in the latent space which results in the state-of-the-art performance in machine learning tasks.

To better understand this behavior, we analyze the NT-Xent loss by examining its minimizer and the dynamics of gradient descent. In order to overcome challenges caused by the nondifferentiability of the angular similarity $\simcost_f$ and the nonuniqueness of solutions (e.g., any $kf$ is also a minimizer for $k > 0$), we reformulate the loss to simplify the analysis. This leads to a generalized formulation of the NT-Xent loss in \cref{eq:cost-discrete}.


\begin{definition}\label{def:new-form}
    The cost function we consider for contrastive learning is
    \begin{equation}\label{eq:cost-dis-new}
        \inf_{f \in \mathcal{C}}  \left\{ L(f) := \bbE_{\substack{x\sim \mu, T,T' \sim \nu}} \Psi \left(
            \frac{\bbE_{y\sim\mu} \eta_f(T(x),T'(y))}{\eta_f(T(x),T'(x))}  \right) \right\},
    \end{equation}
    where $\Psi:\mathbb{R}\rightarrow\mathbb{R}$ is a nondecreasing function, $\mathcal{C}$ is a constraint set, and $\eta_f$ is defined as
    \begin{equation}\label{eq:eta}
        \begin{aligned}
            \eta_f(x,y) = \eta(\|f(x) - f(y)\|^2/2),
        \end{aligned}
    \end{equation}
    where $\eta:\mathbb{R}_{\geq0}\rightarrow \mathbb{R}$ is a differentaible similarity function that is maximized at $0$.
\end{definition}
The formulation in \cref{eq:cost-dis-new} generalizes the original formulation in \cref{eq:cost-discrete} by removing the indicator function $\mathds{1}_{x \neq y}$, as the effect of this function becomes negligible when a large $n$ is considered. Furthermore, the generalized formulation introduces a differentaible similarity function. This simplifies the analysis of the minimizer in the variational formulation. The generalized formulation can easily be related to the original cost function in \cref{eq:cost-discrete} by setting $\Psi(t)=\log(1+t)$, $\eta(t) = e^{-t/\tau}$ and defining $\mathcal{C} = \{ f:\mathbb{R}^D\rightarrow \mathbb{S}^{d-1}\}$. Then, the similarity function $\eta_f$ retains the same interpretation as angular similarity $\simcost_f$. This is because, if $f$ lies on the unit sphere in $\mathbb{R}^d$, and so
\[
\exp\left(- \frac{1}{2\tau} \|f(x) - f(y)\|^2\right)
= C\exp\left(\frac{1}{\tau}\frac{f(x) \cdot f(y)}{\|f(x)\|\|f(y)\|}\right),
\]
where $C = \exp(-1/\tau)$. The consideration of the constraint also resolves the issue in \cref{eq:cost-discrete}, where $kf$, for any $k \in \mathbb{R}$, could be a minimizer of \cref{eq:cost-discrete}. Thus, in the end, the introduced formulation in \cref{eq:cost-dis-new} remains fundamentally consistent with the original NT-Xent cost structure.

\section{Optimality condition}\label{sec:opt}

In this section, we aim to find the optimality condition for \cref{eq:cost-dis-new} and analyze properties of the minimizers. Using the first optimality condition we can characterize the minimizer of the NT-Xent loss in \cref{eq:cost-dis-new}. Theorem \ref{thm:xent-min} below describes the possible local minimizers of \cref{eq:cost-dis-new}, considering the constraint set defined as
\[
\mathcal{C}=\{f: \mathbb{R}^D \rightarrow \mathbb{S}^{d-1}\}.
\]
This spherical constraint is motivated by practical considerations. In most contrastive learning algorithms, it is standard to normalize feature representations to unit norm before computing similarities, so that embeddings lie on the unit sphere. This normalization stabilizes training, prevents feature norm collapse, and ensures that the similarity measure depends only on angular relationships between samples. Accordingly, analyzing the optimization problem under the spherical constraint reflects the common practice used in modern contrastive learning frameworks.

\begin{theorem}[Abundance and Stability of Stationary Embeddings]
\label{thm:xent-min}

Let $\mu \in \mathbb{P}(\mathbb{R}^D)$ be an arbitrary data distribution (which may be a continuous measure or a general finite empirical dataset), and consider the NT-Xent loss over the constraint set
$\mathcal C = \{ f:\mathbb R^D \to \mathbb S^{d-1} \}.$ Suppose $f$ is an invariant embedding such that the pushforward measure $f_\#\mu$ is a symmetric discrete measure satisfying
\begin{equation}\label{eq:sym-cond}
    \int_{\mathbb S^{d-1}} h(x_1,y)\,df_\#\mu(y) = \int_{\mathbb S^{d-1}} h(x_2,y)\,df_\#\mu(y)
\end{equation}
for all $x_1,x_2$ in the support of $f_\#\mu$ and all anti-symmetric functions $h(x,y) = -h(y,x)$. Then $f$ is a stationary point of the NT-Xent loss.

Assume in addition that $f$ collapses the underlying data distribution $\mu$ into exactly $K$ discrete cluster centers $\xi_1, \dots, \xi_K \in \mathbb{S}^{d-1}$, meaning the pushforward measure is a sum of Dirac masses
\[
f_\#\mu = \sum_{k=1}^K p_k \delta_{\xi_k},
\]
where $p_k > 0$ represents the mass of the data mapped to center $\xi_k$, and the measure $f_\#\mu$ satisfies the symmetry condition in \cref{eq:sym-cond}. 

Then there exists a constant $c>0$ depending only on $\tau$, $K$, and the cluster geometry such that for every admissible perturbation $h:\mathbb R^D \to \mathbb R^d$ satisfying the tangency condition
\[
\langle f(x),h(x)\rangle = 0 \quad \text{for all } x,
\]
the second variation follows
\[
\delta^2 L(f)(h,h) \geq c \sum_{k \neq l} e^{-\|\xi_k-\xi_l\|^2/(2\tau)} \|h(\xi_k)-h(\xi_l)\|^2 p_k p_l > 0
\]
whenever $h$ is orthogonal to infinitesimal rotations of the sphere, namely
\[
\int \langle h(x),Af(x)\rangle d\mu(x)=0 \quad \text{for all skew-symmetric } A.
\]
Consequently, $f$ is a strict local minimizer of the NT-Xent loss modulo global rotations of the sphere.
\end{theorem}

\begin{proof}
We first establish the stationary condition. The constrained problem
\[
\min_{f:\mathbb R^D\to \mathbb S^{d-1}} L(f)
\]
can be treated via Lagrange multipliers. Introducing a scalar multiplier $\lambda(x)$ for the pointwise constraint $\|f(x)\|=1$, the Euler–Lagrange equation takes the form
\[
\frac{\delta L}{\delta f}(x) - \lambda(x) f(x) = 0.
\]
Using the first variation formula for the NT-Xent functional, one obtains
\[
\frac{\delta L}{\delta f}(x) = \int \Psi'(G(f,x))\,\eta_f'(x,y)\,(f(x)-f(y))\,d\mu(y).
\]
Since $f$ is invariant and takes values on the sphere, the factor $\Psi'(G(f,x))$ is constant for all $x$ in the support of $f_\#\mu$. Writing the stationarity condition in terms of the pushforward measure gives
\[
\lambda(x) = C \int_{\mathbb S^{d-1}} r'(|x-y|^2/2)\,(x-y)\,df_\#\mu(y),
\]
where $C$ is constant. The function $h(x,y)=r'(|x-y|^2/2)(x-y)$ is anti-symmetric. By the assumed symmetry of $f_\#\mu$, the integral above is independent of $x$ for any point in the support. Therefore $\lambda(x)$ is constant and the Euler–Lagrange equation is satisfied, proving that $f$ is stationary.

We now analyze the second variation. Consider perturbations $f_t(x)=\frac{f(x)+t h(x)}{\|f(x)+t h(x)\|}$. The tangency condition $\langle f,h\rangle=0$ ensures that this curve remains on the sphere to first order.

Using the standard second variation formula for pairwise interaction energies, one obtains
\[
\delta^2 L(f)(h,h) = \iint W(x,y)\Big( \langle f(x)-f(y),h(x)-h(y)\rangle^2 - \tau \|h(x)-h(y)\|^2 \Big)\,d\mu(x)d\mu(y),
\]
where
\[
W(x,y) = \frac{e^{-\|f(x)-f(y)\|^2/(2\tau)}}{\tau^2(1+G(x)^2/2)}.
\]

Because $f_\#\mu$ collapses the data into $K$ distinct Dirac masses $\{\xi_k\}_{k=1}^K$, the double integral over the underlying data domain (regardless of whether $\mu$ is a continuous distribution or a discrete finite sum) reduces strictly to a finite double sum over these cluster centers. Let $h_k = h(\xi_k)$ represent the perturbation at center $k$. The second variation becomes
\[
\delta^2 L(f)(h,h) = \sum_{k=1}^K \sum_{l=1}^K W_{kl} \Big( \langle \xi_k-\xi_l,h_k-h_l\rangle^2 - \tau \|h_k-h_l\|^2 \Big) p_k p_l.
\]

Since points mapped to the same center coincide ($\xi_k = \xi_l$ when $k=l$), the intra-cluster terms vanish, and only inter-cluster terms ($k \neq l$) contribute. For distinct clusters, the vectors $\xi_k-\xi_l$ are bounded away from zero. Using the tangency condition and orthogonality to rotational modes, one can decompose
\[
h_k-h_l = \alpha_{kl}(\xi_k-\xi_l) + v_{kl},
\]
where $v_{kl}$ is orthogonal to $\xi_k-\xi_l$. Substituting this decomposition yields
\[
\langle \xi_k-\xi_l,h_k-h_l\rangle^2 = \|\xi_k-\xi_l\|^2 \alpha_{kl}^2,
\]
while
\[
\|h_k-h_l\|^2 = \|\xi_k-\xi_l\|^2 \alpha_{kl}^2 + \|v_{kl}\|^2.
\]
Because $\|\xi_k-\xi_l\|$ is uniformly bounded below for distinct clusters, the quadratic form above becomes coercive
\[
\delta^2 L(f)(h,h) \geq c \sum_{k \neq l} W_{kl}\|h_k-h_l\|^2 p_k p_l
\]
for some $c>0$. 

Finally, if $h(x)=Af(x)$ for a skew-symmetric matrix $A$, then $h_k-h_l=A(\xi_k-\xi_l)$ and all pairwise differences lie in a rotation direction, yielding $\delta^2 L(f)(h,h)=0$. Excluding these global rotational modes provides strict positivity. This establishes local minimality modulo rotations.
\end{proof}

The alignment--uniformity framework of \cite{wang2020understanding} provides a foundational perspective on contrastive learning. Their analysis shows that, in the asymptotic limit of infinitely many negative samples, the continuous uniform distribution emerges as the global minimizer of the repulsive term. The following corollary connects our variational framework with this uniformity hypothesis, demonstrating that the uniform distribution naturally satisfies the symmetry condition established in \Cref{thm:xent-min}. By analyzing the exact NT-Xent loss directly, we relax the requirement for an asymptotic limit and broaden the set of stationary solutions to include not only continuous distributions but also highly symmetric discrete configurations.

\begin{corollary}[Stationarity and Stability of the Uniform Measure]
\label{cor:uniform_stationary}
Let $\sigma_{d-1}$ denote the continuous uniform probability measure on the sphere $\mathbb{S}^{d-1}$. If the embedding map $f$ perfectly disperses the data such that its pushforward measure is uniform ($f_\#\mu = \sigma_{d-1}$), then $f$ is not only a stationary point of the NT-Xent loss, but also a stable local minimizer modulo global rotations.
\end{corollary}

\begin{proof}
For $f_\#\mu = \sigma_{d-1}$, the integral vector field driving the gradients, $v(x) = \int_{\mathbb{S}^{d-1}} r'(|x-y|^2/2)(x-y) \,d\sigma_{d-1}(y)$, must be equivariant under all global rotations $R \in \mathrm{SO}(d)$ due to the perfect isotropy of the uniform measure. The only vector field satisfying this universal symmetry is purely radial, $v(x) = cx$ for some scalar $c \in \mathbb{R}$. Because $x \in \mathbb{S}^{d-1}$, the radial projection $\langle v(x), x \rangle = c$ evaluates to an identical constant for all $x$, natively satisfying the Euler-Lagrange stationarity condition of \Cref{thm:xent-min}.

The NT-Xent repulsive denominator is governed by the interaction energy $E(\mu) = \iint e^{\langle x, y \rangle / \tau} \,d\mu(x) d\mu(y)$. From \cite{cohn2007universally, wang2020understanding}, it is established that completely monotone potentials, such as the exponential function, are uniquely globally minimized by the continuous uniform measure. Because $\sigma_{d-1}$ achieves the strict global minimum of this energy landscape, any admissible perturbation $h(x)$ that is not a trivial rigid rotation must strictly increase the functional. Consequently, the second variation $\delta^2 L(f)(h,h)$ is strictly positive for all non-rotational perturbations, confirming $\sigma_{d-1}$ is a locally stable attractor.
\end{proof}

While the asymptotic analysis predicts continuous uniformity as the ideal solution, \Cref{thm:xent-min} reveals that the actual finite-temperature NT-Xent landscape is vastly more complex. Perfect continuous uniformity is merely one stationary element within a much broader class of symmetric configurations. 

More importantly, the theorem establishes the strict local stability of \textit{discrete} symmetric configurations. As illustrated in \Cref{fig:dirac_collapse_comparison}, these stationary embeddings include not only desirable configurations that preserve true class topologies (\Cref{fig:dirac_collapse_comparison}A) but also completely degenerate, class-scrambling states (\Cref{fig:dirac_collapse_comparison}B). In these "bad" local minimizers, arbitrary data points collapse into meaninglessly symmetric Dirac masses on the sphere (e.g., $f_\#\mu = \frac{1}{K}\sum_{k=1}^K \delta_{\xi_k}$). Because these degenerate configurations perfectly satisfy the discrete symmetric integral condition, the repulsive gradients of the contrastive denominator perfectly balance out, trapping the network in a locally stable configuration modulo rotational symmetries. 

\begin{figure}[htbp]
    \centering
    \begin{minipage}{0.48\textwidth}
        \centering
        \begin{tikzpicture}[>=stealth, scale=0.8]
            \draw[thick, dashed, fill=gray!5, rounded corners] (-4, -2) rectangle (-1, 2);
            \node at (-2.5, -2.5) {Data Space ($\mathbb{R}^D$)};
            
            \fill[blue!40] (-3, 1) ellipse (0.4 and 0.3);
            \fill[red!40] (-1.8, 0) ellipse (0.5 and 0.4);
            \fill[green!40] (-2.8, -1) ellipse (0.3 and 0.5);
            \node at (-3, 1) {$\mu_1$};
            \node at (-1.8, 0) {$\mu_2$};
            \node at (-2.8, -1) {$\mu_3$};

            \draw[thick] (3, 0) circle (1.5);
            \node at (3, -2.5) {Latent Sphere ($\mathbb{S}^{d-1}$)};
            
            \coordinate (xi1) at (3, 1.5); 
            \coordinate (xi2) at (4.3, -0.75); 
            \coordinate (xi3) at (1.7, -0.75); 

            \fill[blue!80!black] (xi1) circle (0.1) node[above] {$\xi_1$};
            \fill[red!80!black] (xi2) circle (0.1) node[right] {$\xi_2$};
            \fill[green!80!black] (xi3) circle (0.1) node[left] {$\xi_3$};

            \draw[->, thick, blue, bend left=15] (-2.5, 1) to (2.8, 1.4);
            \draw[->, thick, red] (-1.3, 0) to (4.1, -0.6);
            \draw[->, thick, green, bend right=15] (-2.5, -1) to (1.9, -0.6);

            \draw[dashed, gray] (xi1) -- (xi2) -- (xi3) -- cycle;
            \node[gray, text width=2cm, align=center] at (3, 0.4) {\scriptsize symmetric discrete measure};
        \end{tikzpicture}
        \caption*{(A) Class-Preserving Feature Collapse}
    \end{minipage}
    \hfill
    \begin{minipage}{0.48\textwidth}
        \centering
        \begin{tikzpicture}[>=stealth, scale=0.8]
            \draw[thick, dashed, fill=gray!5, rounded corners] (-4, -2) rectangle (-1, 2);
            \node at (-2.5, -2.5) {Data Space ($\mathbb{R}^D$)};
            
            \fill[blue!40] (-3, 1) ellipse (0.4 and 0.3);
            \fill[red!40] (-1.8, 0) ellipse (0.5 and 0.4);
            \fill[green!40] (-2.8, -1) ellipse (0.3 and 0.5);
            \node at (-3, 1) {$\mu_1$};
            \node at (-1.8, 0) {$\mu_2$};
            \node at (-2.8, -1) {$\mu_3$};

            \draw[thick] (3, 0) circle (1.5);
            \node at (3, -2.5) {Latent Sphere ($\mathbb{S}^{d-1}$)};
            
            \coordinate (xi1) at (3, 1.5); 
            \coordinate (xi2) at (4.3, -0.75); 
            \coordinate (xi3) at (1.7, -0.75); 

            \fill[gray!70!black] (xi1) circle (0.1) node[above] {$\xi_1$};
            \fill[gray!70!black] (xi2) circle (0.1) node[right] {$\xi_2$};
            \fill[gray!70!black] (xi3) circle (0.1) node[left] {$\xi_3$};

            \draw[->, thick, blue, bend left=15] (-2.5, 1) to (2.8, 1.4); 
            \draw[->, thick, blue] (-2.5, 1) to (4.1, -0.6); 
            
            \draw[->, thick, red] (-1.3, 0) to (2.8, 1.4); 
            \draw[->, thick, red, bend right=15] (-1.3, 0) to (1.9, -0.6); 
            
            \draw[->, thick, green, bend right=15] (-2.5, -1) to (4.1, -0.6); 
            \draw[->, thick, green] (-2.5, -1) to (1.9, -0.6); 


            \draw[dashed, gray] (xi1) -- (xi2) -- (xi3) -- cycle;
            \node[gray, text width=2cm, align=center] at (3, 0.4) {\scriptsize symmetric discrete measure};
        \end{tikzpicture}
        \caption*{(B) Class-Scrambling Collapse (Bad Clustering)}
    \end{minipage}
    
    \caption{Visualizing good and bad local minimizers of the NT-Xent loss from \Cref{thm:xent-min}. \textbf{A:} Feature map $f$ correctly preserves the true data topology, mapping distinct classes to symmetric cluster centers such that $f_\#\mu$ satisfies \cref{eq:sym-cond}. This is a successful, desirable result for clustering. \textbf{B:} An alternative feature map, also a stable local minimizer due to the perfectly symmetric target discrete measure, but 'bad' for clustering. It collapses the true data topology, meaninglessly reassigning data points from different classes into the generic target masses.}
    \label{fig:dirac_collapse_comparison}
\end{figure}

The numerical experiments in \Cref{fig:exp-cluster} further demonstrate that the finite-temperature NT-Xent loss exhibits pronounced plateau behavior for such clustered configurations. Increasing the number of clusters or separating them further improves the loss only up to a temperature-dependent threshold, beyond which many distinct discrete configurations achieve nearly identical objective values. Crucially, this plateau implies that as the number of clusters grows toward infinity, the asymptotic regime approximating the continuous uniform measure, the loss value remains virtually indistinguishable from that of finite, discrete configurations. This reveals that, rather than possessing a uniquely isolated global minimum, the loss landscape is characterized by a vast, complex valley of minima where perfectly uniform representations and highly degenerate discrete collapses are evaluated almost equally.

These observations highlight a critical limitation in explaining contrastive learning solely through the lens of the objective function. If the finite-temperature NT-Xent landscape is uniformly flat across an abundance of locally stable, degenerate configurations that completely ignore the underlying data distribution $\mu$, why does practical contrastive training consistently find meaningful representations?

In practice, training begins from a random parameterization in a neural network, not from an arbitrary stationary function. Therefore, the trajectory is governed by the training dynamics induced by the network's architecture and gradient flow, rather than simply descending into the nearest stationary trap of the functional objective. This dictates that the neural network parameterization and its optimization dynamics act as the crucial \emph{selection mechanism}. Understanding this mechanism, and identifying the exact dynamical conditions under which training avoids meaningless dimensional collapse and preferentially triggers the separation of true latent classes, is the focus of the subsequent sections.

\begin{figure}[!t]
    \centering
    \includegraphics[width=1.0\linewidth]{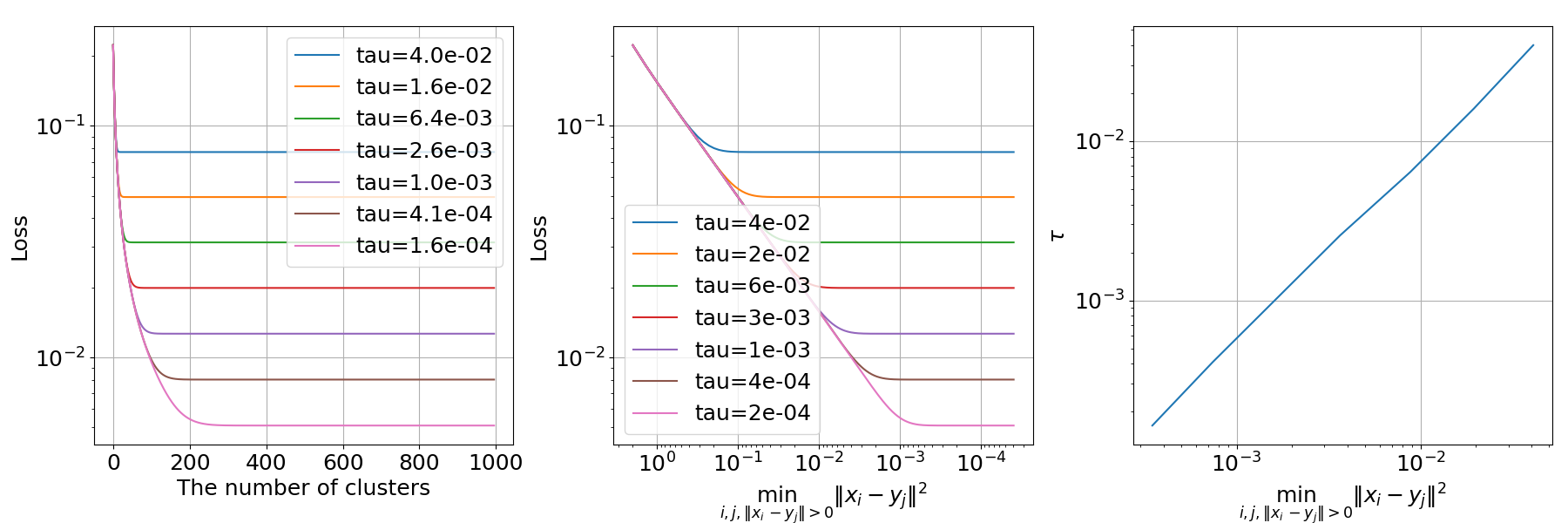}
        \caption{
 The figure shows the NT-Xent loss for different embedded distributions $f_\#\mu = \frac{1}{K}\sum^{K}_{i=1} \delta_{x_i}$ with $x_i$ on $\mathbb{S}^1$. The first plot shows the loss decreasing with the number of clusters, then plateauing. The second shows the loss decreasing with the minimum squared distance between cluster points, stopping at a threshold. Both suggest that increasing clusters and decreasing $\tau$ reduce the loss. The third plot shows a linear relationship between $\tau$ and the minimum distance.
            }
    \label{fig:exp-cluster}
\end{figure}

\section{Optimization of Neural Networks}\label{sec:opt-nn}
In this section, we begin our study of contrastive learning by focusing on the optimization of neural networks. Our goal is to understand how the training dynamics of a neural network shape the learned feature representation and how the structure of the data distribution enters the latent space through the neural kernel.

We first introduce a gradient flow formulation of neural network training for contrastive objectives. This viewpoint allows us to describe the evolution of features continuously in time and to analyze which properties of the kernel and the data are essential for producing meaningful representations. Throughout this section, we use the notation $\lrbr{n} = \{1,\dots,n\}$.

The main purpose of this section is therefore to clarify how contrastive learning can be expressed in terms of neural network gradient dynamics and to identify the key conditions under which informative and structured features emerge during training. For better readability of the presentation, most of the technical proofs are deferred to the appendix.

\subsection{Gradient flow from neural network parameters}

Let $w \in \mathbb{R}^m$ be a vector of neural network parameters, $\{x_1, \ldots, x_n\} \subset \mathbb{R}^D$ be data samples, and $f(w, x_i) = (f^1(w, x_i), \ldots, f^d(w, x_i))^{\top} \in \mathbb{R}^d$ be an embedding function where each function $f^k: \mathbb{R}^{n+D} \to \mathbb{R}$ is a scalar function for $k = 1, \ldots, d$. Consider a loss function $L = L(y,x): \mathbb{R}^{d+D} \to \mathbb{R}$ with respect to $w$
\begin{equation}\label{eq:loss-w}
    \mathcal{L}(w) = \frac{1}{n} \sum_{i=1}^n L(f(w, x_i),x_i).
\end{equation}
Let $w(t)$ be a vector of neural network parameters as a function of time $t$. The gradient descent flow can be expressed as 
\[
\dot{w}(t) = -\nabla \mathcal{L}(w).
\]
Because $\mathcal{L}$ is highly non-convex, analyzing its gradient flow in weight space is notoriously difficult. By shifting the focus from weight dynamics to the temporal evolution of the neural network's outputs on the training data, we can derive an alternative gradient flow with more favorable properties for analysis. The following proposition characterizes this gradient flow as derived from the loss function $\mathcal{L}$. The proof of this proposition is provided in the appendix and follows the methodology established in \cite{calder2025introduction}.

\begin{proposition}\label{prop:w-gf}
    Let $w(t)$ be a vector of neural network parameters as a function of time $t$. Consider a set of data samples $\{x_1, \ldots, x_n\}$. Define $z_i(t) = f(w(t), x_i)$ for each $i \in \lrbr{n}$.
    Then, $z_i(t)$ satisfies the following ordinary differential equation (ODE):
    \begin{equation}\label{eq:gf-z}
        \begin{aligned}
            \dot{z}_i(t) &= - \frac{1}{n}\sum_{j=1}^{n} K_{ij}(t) \nabla_z L(z_j(t),x_j),
        \end{aligned}
    \end{equation}
    where the kernel matrix $K_{ij} \in \mathbb{R}^{d\times d}$ is given by
    \begin{align}\label{eq:kernel-K}
        (K_{ij}(t))^{kl} 
        = K^{kl}_{ij}(t)
        = (\nabla_w f^{k}(w(t), x_i) )^\top(\nabla_w f^{l}(w(t), x_j)).
    \end{align}
\end{proposition}
\begin{remark}\label{rem:ntk}
The viewpoint in \cref{prop:w-gf}, which lifts the training dynamics from neural network parameters to a function space formulation, is closely related to the perspective underlying the Neural Tangent Kernel (NTK) framework \cite{jacot2018neural}. Analyzing continuous-time gradient flow in function space has become a central tool for understanding the optimization and evolution of parameterized models. In the classical NTK regime \cite{lee2019wide}, one considers an infinite-width limit in which the neural tangent kernel remains effectively constant during training. As a result, the dynamics reduce to gradient descent in a reproducing kernel Hilbert space with a fixed kernel.

In contrast, throughout our analysis we do not assume an infinite-width network, thereby remaining closer to the finite-dimensional regimes encountered in practical neural network training. In this setting, the kernel is evaluated directly on the training data and evolves during training, leading to kernel matrices that are both data dependent and time varying. The necessity of modeling these dynamic, time-varying function space evolutions has been increasingly recognized in rigorous mathematical treatments of deep networks and transformers \cite{huang2020dynamics, geshkovski2025mathematical}. This feature plays a central role in the analysis that follows.

It is also important to emphasize that \cref{prop:w-gf} is quite general and does not rely on the specific structure of neural networks. The result holds for any parameterized function class, and up to this point we have not used any properties particular to neural network architectures. This universality of function space dynamics aligns with broader findings that generalized gradient descent on finite parameterized models naturally integrates a time-varying tangent kernel \cite{domingos2020every}.
\end{remark}
\begin{remark}
    The training dynamics in the absence of a neural network can be expressed as
    \begin{equation}\label{eq:z-no-nn}
        \dot{z}_i(t) = -  \nabla_z L(z_i(t),x_i)
    \end{equation}
    where $K_{ij}$ is set to be identity matrices. In contrast to \cref{eq:gf-z}, the above expression shows that the training dynamics on the $i$-th point $z_i$ are influenced solely by the gradient of the loss function  at $x_i$, and there is no mixing of the data via the neural kernel $K$ (since here it is the identity matrix). 
\end{remark}

Using \Cref{prop:w-gf}, we can analyze the invariance-preserving properties (and possible failures) of gradient descent, both without and with the neural network kernel. The following proposition compares the vanilla gradient flow with the gradient flow induced by the neural network kernel matrix. The proof is provided in the appendix.
\begin{proposition}\label{prop:nn-toy-map}
    Consider the gradient descent iteration from a gradient flow without a neural network in \cref{eq:z-no-nn}, where $z_i^{(k)} = f(w^{(k)}, x_i)$ for all $i \in \lrbr{n}$, and
    \begin{equation}\label{eq:z-no-nn-gd}
        z_i^{(k+1)} = z_i^{(k)} - \sigma \nabla_z L(z_i^{(k)}, x_i),
    \end{equation}
    with $\sigma$ as the step size. If $f(w^{(0)}, \cdot)$ is invariant to perturbations from $\nu$, as defined in \cref{eq:inv},
    then $f(w^{(k)}, \cdot)$ remains invariant for all gradient descent iterations.
    
    On the other hand, in the case of a gradient descent iteration from \cref{eq:gf-z},
    \begin{equation}\label{eq:z-nn-gd}
        z_i^{(k+1)} = z_i^{(k)} - \frac{\sigma}{n} \sum_{j=1}^n K_{ij}^{(k)} \nabla_z L(z_j^{(k)}, x_j),
    \end{equation}
    the invariance of $f$ at the $(k+1)$-th iteration holds only if $f$ is invariant at the $k$-th iteration and additionally satisfies the condition $\nabla_w f(w^{(k)}, T(x)) = \nabla_w f(w^{(k)}, x)$ for all $x \in \mathbb{R}^D$ and $T \in \mathcal{T}$.
\end{proposition}

\Cref{prop:nn-toy-map} highlights a fundamental distinction between optimization dynamics with and without neural network parameterization. In the standard gradient descent update \cref{eq:z-no-nn-gd}, the evolution takes place directly in feature space. In this case, the updates depend only on the loss gradients at the feature points themselves, and there is no coupling induced by a parameterized map. Consequently, if the initial representation $f(w^{(0)},\cdot)$ is invariant to perturbations from $\nu$, this invariance is preserved throughout the iterations.

From another perspective, this setting corresponds to the special case where the kernel matrix is the identity. The gradient flow then acts independently on each feature point $z_i$, and the optimization behaves like a purely local descent driven solely by the loss function, without any interaction mediated by a parameterized model.

In contrast, when the representation is induced by a neural network, the update \cref{eq:z-nn-gd} involves the neural kernel matrix $K^{(k)}$. This introduces a coupling between feature points, as each update now depends on gradients evaluated at all other data points. As a result, even if the representation is initially invariant, maintaining this invariance requires an additional structural condition on the parameter gradients, namely that $\nabla_w f(w^{(k)},T(x))=\nabla_w f(w^{(k)},x)$ for all transformations $T\in\mathcal{T}$. Since this condition is generally not preserved during training, the optimization can break the initial invariance and lead to fundamentally different dynamics from those of standard gradient descent.

This illustrates a key role of neural parameterization: the kernel induced by the network governs how information propagates across data points during training. A large body of work has demonstrated that this coupling substantially alters optimization behavior. For instance, \cite{xu2019frequency,xu2019training} established the frequency principle, showing that neural network training dynamics are strongly biased toward low frequency components, in sharp contrast to the behavior of vanilla gradient descent.


\subsection{Kernel Dynamics and Finite-Time Linear Separability}\label{sec:kernel_separation}

The purpose of this section is to explain, at a mechanistic level, how meaningful feature structure and linear separability emerge from neural network training under contrastive learning. Building on the previous discussion, the central object of interest is the \emph{time-evolving neural kernel}, which governs how gradient information propagates across data points during training. Our goal is to understand how the dynamics induced by this kernel transform initially unstructured features into linearly separable representations.

Rather than working in a lazy or linearized regime, we analyze the fully nonlinear training dynamics in which the kernel itself evolves over time. The key idea is to decompose the argument into two conceptual steps. First, we show that, under mild regularity assumptions, the short-time evolution of features admits an exact representation in terms of a kernel operator that can be interpreted as a graph Laplacian acting on the data. Second, we identify a quantitative \emph{spectral dominance condition} on this kernel that guarantees the emergence of linear separability within finite time.

The assumptions required to derive the kernel dynamics are and mild. The only substantive requirement is a spectral discrepancy condition, which characterizes when the evolving kernel exhibits sufficient contrast between within-cluster and cross-cluster interactions. We demonstrate empirically that such a spectral structure naturally develops during training. Explaining why nonlinear kernel dynamics tend to produce this spectral organization is an important question in its own right and lies beyond the scope of this work. Our focus here is instead to isolate and characterize the mechanism that triggers linear separability once this spectral condition is in place.

For clarity of exposition, we present the analysis in the simplest setting of 1D features. This setting already captures the essential mechanism by which evolving kernel dynamics generate spectral dominance and induce separation. Extensions to higher-dimensional feature spaces are conceptually straightforward and require only notational generalizations. We adopt this simplified framework to make the core ideas transparent without introducing unnecessary technical complexity.


\subsection{Augmented Dataset Formulation}

Let $\{x_1,\dots,x_n\}$ be a dataset, partitioned into $N$ semantic clusters
$C_1,\ldots,C_N$.
For each $i=1,\dots,n$, we generate $m$ augmented samples
\[
x_{i,1},\dots,x_{i,m},
\]
and define the augmented dataset
\[
\widetilde X := \{x_{i,a} : i=1,\dots,n,\ a=1,\dots,m\},
\qquad N := nm.
\]
Let $f(w,\cdot)$ be a neural network with scalar output, and define the features
\[
z_{i,a}(t) := f(w(t),x_{i,a}) \in \mathbb R,
\]
where $w(t)$ evolves by gradient flow of the empirical contrastive loss
\begin{equation}\label{eq:augmented_contrastive_loss}
\widetilde{\mathcal L}(z)
=
\frac{1}{nm^2}
\sum_{i,a,b}
\log
\frac{
\frac{1}{nm}\sum_{j,b}
\exp\left(-\frac{(z_{i,a}-z_{j,b})^2}{2\tau}\right)
}{
\exp\left(-\frac{(z_{i,a}-z_{i,b})^2}{2\tau}\right)
}.
\end{equation}
While Equation \eqref{eq:augmented_contrastive_loss} is heavily inspired by the NT-Xent loss introduced in SimCLR, it features several key modifications designed for theoretical tractability. Standard SimCLR operates on high-dimensional representations projected onto a unit hypersphere and relies on cosine similarity. In contrast, our formulation considers scalar features $z \in \mathbb{R}$ and employs the negative squared Euclidean distance, a necessary substitution since cosine similarity is degenerate in one dimension. 

In practice, contrastive learning typically trains on small batches and randomly generates a few augmented views per data point at each step. In our formulation, we use a fixed set of $m$ augmented views for the entire dataset and study the deterministic gradient flow.  While this approach removes the randomness of constantly sampling new augmentations, it still captures the essential learning dynamics. Randomly generating views during training estimates the average effect of all possible augmentations over time. By using a fixed set of $m$ views, our model follows this average, expected path. Thus, the core mechanisms driving feature separation are preserved, meaning the deterministic flow accurately reflects the expected behavior of the stochastic process. 

The primary advantage of this simplification is that it provides a clear, tractable framework. By stripping away the unpredictable fluctuations of random sampling, we are able to directly isolate and prove how the structural properties of the augmentations drive the emergence of linear separability.

Despite these structural simplifications, the formulation preserves the core contrastive mechanics of SimCLR. The objective retains the fundamental InfoNCE log-sum-exp structure, thereby maintaining the essential ``push-pull'' dynamics. The denominator explicitly minimizes the distance between two distinct augmented views ($a$ and $b$) of the same underlying sample $i$ (positive pairs), while the numerator sum maximizes the distance between a view and all other samples in the dataset (negative pairs). The temperature parameter $\tau$ likewise continues to govern the scale of these similarity penalties. By isolating the pure effect of this contrastive objective, this idealized framework directly enables us to formulate exact differential equations for the feature evolution $z(t)$ and formally prove the conditions under which features successfully separate into distinct semantic clusters.

\subsection{Geometric Formulation and Augmentation Connectivity}\label{sec:geometric_formulation}

To bridge the gap between microscopic instance variance (which the contrastive loss minimizes) and macroscopic class dispersion (which we desire for clustering), we must first formalize the geometric properties of the latent space and the augmentation overlap graph. For a dataset consisting of samples $\{x_i\}_{i=1}^n$, suppose that each sample admits $m$ augmented versions $\{x_{i,a}\}_{a=1}^m$. For a feature map $f$, define the corresponding feature vectors $z_{i,a} = f(x_{i,a})$, and let $z = \{z_{1,1}, \dots, z_{n,m}\}$ denote the collection of all feature vectors. Assume the dataset is partitioned into $N$ clusters $\{C_k\}_{k=1}^N$.

We define the sample-mean projection $P_{\mathrm{mean}}$ and the class-mean projection $P_{\mathrm{class}}$ as orthogonal projection matrices. For $P_{\mathrm{mean}}$, we define $n$ mutually orthonormal vectors $u_1, \dots, u_n \in \mathbb{R}^{nm}$. The $(j, b)$-th entry of $u_i$ is $1/\sqrt{m}$ if $j=i$, and $0$ otherwise. Letting $U_{\mathrm{mean}} = [u_1, \dots, u_n]$, the orthogonal projection is given by $P_{\mathrm{mean}} = U_{\mathrm{mean}} U_{\mathrm{mean}}^\top$. Similarly, for $P_{\mathrm{class}}$, we define $N$ mutually orthonormal vectors $v_1, \dots, v_N \in \mathbb{R}^{nm}$. The $(i,a)$-th entry of $v_k$ is $1/\sqrt{m|C_k|}$ if $i \in C_k$, and $0$ otherwise. Because the classes $C_k$ do not overlap, these vectors are orthogonal. Letting $U_{\mathrm{class}} = [v_1, \dots, v_N]$, the projection is $P_{\mathrm{class}} = U_{\mathrm{class}} U_{\mathrm{class}}^\top$.

Applying these projection matrices to the feature matrix $z \in \mathbb{R}^{nm \times d}$ yields
\[
\begin{aligned}
    Tr(z^\top P_{\mathrm{mean}} z) &= \sum_{i=1}^n \frac{1}{m} \left\| \sum_{a=1}^m z_{i,a} \right\|^2, \\
    Tr(z^\top P_{\mathrm{class}} z) &= \sum_{k=1}^N \frac{1}{m|C_k|}\left\| \sum_{i\in C_k} \sum_{a=1}^m z_{i,a} \right\|^2.
\end{aligned}
\]
These projection operators measure how features align with different data averages. $P_{\mathrm{mean}}$ captures instance-level centers by averaging the augmented views of each specific sample. $P_{\mathrm{class}}$ captures cluster-level centers. The orthogonal complement $(I - P_{\mathrm{mean}})$ measures the intra-instance augmentation variance, the physical spread of augmented views diverging from their base sample.

To provide a clear geometric framework for our analysis, we define the augmentation graph Laplacian, $L_{\mathrm{aug}} \in \mathbb{R}^{nm \times nm}$. For any two samples $(i,a)$ and $(j,b)$, the connection weight is determined by the similarity between the raw data points $x_{i,a}$ and $x_{j,b}$ in the input space. By using these geometric operators and the graph Laplacian, we can clearly define what constitutes a ``good'' augmentation strategy. Specifically, we assume that the connections within the same class are strong, while the connections between different classes remain weak or non-existent.

Building upon these geometric operators and the graph laplacain from augmented points, we formalize our requirements for the dataset's augmentation strategy. This assumption replaces abstract isolation metrics with the physical structure of the data graph.

\begin{assumption}[Augmentation Overlap Graph]
\label{assum:augmentation_graph}
For each cluster $k$ and its corresponding normalized graph Laplacian $L_{\mathrm{aug}}$, we assume the augmentation strategy satisfies two properties
\begin{enumerate}
    \item \textbf{Graph Connectivity ($\lambda_2$):} The augmentations form a connected graph across the cluster, meaning its connectivity satisfies $\lambda_2(L_{\mathrm{aug}}) > 0$.
    \item \textbf{Local-Global Variance Bound ($\eta$):} The neural network constrains the global feature variation across the connected augmentation graph proportionally to the local intra-instance variance. Formally, there exists a bounding constant $\eta > 0$ such that total Dirichlet energy satisfies $z^\top L_{\mathrm{aug}} z \leq \eta \|z\|^2$.
\end{enumerate}
\end{assumption}

For contrastive learning to successfully group similar data points together, our data augmentations must meet two key conditions. First, the augmentations must be balanced to appropriately link related data without scrambling the entire dataset. Properly tuned augmentations cast a wide net, creating variations of one image that overlap with variations of a neighboring, similar image. This overlap ties related data points together into a highly connected web (yielding a large $\lambda_2$). 

However, there are two ways this connectivity can fail. On one extreme, if we use weak augmentations, each original image is left isolated on its own island, breaking the web apart ($\lambda_2 \to 0$). On the other extreme, if we use overly aggressive augmentations, the variations become completely generic. This overly strong augmentation entangles every data point with every other data point in the dataset, regardless of their actual class. The web becomes a completely tangled, indistinguishable mess, blurring the boundaries between concepts and forcing the network to collapse all features together. 

Second, the network must obey the local-global variance bound ($\eta$). In simple terms, this means the network shouldn't produce wild, unpredictable features. The total feature differences across the entire connected web (the global variation) must be kept in check by how much the features vary when we just augment a single datum (the local variation). 

With our structural assumptions on the augmentation graph formalized, we gain a crucial geometric property linking local data augmentations to the global class structure. While contrastive learning inherently pulls augmented views of the same data point together to reduce local variation, our graph assumptions propagate this local behavior to the entire class.

The following proposition demonstrates that the deviation between a point's average representation and its overall class average, given by $\|(P_{\mathrm{mean}} - P_{\mathrm{class}})z\|^2$, is bounded by this local variation scaled by $\frac{\eta}{\lambda_2}$. A graph that is both well-connected and well-clustered exhibits a larger spectral gap $\lambda_2$, which forcefully shrinks this bound. Consequently, if the network successfully minimizes local augmentation variation, these structural parameters guarantee that all points within a class stay tightly grouped around their shared center.

\begin{proposition}[Bound on Class Dispersion]\label{prop:poincare_dispersion}
Suppose Assumption \ref{assum:augmentation_graph} holds. Then
\[
\|(P_{\mathrm{mean}} - P_{\mathrm{class}}) z\|^2 \leq \frac{\eta}{\lambda_2} \|z\|^2.
\]
\end{proposition}
\begin{proof}
Define $u = z - P_{\mathrm{class}}z$. The null space of $L_{\mathrm{aug}}$ is spanned by the global constant vector $\mathbf{1}$. By definition, $P_{\mathrm{class}}z$ assigns each sample the mean of its respective class, ensuring that the sum of the elements in $u$ over any class $\mathcal{C}_k$ is exactly zero. Thus, the global sum of $u$ is zero, meaning $u \perp \mathbf{1}$. Since $u$ is orthogonal to the null space of $L_{\mathrm{aug}}$, the Poincaré inequality yields
\begin{equation}\label{pru1}
    \|u\|^2 \leq \frac{1}{\lambda_2} u^\top L_{\mathrm{aug}} u
\end{equation}
By \Cref{assum:augmentation_graph} and using the fact that $\|u\| = \|(I - P_{\mathrm{class}})z\| \leq \|z\|$, we have
\begin{equation}\label{prmean2}
    u^\top L_{\mathrm{aug}} u \leq \eta \|u\|^2 = \eta \|z\|^2
\end{equation}
Combining \cref{pru1} and \cref{prmean2}
\begin{equation}\label{prumean1}
    \|u\|^2 \leq \frac{\eta}{\lambda_2} \| z\|^2
\end{equation}

Next, we orthogonally decompose $u$ as follows
\[
u = z - P_{\mathrm{class}}z = (I - P_{\mathrm{mean}})z + (P_{\mathrm{mean}} - P_{\mathrm{class}})z
\]
By definition, $(I - P_{\mathrm{mean}})z$ and $(P_{\mathrm{mean}} - P_{\mathrm{class}})z$ are orthogonal, yielding
\[
\|u\|^2 = \|(I - P_{\mathrm{mean}})z\|^2 + \|(P_{\mathrm{mean}} - P_{\mathrm{class}})z\|^2 \geq \|(P_{\mathrm{mean}} - P_{\mathrm{class}})z\|^2
\]
Combining the above with \cref{prumean1} yields the desired result.
\end{proof}

\subsection{Linearization of the Contrastive Loss Near Initialization}

To understand how contrastive learning behaves at the very beginning of training, we look at the loss function when the network's features are still very small (close to zero). In this early stage, we can use a low-order Taylor expansion to approximate the complex, nonlinear loss function with a much simpler quadratic shape.

This quadratic approximation is governed by a specific matrix structure. Essentially, the loss acts as a tug-of-war between two forces: a global repulsion that pushes all features away from each other to spread them out, and an attraction that pulls different augmented versions of the same data point together.

By breaking the loss down this way, we can clearly see the spectral structure that drives early training. It shows us exactly why the eigenvalues and eigenvectors of this matrix control when features begin to separate, setting the stage for the spectral triggers we analyze later. 

Crucially, the geometric connectivity defined in \Cref{assum:augmentation_graph} directly controls the validity of our theoretical results in this regime. In the early training approximation (\Cref{prop:contrastive_taylor} and \Cref{prop:local_ntk_window}), the ratio $\eta/\lambda_2$ explicitly bounds the residual error of the Taylor expansion. A well-connected augmentation graph allows the linear repulsive operator $H$ to cleanly dominate the gradient flow, ensuring the macroscopic features separate robustly.

\begin{proposition}[Linearization of the Contrastive Loss] \label{prop:contrastive_taylor}

Consider the augmented contrastive loss functional from \cref{eq:augmented_contrastive_loss}. Under \Cref{assum:augmentation_graph}, in a restricted neighborhood of the origin $z=0$, the gradient of the loss admits the expansion
\[\nabla_z \widetilde{\mathcal L}(z) 
= - \frac{1}{\tau} Hz + \mathcal{O}\left(\frac{1}{nm\tau} \sqrt{\frac{\eta}{\lambda_2}} \|(I - P_{\mathrm{mean}}) z\| + \frac{\|z\|^3}{\tau^2(nm)^2}\right).\]
Here, $H$ is defined as
\[H = \frac{2}{nm} (P_{\mathrm{class}} - P_{\mathrm{global}}),\]
where $P_{\rm{global}} = \frac{1}{nm} \mathbf{1}^\top \mathbf{1}$ is defined as finds the overall average of every feature across the dataset.
\end{proposition}

\begin{remark}[Dynamics of $H$]
The matrix $H$ is a block matrix that drives the main behavior of the contrastive loss. The first operator, $P_{\mathrm{class}}$, acts as an attractive force that pulls data points toward the center of their own class. The second operator, $P_{\mathrm{global}}$, acts as a repulsive force over all the data. By taking the difference between these two, $H$ effectively creates a dynamic that attracts data points within the same class while repelling data points from different classes.
\end{remark}

Building on this approximation, the next proposition explains feature evolution during early training. If initialized with very small features, the complex training process can be modeled accurately as a simple linear process driven by the NTK. 

\begin{proposition}[Local NTK linearization window] \label{prop:local_ntk_window}

Let $\{x_i\}_{i=1}^N$ be training samples and let $z_{i,a}(t)$ denote the feature values produced by a neural network evolving under gradient flow of the augmented contrastive loss $\dot\theta(t) = -\nabla_\theta \widetilde{\mathcal L}(\theta(t)).$ Let $\varepsilon>0$ and $\tau>0$ be the temperature parameter and let $z(t)\in\mathbb R^{nm}$ denote the stacked feature vector. Assume:

\begin{itemize}
\item[(i)] There exist constants $C_K>0$ and $T_0>0$ such that the kernel matrix satisfies 
\[\|K(t)\| \leq C_K \quad\text{for all } t \in [0,T_0].\]

\item[(ii)] The initialization satisfies $\|z(0)\| \leq \varepsilon,$ where $\varepsilon>0$ is sufficiently small so that $\varepsilon^{3/4} \ll \tau.$ 

\item[(iii)] From \Cref{prop:contrastive_taylor}, and noting that $\|(I-P_{\mathrm{mean}})z\| \le \|z\|$, there exists a constant $C_R$ such that
\[\|\nabla_z \widetilde{\mathcal{L}}(z) + \frac{1}{\tau} Hz\| \leq C_R \left(\frac{\|z\|}{nm\tau}\sqrt{\frac{\eta}{\lambda_2}} + \frac{\|z\|^3}{(nm)^2\tau^2}\right).    \]
\end{itemize}

Define the time window
\[T_\varepsilon:= \min\!\left\{ T_0,\; \frac{nm\tau}{2C_K\left(2 + C_R\sqrt{\frac{\eta}{\lambda_2}}\right)}\log \left(\frac{1}{\varepsilon}\right) \right\}.\]
For all $t\in[0,T_\varepsilon]$, the gradient flow is linearized as follows
\[\dot z(t) =
\frac{1}{\tau} K(t) H z(t) + R(t), \quad \|R(t)\|
\leq
C_K C_R \left(\frac{\varepsilon^{1/2}}{nm\tau}\sqrt{\frac{\eta}{\lambda_2}} + \frac{\varepsilon^{3/2}}{(nm)^2\tau^2}\right).\]
Furthermore, the norm of the features is controlled by
$\|z(t)\|\leq \sqrt{\varepsilon}$ for all $t \in [0, T_\varepsilon].$
\end{proposition}

\Cref{prop:local_ntk_window} confirms that we can safely view the early stages of training entirely through the lens of kernel dynamics. As long as the network starts with small initial values, its complex learning trajectory acts like a straightforward linear process driven by the NTK for a meaningful amount of time.

\subsection{The spectral mechanism for phase transition}

We now introduce the precise mechanism that drives the transition from a messy, entangled dataset to one that is cleanly separated. The core idea is that the evolving kernel dynamics naturally amplify a specific contrastive mode in the data. Once the kernel's spectral alignment, measured by the coercivity metric $\delta$ defined below, reaches a critical positive threshold, this single mode takes over.

\begin{theorem}[Finite-Time Weak Linear Separation]\label{thm:lin_sep}
Suppose the features $z(t)$ evolve under the augmented contrastive loss gradient flow, yielding a local linearization window $[t_*,\, t_* + T_\varepsilon]$ per \Cref{assum:augmentation_graph,prop:local_ntk_window}. 

Let $v$ be the contrast vector that takes the value $1$ on the first class and $-1$ on the second class. Let $v_c = v/\|v\|$ be the normalized contrast direction ($P_{\mathrm{global}} v_c = 0$). Define the inter-cluster signal $\alpha(t)=v_c^\top z(t)$, the intra-cluster residual $w(t)=(I - P_{\mathrm{class}}) z(t)$, and the weak separation margin $\Phi(t)=\alpha(t)-\|w(t)\|$, which measures the discrepancy between the inter-cluster separation and the intra-cluster spread.

Assume the minimum spectral alignment $\delta := \min_{t\in[t_*,t_*+T_\varepsilon]} v_c^\top K(t) v_c$ is sufficiently large such that
\begin{equation}\label{eq:initial_feature_condition}
    \gamma:=\frac{C_K C_R}{2\delta} \left( \sqrt{\frac{\eta}{\lambda_2}} \varepsilon^{1/2} + \frac{nm\,\varepsilon^{3/2}}{\tau} \right) < \min\left\{\alpha(t_*), \sqrt{\frac{\varepsilon}{2}} \right\} .
\end{equation}
Then, there exists a separation time $T_{\mathrm{sep}} \in [t_*,\,t_*+T_\varepsilon]$ such that $\Phi(t)>0$ for all $t\in[T_{\mathrm{sep}},\,t_*+T_\varepsilon]$. Consequently, the classes become weakly linearly separable along $v_c$ within the active time window.
\end{theorem}
The theorem models feature evolution as a competition between the inter-cluster separation signal $\alpha(t)$ and the intra-cluster dispersion noise $\|w(t)\|$. The margin $\Phi(t)$ measures which force dominates (illustrated in \Cref{fig:1d_weak_separation}). The spectral alignment $\delta$ drives the exponential growth of the signal $\alpha(t)$. However, the local linearization introduces an approximation error, bounded by $\gamma$. As shown in \Cref{fig:weak_separation_dynamics}, the condition $\gamma < \alpha(t_*)$ ensures the initial signal overpowers this error, allowing $\alpha(t)$ to grow exponentially within the linearization time window $[t_*, t_*+T_\varepsilon]$.

Because the total feature norm is bounded by $\varepsilon$ within the local time window, the signal and the dispersion are strictly coupled. As $\alpha(t)$ amplifies, the maximum possible noise $\|w(t)\|$ is forced to shrink. This guarantees that the signal overtakes the noise ($\Phi(t) > 0$), pulling the clusters cleanly apart before the linearization window expires.

\begin{figure}[ht]
\centering
\begin{tikzpicture}[scale=1.15, every node/.style={font=\normalsize}]

    \node[anchor=north west, font=\small] at (-5.5, 2.2) {\textbf{Separable Case} (Weak Linear Separation)};
    \draw[->, thick, gray!80] (-5.5, 0) -- (5.5, 0);
    
    \filldraw[black] (0,0) circle (2pt) node[below left=2pt] {\textbf{0}};

    \coordinate (CenterT_B) at (3.0, 0);
    \coordinate (CenterT_R) at (-3.0, 0);
    
    \filldraw[blue] (CenterT_B) circle (3pt) node[below=5pt, text=blue!80!black] {\small Mean $\mu_1$};
    \filldraw[red] (CenterT_R) circle (3pt) node[below=5pt, text=red!80!black] {\small Mean $\mu_2$};

    \foreach \x in {1.6, 2.2, 2.5, 3.1, 3.4, 3.8, 4.2, 4.6, 4.9} { \fill[blue!80!black] (\x, 0) circle (2.5pt); }
    \fill[blue!80!black] (-2.5, 0) circle (2.5pt); 

    \foreach \x in {-4.8, -4.4, -4.0, -3.5, -2.8, -2.2, -1.9, -1.6} { \fill[red!80!black] (\x, 0) circle (2.5pt); }
    \fill[red!80!black] (2.0, 0) circle (2.5pt); 
    \fill[red!80!black] (2.8, 0) circle (2.5pt); 


    \draw[<->, thick, blue!80!black] (-3.5, 1.2) -- (3.4, 1.2) node[midway, fill=white, inner sep=2pt] {\small Signal $\alpha(t)$};

    \draw[<->, thick, gray!80!black] (0.9, 0.7) -- (3.6, 0.7) node[midway, fill=white, inner sep=2pt] {\small $\|w(t)\|$};
    \draw[<->, thick, gray!80!black] (-4.0, 0.7) -- (-1.1, 0.7) node[midway, fill=white, inner sep=2pt] {\small $\|w(t)\|$};

    \draw[<->, thick, green!60!black] (-0.6, 0.4) -- (0.6, 0.4) node[midway, fill=white, inner sep=1pt] {\small $\Phi(t)$};

    \begin{scope}[yshift=-3.5cm]

    \node[anchor=north west, font=\small] at (-5.5, 2.5) {\textbf{Non-Separable Case} ($\Phi(t) \leq 0$, highly jumbled)};

    \draw[->, thick, gray!80] (-5.5, 0) -- (5.5, 0);
    
    \filldraw[black] (0,0) circle (2pt) node[below left=2pt] {\textbf{0}};

    \coordinate (CenterB_B) at (1.0, 0);
    \coordinate (CenterB_R) at (-1.0, 0);
    
    \filldraw[blue] (CenterB_B) circle (3pt) node[below=5pt, text=blue!80!black] {\small Mean $\mu_1$};
    \filldraw[red] (CenterB_R) circle (3pt) node[below=5pt, text=red!80!black] {\small Mean $\mu_2$};

    \foreach \x in {-1.8, -1.2, -0.5, 0.2, 0.6, 1.2, 1.7, 2.2, 2.6, 3.1} { \fill[blue!80!black] (\x, 0) circle (2.5pt); }
    \foreach \x in {-3.2, -2.5, -2.1, -1.6, -1.1, -0.6, -0.1, 0.4, 1.1, 1.9} { \fill[red!80!black] (\x, 0) circle (2.5pt); }


    \draw[<->, thick, blue!80!black] (-1.3, 1.6) -- (1.3, 1.6) node[midway, fill=white, inner sep=2pt] {\small Signal $\alpha(t)$};
    \draw[dotted, gray] (1.0, 0.3) -- (1.0, 1.8);

    \draw[<->, thick, gray!80!black] (-2.4, 1.1) -- (1.4, 1.1) node[midway, fill=white, inner sep=2pt] {\small $\|w(t)\|$ (red)};
    
    \draw[<->, thick, gray!80!black] (-1.4, 0.6) -- (2.4, 0.6) node[midway, fill=white, inner sep=2pt] {\small $\|w(t)\|$ (blue)};
    \end{scope}

\end{tikzpicture}

\caption{1D representation illustrating the contrastive learning dynamics. The inter-cluster signal $\alpha(t)$ reflects the separation between the class means. The dispersion $\|w(t)\|$ represents how dispersed the particles are within the class. \textbf{Top:} Weak linear separability where $\alpha > \|w\|$. A positive margin $\Phi(t) > 0$ exists between the dispersion bounds, even if a few outlier points mix across the boundary. \textbf{Bottom:} When $\alpha \leq \|w\|$, the dispersion overtakes the signal, causing a negative $\Phi(t)$ and heavily jumbled classes.}
\label{fig:1d_weak_separation}
\end{figure}

\begin{figure}[ht]
\centering
\begin{tikzpicture}[font=\normalsize]

    \draw[->, thick] (0, 0) -- (6.5, 0) node[right] {\small $t$};
    \draw[->, thick] (0, 0) -- (0, 4.5) node[above] {\small Amplitude};

    \draw[thick] (0, 2pt) -- (0, -2pt) node[below] {\small $t_*$};
    \draw[dashed, thick, gray] (5.5, 0) node[below, text=black] {\small $t_* + T_\varepsilon$} -- (5.5, 4);

    \fill[red!10] (0, 0) rectangle (6, 0.6);
    \draw[thick, red!80!black, dashed] (0, 0.6) -- (6, 0.6);
    \node[left, text=red!80!black] at (0, 0.3) {\small Error $\gamma$};

    \draw[thick, blue!80!black, smooth] 
        (0, 1.2) .. controls (2.0, 1.3) and (4.0, 2.0) .. (5.5, 4.0)
        node[right] {\small Signal $\alpha(t)$};
    \filldraw[blue!80!black] (0, 1.2) circle (2pt) node[left] {\small $\alpha(t_*)$};

    \draw[thick, orange!80!black, smooth] 
        (0, 3.0) .. controls (2.0, 2.8) and (3.5, 2.2) .. (5.5, 0.8)
        node[right] {\small Noise $\|w(t)\|$};

    \draw[<->, thick, blue!80!black] (0.5, 0.6) -- (0.5, 1.2);
    \node[right, text=blue!80!black, font=\small] at (0.5, 0.9) {$\alpha(t_*) > \gamma$};

    \draw[-stealth, thick, blue!50] (2.5, 2.8) -- (2.5, 3.8);
    \node[right, text=blue!80!black, font=\small, align=left] at (2.5, 3.4) {Growth driven\\by alignment $\delta$};

    \coordinate (cross) at (3.15, 2.15); 
    \filldraw[black] (cross) circle (2pt);
    \draw[dashed, thick] (cross) -- (3.15, 0) node[below] {\small $T_{\mathrm{sep}}$};

    \draw[decorate,decoration={brace,amplitude=5pt,mirror}, thick, green!60!black] 
        (3.15, -0.6) -- (5.5, -0.6) node[midway, below=6pt] {\small Separation: $\Phi(t) > 0$};

\end{tikzpicture}
\caption{The minimum alignment $\delta$ drives the separation signal $\alpha(t)$ to grow exponentially from its initial state. The condition $\gamma < \alpha(t_*)$ guarantees the error floor $\gamma$ does not suppress this initial growth. As the signal outpaces the noise $\|w(t)\|$, the system enters the strictly positive separation zone ($\Phi(t) > 0$) before the local linearization window expires at $t_* + T_\varepsilon$.}
\label{fig:weak_separation_dynamics}
\end{figure}

\begin{proof}
Consider the following decomposition of $z(t)$
\begin{equation}\label{thmpr1}
    z(t) = \alpha(t) v_c + w(t) + c(t)\mathbf{1},
\end{equation}
where $\alpha(t) = v_c^\top z(t)$ is the macroscopic inter-cluster separation amplitude, $w(t) = P_W z(t)$ is the total zero-mean intra-cluster dispersion vector, and $c(t)\mathbf{1} = P_{\mathrm{global}} z(t)$ is the uniform global shift. By definition of the orthogonal projection operators, the vectors $v_c$, $w(t)$, and $\mathbf{1}$ are mutually orthogonal.

Next, we determine how the linear operator $H$ defined in \Cref{prop:contrastive_taylor} affects on each vector. Because $v_c \in U_{\mathrm{class}}$, it follows $P_{\mathrm{class}} v_c = v_c$. Consequently, the operator evaluates to $H v_c = \frac{2}{nm} (P_{\mathrm{class}} - P_{\mathrm{global}}) v_c = \frac{2}{nm} v_c$. The operator completely removes the global mean, yielding $H(c(t)\mathbf{1}) = 0$. By definition, $w(t) \in \operatorname{Range}(I - P_{\mathrm{class}})$, which implies $P_{\mathrm{class}} w(t) = 0$. Because $w(t)$ inherently possesses a zero global mean, $P_{\mathrm{global}} w(t) = 0$. Substituting these identities perfectly removes the cross-coupling, yielding $H w(t) = 0$.

Applying $H$ to \cref{thmpr1} and putting into the linearized gradient flow equation $\dot{z}(t) = \frac{1}{\tau} K(t) H z(t) + R(t)$, we obtain
\begin{equation}
\dot{z}(t) = \frac{2}{nm\tau} \alpha(t) K(t) v_c + R(t).
\end{equation}
We extract the time evolution of the macroscopic separation amplitude $\alpha(t)$ by applying the fixed linear test function $v_c$. Taking the Euclidean inner product provides
\begin{equation*}
\dot{\alpha}(t) = v_c^\top \dot{z}(t) = \frac{2}{nm\tau} \alpha(t) v_c^\top K(t) v_c + v_c^\top R(t).
\end{equation*}
By \Cref{prop:local_ntk_window} and the definition of $\delta$, 
\begin{equation*}
\dot{\alpha}(t) \geq \frac{2\delta}{nm\tau}  \alpha(t) - C_K C_R \left(\frac{\varepsilon^{1/2}}{nm\tau}\sqrt{\frac{\eta}{\lambda_2}} + \frac{\varepsilon^{3/2}}{\tau^2}\right).
\end{equation*}
Solving the inequality, we obtain
\begin{align}\label{thmpralpha}
    \alpha(t) 
    \geq 
    e^{\frac{2\delta}{nm\tau}(t-t_*)}\left(\alpha(t_*) - \gamma\right) + \gamma,
\end{align}
where $\gamma$ is defined in \cref{eq:initial_feature_condition}. By \cref{eq:initial_feature_condition}, $\alpha(t_*) - \gamma > 0$ and thus $\alpha(t)$ increases exponentially.

Next, by orthogonality in \cref{thmpr1},
\[\|z(t)\|^2 = \alpha(t)^2 + \|w(t)\|^2 + nm\,c(t)^2.\]
By \Cref{prop:local_ntk_window}, we have $\|z(t)\| \le \sqrt{\varepsilon}$ on $[t_*,t_*+T_\varepsilon]$. Hence $\|w(t)\|^2 \le \varepsilon - \alpha(t)^2$.
Therefore,
\[\Phi(t)=\alpha(t)-\|w(t)\|
\ge
\alpha(t)-\sqrt{\varepsilon-\alpha(t)^2}.\]
It follows that $\Phi(t)>0$ whenever $\alpha(t) > \sqrt{\varepsilon/2}$. Using \cref{thmpralpha}, this condition is satisfied for $t \geq T_{\mathrm{sep}}$ where 
\[T_{\mathrm{sep}}
=
t_*+
\frac{nm\tau}{2\delta}
\log\!\left(
\frac{\sqrt{\varepsilon/2}-\gamma}{\alpha(t_*)-\gamma}
\right).\]
Finally, to ensure $T_{\mathrm{sep}}\le t_*+T_\varepsilon$, it suffices that
\[\frac{nm\tau}{2\delta}
\log\!\left(
\frac{\sqrt{\varepsilon/2}-\gamma}{\alpha(t_*)-\gamma}
\right)
\le
T_\varepsilon,\]
or equivalently,
\[\frac{\sqrt{\varepsilon/2}-\gamma}{\alpha(t_*)-\gamma}
\le
\exp\!\left(\frac{2\delta}{nm\tau}T_\varepsilon\right).\]
Under this condition, we obtain $T_{\mathrm{sep}}\in[t_*,t_*+T_\varepsilon]$, and for all $t\in[T_{\mathrm{sep}},\,t_*+T_\varepsilon]$, $\Phi(t)>0.$ This completes the proof.
\end{proof}

\Cref{thm:lin_sep} guarantees weak separation ($\Phi(t) > 0$), meaning the inter-cluster signal $\alpha(t)$ exceeds the overall dispersion $\|w(t)\|$. However, this only separates the cluster bulks; extreme outliers may still overlap. The following corollary connects this global metric to strict point-by-point classification, showing exactly how large $\Phi(t)$ must be to guarantee strict pairwise separation between all individual particles.

\begin{corollary}[Strict Pairwise Separation]\label{cor:pairwise_sep}
Under the conditions of \Cref{thm:lin_sep}, suppose the dataset has two clusters, $\mathcal{C}_1$ and $\mathcal{C}_2$. Let $z_k(t)$ be the projection of the $k$-th particle's feature along the contrast axis. For any $t \in [T_{\mathrm{sep}},\,t_*+T_\varepsilon]$, if the weak separation margin $\Phi(t)$ is large enough such that
\begin{equation}\label{eq:phi-condition}
    \Phi(t) > \left( \sqrt{\frac{nm}{2}} - 1 \right) \|w(t)\| , 
\end{equation}
then the classes are strictly separated. That is, for any particle $i \in \mathcal{C}_1$ and $j \in \mathcal{C}_2$, we have $z_i(t) > z_j(t)$.
\end{corollary}

\begin{proof}
Let $v$ be the contrast vector where $v_k = 1$ for $\mathcal{C}_1$ and $v_k = -1$ for $\mathcal{C}_2$. Its length is $\|v\| = \sqrt{nm}$, making the normalized direction $v_c = \frac{v}{\sqrt{nm}}$.

We can break down each particle's feature $z_k(t)$ into a shared global shift $c(t)$, an inter-cluster signal, and an individual noise term $w_k(t)$
\[ z_k(t) = c(t) + \alpha(t)(v_c)_k + w_k(t) . \]
For a particle $i \in \mathcal{C}_1$ ($v_i = 1$) and a particle $j \in \mathcal{C}_2$ ($v_j = -1$), their features are
\begin{align*}
    z_i(t) &= c(t) + \frac{\alpha(t)}{\sqrt{nm}} + w_i(t), \quad
    z_j(t) = c(t) - \frac{\alpha(t)}{\sqrt{nm}} + w_j(t).
\end{align*}
The difference between them is
\[ z_i(t) - z_j(t) = \frac{2\alpha(t)}{\sqrt{nm}} + (w_i(t) - w_j(t)) . \]
By the Cauchy-Schwarz inequality, the maximum difference between any two noise components is bounded by $|w_i(t) - w_j(t)| \leq \sqrt{2}\|w(t)\|$. This gives
\[ z_i(t) - z_j(t) \geq \frac{2\alpha(t)}{\sqrt{nm}} - \sqrt{2}\|w(t)\|= 2\Phi(t) - 2\|w(t)\|\left(\sqrt{\frac{nm}{2}} - 1\right), \]
where the above quantity is bounded below by $0$ from the condition on $\Phi$ in \cref{eq:phi-condition}. This completes the proof.
\end{proof}

\subsection{Continuum Limit of the Latent Dynamics}

Having established the separation dynamics for a finite dataset, we now extend our analysis to the continuum limit where the number of data points tends to infinity while the network parameterization remains finite. In this regime, the empirical feature distribution transitions into a continuous mass density over the latent space $\mathbb{R}^d$, and discrete summations over individual data points become functional integrations.

To formalize the data generation purely as a functional domain, let $\mathcal{X}$ be a base sample space with a source distribution $\mu(x)$, partitioned into $N$ disjoint class domains $\{X_k\}_{k=1}^N$. We define the conditional feature density $\rho_x(z)$ as the normalized mass distribution of latent features generated strictly from the augmentations of a single base sample $x \in \mathcal{X}$ mapped through the neural network. 

The global latent feature density across the entire dataset is thus the marginal integration of these local densities over the source space
\begin{equation*}
    \rho_0(z) = \int_{\mathcal{X}} \rho_x(z) \,\mathrm{d}\mu(x).
\end{equation*}
Define the continuous density ratio $\omega(x | z) = \frac{\rho_x(z) \mu(x)}{\rho_0(z)}$, which represents the mass contribution of the data sample $x$ to the total density at $z$. Similarly, we define the class-conditional latent density for each domain $X_k$ as $\rho_{C_k}(z) = \frac{1}{\mu(X_k)} \int_{X_k} \rho_x(z) \,\mathrm{d}\mu(x)$.

In the continuum limit, the discrete empirical loss in \cref{eq:augmented_contrastive_loss} must transition to a functional over the density space. By replacing the discrete intra-instance and global summations with integrations over the conditional and total densities, the continuous loss is defined as
\begin{equation}\label{eq:population_loss}
    \mathcal{J}[\{\rho_x\}] = \int_{\mathcal{X}} \iint_{\mathbb{R}^d \times \mathbb{R}^d} \rho_x(z) \rho_x(z') \log \left( \frac{\int_{\mathbb{R}^d} \exp\left(-\frac{\|z-\tilde{z}\|^2}{2\tau}\right) \rho_0(\tilde{z}) \,\mathrm{d}\tilde{z}}{\exp\left(-\frac{\|z-z'\|^2}{2\tau}\right)} \right) \,\mathrm{d}z \,\mathrm{d}z' \,\mathrm{d}\mu(x).
\end{equation}
The empirical loss $\widetilde{\mathcal L}(z)$ is exactly recovered from $\mathcal{J}[\{\rho_x\}]$ as $n,m \to \infty$ by substituting the empirical source measure $\mu^{(n)}(x) = \frac{1}{n}\sum_i \delta(x - x_i)$ and the empirical conditional measure $\rho_{x_i}^{(m)}(z) = \frac{1}{m}\sum_a \delta(z - z_{i,a})$.

In the Lagrangian perspective, the feature flow map $Z_t \in L^2(\rho_0; \mathbb{R}^d)$ tracks the trajectory of a mass element starting at $z_0$. We define the continuous orthogonal projection operators $\mathcal{P}_{\mathrm{global}}$, $\mathcal{P}_{\mathrm{mean}}$, and $\mathcal{P}_{\mathrm{class}}$ as integral operators on the Hilbert space $L^2(\rho_0; \mathbb{R}^d)$
\begin{align*} 
    (\mathcal{P}_{\mathrm{global}} Z_t)(z_0) &= \int_{\mathbb{R}^d} Z_t(z') \rho_0(z') \,\mathrm{d}z', \\ 
    (\mathcal{P}_{\mathrm{mean}} Z_t)(z_0) &= \int_{\mathcal{X}} \left[ \int_{\mathbb{R}^d} Z_t(z') \rho_x(z') \,\mathrm{d}z' \right] \omega(x \mid z_0) \,\mathrm{d}\mu(x), \\ 
    (\mathcal{P}_{\mathrm{class}} Z_t)(z_0) &= \sum_{k=1}^N \left[ \int_{\mathbb{R}^d} Z_t(z') \rho_{C_k}(z') \,\mathrm{d}z' \right] \left( \frac{\rho_{C_k}(z_0) \mu(X_k)}{\rho_0(z_0)} \right).
\end{align*}

\begin{assumption}[Continuous Augmentation Overlap Graph]
\label{assum:continuous_augmentation_concentration}
Let $\mathcal{L}_{\mathrm{aug}}$ be the continuous graph Laplacian  operator (the continuum analogue to the discrete graph Laplacian defined in \Cref{sec:geometric_formulation}) acting on $L^2(\rho_0; \mathbb{R}^d)$. We assume the continuous augmentation domain satisfies
\begin{enumerate}
    \item \textbf{Global Connectivity ($\lambda_2$):} The spectrum of the Laplacian operator satisfies a positive gap $\lambda_2(\mathcal{L}_{\mathrm{aug}}) > 0$.
    \item \textbf{Local-Global Variance Bound ($\eta$):} For any mapping $Z$, the Dirichlet energy is bounded by the intra-instance variance: $\langle Z, \mathcal{L}_{\mathrm{aug}} Z \rangle_{L^2(\rho_0)} \leq \eta \|(I - \mathcal{P}_{\mathrm{mean}}) Z\|^2_{L^2(\rho_0)}$.
\end{enumerate}
From \Cref{prop:poincare_dispersion}, it follows that
    \[\| (\mathcal{P}_{\mathrm{mean}} - \mathcal{P}_{\mathrm{class}}) Z \|^2_{L^2(\rho_0)} \leq \frac{\eta}{\lambda_2} \|(I - \mathcal{P}_{\mathrm{mean}}) Z\|^2_{L^2(\rho_0)}.\]
\end{assumption}

With the geometry established in $L^2(\rho_0)$, the following proposition bridges the discrete system to its continuous counterpart, viewing the learning dynamics as a Wasserstein gradient flow of the population functional $\mathcal{J}$.

\begin{proposition}[Mean-Field Limit]\label{prop:mean_field_wgf}
Let $\rho^{(n,m)}_0$ denote the empirical density, and assume $\rho^{(n,m)}_0 \to \rho_0$ in the sense of distributions as $n,m \to \infty$. The continuous density $\rho_t$ satisfies the continuity equation
\begin{equation}\label{eq:eulerian_continuity}
    \partial_t \rho_t = \nabla_z \cdot \left( \rho_t \mathcal{K}_t \nabla_z \frac{\delta \mathcal{J}}{\delta \rho_t} \right),
\end{equation}
where $\frac{\delta \mathcal{J}}{\delta \rho_t}$ is the first variation of the population functional and $\mathcal{K}_t$ is the integral operator defined by the kernel $K_t(z, z') = \nabla_w Z_t(z) \nabla_w Z_t(z')^\top$. Furthermore, suppose the initial feature distribution is bounded such that $\|Z_0\|_{L^2(\rho_0)} \leq \varepsilon$ for some sufficiently small $\varepsilon > 0$. Then, there exists a time window $t \in [0, T]$ during which the features remain bounded by $\|Z_t\|_{L^2(\rho_0)} \leq \sqrt{\varepsilon}$, and the corresponding Lagrangian flow $Z_t(z_0)$ satisfies the linearized descent
\[\dot{Z}_t = \frac{1}{\tau} \mathcal{K}_t \mathcal{H} Z_t + \mathcal{R}_t,\]
where $\mathcal{H} = 2(\mathcal{P}_{\mathrm{class}} - \mathcal{P}_{\mathrm{global}})$ is the macroscopic projection operator, and the continuous residual $\mathcal{R}_t$ is explicitly bounded in the $L^2(\rho_0)$ norm by
\[\|\mathcal{R}_t\|_{L^2(\rho_0)} \leq \mathcal{O}\left( C_K \left[ \frac{\varepsilon^{1/2}}{\tau}\sqrt{\frac{\eta}{\lambda_2}} + \frac{\varepsilon^{3/2}}{\tau^2} \right] \right).\]
\end{proposition}

Building upon this purely spatial formulation, we present the continuum analogue to the main linear separability result from \Cref{thm:lin_sep}. By tracking the continuous flow map $Z_t(z_0) \in L^2(\rho_0)$, we preserve the latent class identities through the continuous density ratios rather than explicit labels. We construct an orthogonal decomposition of the flow map $Z_t$ into a macroscopic separation amplitude aligned with the target contrast direction, a zero-mean intra-cluster dispersion mapping, and a global uniform shift.

\begin{theorem}[Latent Space Weak Linear Separation via Wasserstein Flow]\label{thm:latent_wgf_separation}
Consider the Lagrangian feature flow map $Z_t \in L^2(\rho_0; \mathbb{R}^d)$ and its corresponding linearization over the local time window $t \in [t_*,\, t_* + T_\varepsilon]$, as established in Proposition \ref{prop:mean_field_wgf}.
Let $v_c \in L^2(\rho_0)$ be the normalized contrast direction ($\|v_c\|_{L^2(\rho_0)} = 1$). Define the continuous spectral alignment 
\[\delta := \min_{t\in[t_*,t_*+T_\varepsilon]} \langle v_c, \mathcal{K}_t v_c \rangle_{L^2(\rho_0)}.\]
Assume $\delta > 0$ and that the flow satisfies the initialization condition
\begin{equation}\label{eq:continuous_initial_condition}
\gamma := \frac{C_K C_R}{2\delta} \left( \sqrt{\frac{\eta}{\lambda_2}} \varepsilon^{1/2} + \frac{\varepsilon^{3/2}}{\tau} \right) < \min\left\{\langle v_c, Z_{t_*}\rangle_{L^2(\rho_0)}, \sqrt{\frac{\varepsilon}{2}} \right\} .\
\end{equation}
Define the weak separation margin
\[\Phi(t) = \alpha(t) - \|W_t\|_{L^2(\rho_0)},\]
where $\alpha(t) = \langle v_c, Z_t \rangle_{L^2(\rho_0)}$ is the macroscopic separation amplitude and $W_t = (I - \mathcal{P}_{\mathrm{class}}) Z_t$ captures the intra-cluster dispersion.

Then there exists a time $T_{\mathrm{sep}} \in [t_*,\,t_*+T_\varepsilon]$ such that the classes are weakly linearly separated, i.e., $\Phi(t)>0$ for all $t\in[T_{\mathrm{sep}},\,t_*+T_\varepsilon]$.
\end{theorem}

\begin{proof}
    The proof structurally mirrors that of \Cref{thm:lin_sep}, adapted for the continuous inner products over $\rho_0$. The complete details are deferred to \Cref{app:continuum_proof}.
\end{proof}

The significance of \Cref{thm:latent_wgf_separation} lies in establishing that the geometric separation mechanism is a fundamental, macroscopic property of contrastive learning dynamics. The continuum analogue proves that the onset of clustering is entirely robust to the infinite-data limit. Regardless of whether we track individual data points or the macroscopic evolution of a purely spatial probability measure, the parameterized optimization dynamics inevitably drive the features toward clean linear separability the moment the critical spectral alignment threshold ($\delta > 0$) is crossed.

\section{Numerical Experiments}\label{sec:num}

In this section, we present a collection of numerical experiments designed to illustrate and validate the theoretical mechanisms developed in the previous sections. These experiments serve two complementary purposes. First, they verify the key assumptions underlying our analysis. Second, they demonstrate that the predicted phenomena arise robustly across a wide range of practical settings and architectures.

Our theory predicts that the early training regime depends critically on the initialization scale and the contrastive temperature. In the first subsection, we conduct controlled experiments to check how the ratio between the initialization scale $\varepsilon$ and the temperature $\tau$ affects the final learned representations. These experiments reveal a clear phase transition. When the initialization is sufficiently small relative to the temperature, the dynamics follow the predicted kernel-driven evolution. Conversely, larger initializations cause nonlinear effects to dominate early and prevent the spectral mechanism from emerging.

In the second subsection, we present simple experiments showing how contrastive learning naturally leads to linear separation. We track the evolution of the kernel matrix during training. We demonstrate that a distinct block structure emerges within the kernel matrix. This structural change explicitly aligns with the formation of tightly clustered representations and the emergence of linear separability among different classes.

Finally, in the third subsection, we provide a direct empirical verification of our main theoretical result, \cref{thm:lin_sep}. We investigate the robustness of the spectral trigger mechanism across different data modalities and neural network architectures. We test this on image data using CIFAR-10 with convolutional architectures, continuous physical fields using PDE-generated data with fully connected models, and natural language using IMDB text data processed by Transformer encoders. Across all these diverse settings, the experiments consistently confirm the theoretical prediction: the development of a kernel spectral gap precedes and strictly drives the emergence of linear separability.

\subsection{Verification of the parameter condition for linear separability.}\label{ssec:veri-lin}
The parameter condition established in Proposition \ref{prop:local_ntk_window} defines the operational regime where the linearized gradient flow remains a valid description of the feature dynamics. A central pillar of this analysis is the ratio between the initial intra-cluster feature distance, $\varepsilon$, and the temperature parameter, $\tau$. From a practical standpoint, these two parameters are among the easiest to set and tune during model training. Analytically, their ratio is fundamental because it strictly dictates whether the linearization required for our approximation holds. The following results demonstrate that meaningful separation dynamics only occur when this ratio satisfies our established conditions, validating its critical importance in both theory and practice.

The mechanism underlying this ratio dictates the duration of the stability window $T_\varepsilon$. When the ratio $\varepsilon/\tau$ is sufficiently small, the linearization is highly stable, granting the contrastive dynamics enough time to amplify the separation mode. Conversely, when $\varepsilon$ is too large, the intra-cluster distance becomes too wide to effectively feel the scaling effect from the temperature. Specifically, the local feature updates fail to properly interact with the terms governed by the squared distance over temperature (e.g., proportional to $\|z_i - z_j\|^2/\tau$). Because the local neighborhood does not properly capture this temperature effect, the feature dynamics stall entirely without producing any macroscopic separation.

To illustrate the sensitivity of the training outcome to this condition, Figure \ref{fig:nn-linear-sep-on-off} compares the evolution of the same neural network architecture across distinct parameter regimes, all trained using the same learning rate. In the favorable regime where the ratio $\varepsilon/\tau$ is small, characterized by $\varepsilon=$ 0.02 and $\tau=$ 0.1, the system exhibits early linear separation. This confirms that a stable configuration, supported by a low initial dispersion and moderate temperature, efficiently guides the features into disjoint clusters, achieving linear separation in just 20 iterations.

As the initial dispersion grows, the dynamics transition into a delayed separation regime. For instance, with $\varepsilon=$ 1 and $\tau=$ 0.01, the network eventually achieves a linearly separable state, but the transition is significantly postponed, occuring after 3000 iterations. Furthermore, unlike the clear, distinct clustering achieved in the small ratio regime, this delayed separation is not entirely clean, and noticeable mixing between the clusters is still observed. In the extreme case where the ratio is overly large, such as $\varepsilon=$ 1 and $\tau=$ 0.001, the dynamics fail to take effect. The large initial distance effectively nullifies the required temperature scaling, suppressing the growth of the contrastive mode. In this state, the features remain entangled, and absolutely nothing happens even after 10,000 iterations, demonstrating that linear separation completely fails when the initial feature scale is severely mismatched with the temperature parameter.
\begin{figure}[t]
    \centering

    \begin{minipage}{0.35\linewidth}
        \centering
        \includegraphics[width=\linewidth]{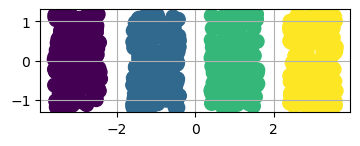}
        \caption*{Clustered input data}
    \end{minipage}

    \vspace{0.5em}

    \begin{minipage}{0.30\linewidth}
        \centering
        \includegraphics[width=\linewidth]{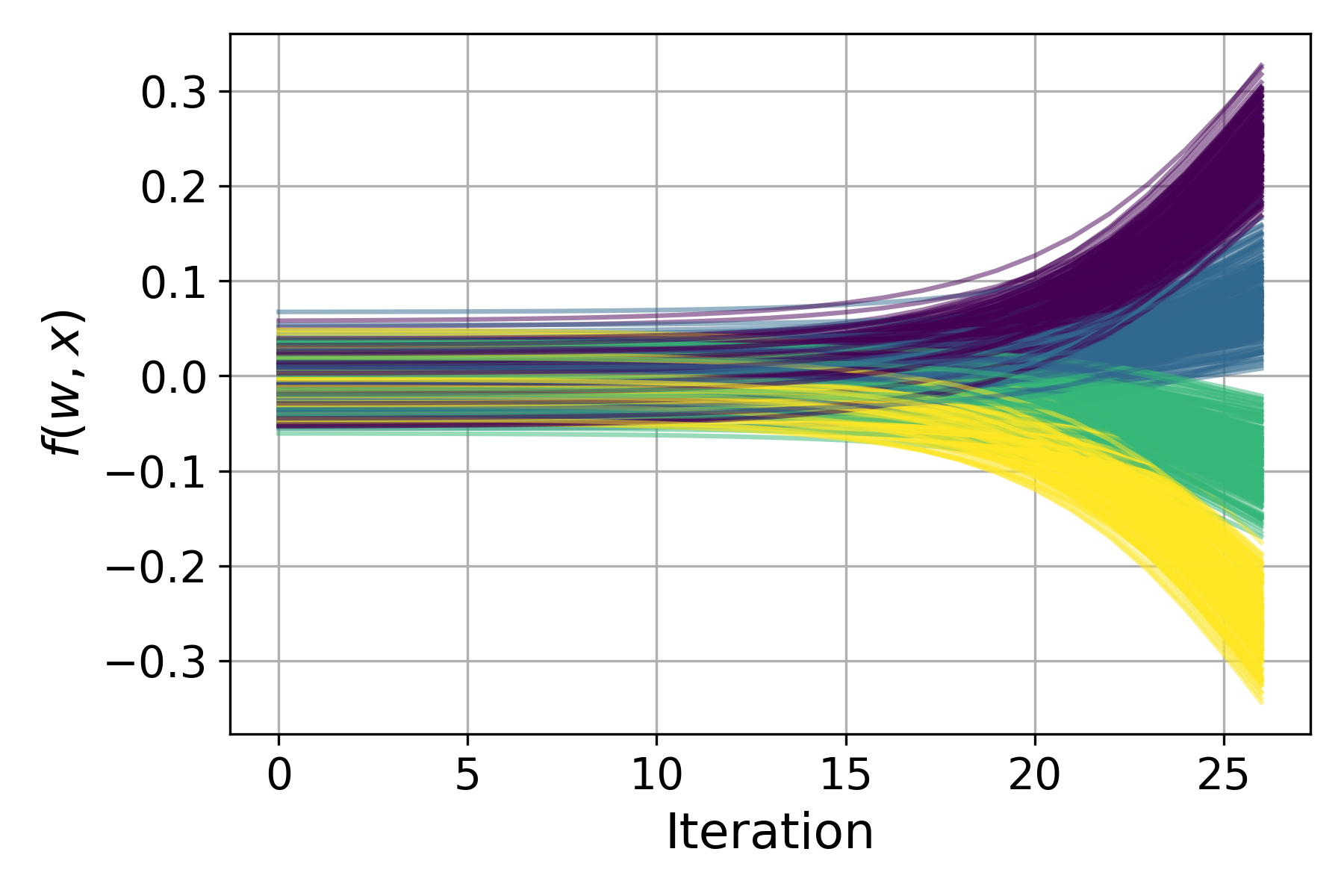}
        \caption*{Small ratio $\varepsilon/\tau$: early separation}
    \end{minipage}
    \hfill
    \begin{minipage}{0.30\linewidth}
        \centering
        \includegraphics[width=\linewidth]{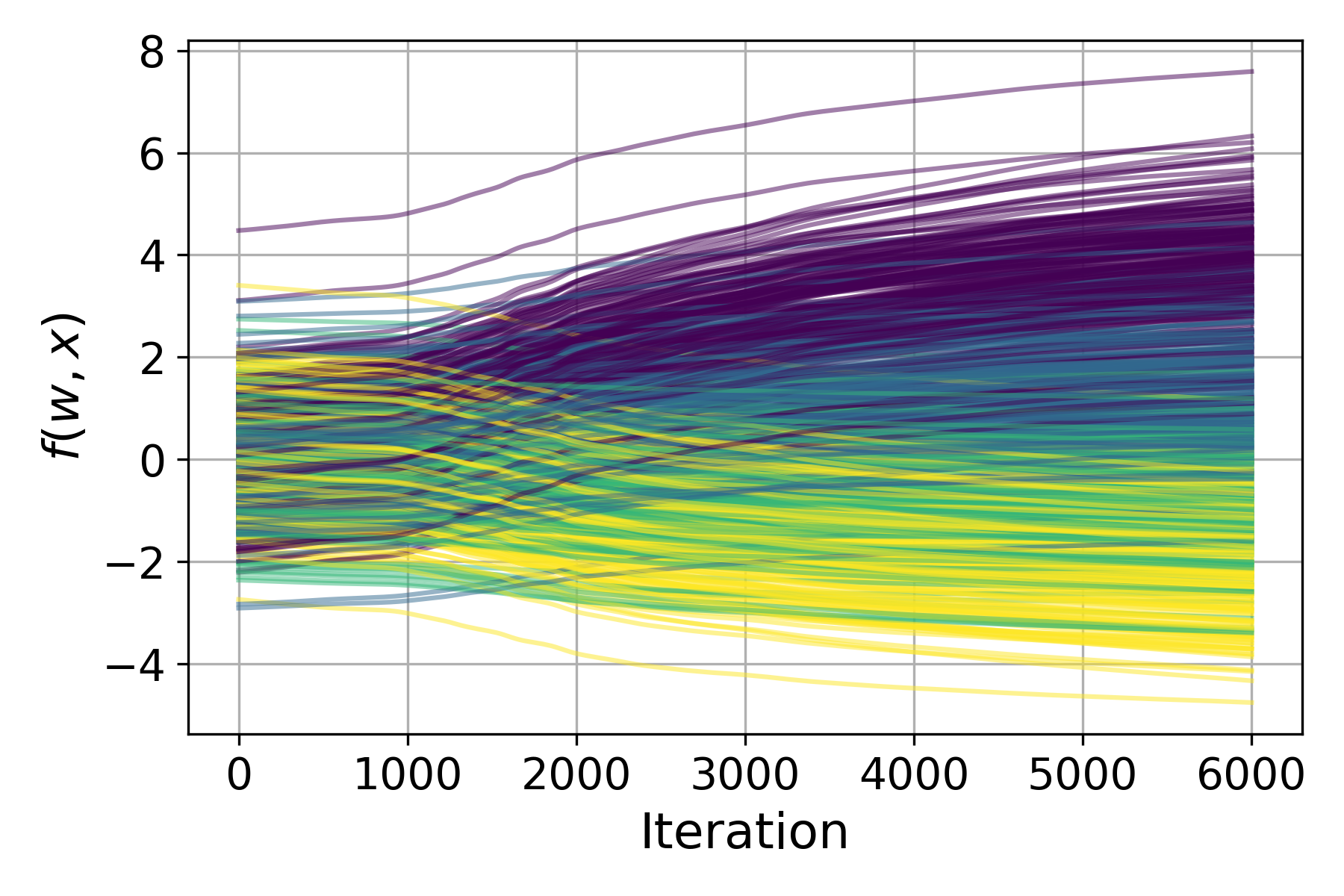}
        \caption*{Large $\varepsilon/\tau$: delayed separation}
    \end{minipage}
    \hfill
    \begin{minipage}{0.30\linewidth}
        \centering
        \includegraphics[width=\linewidth]{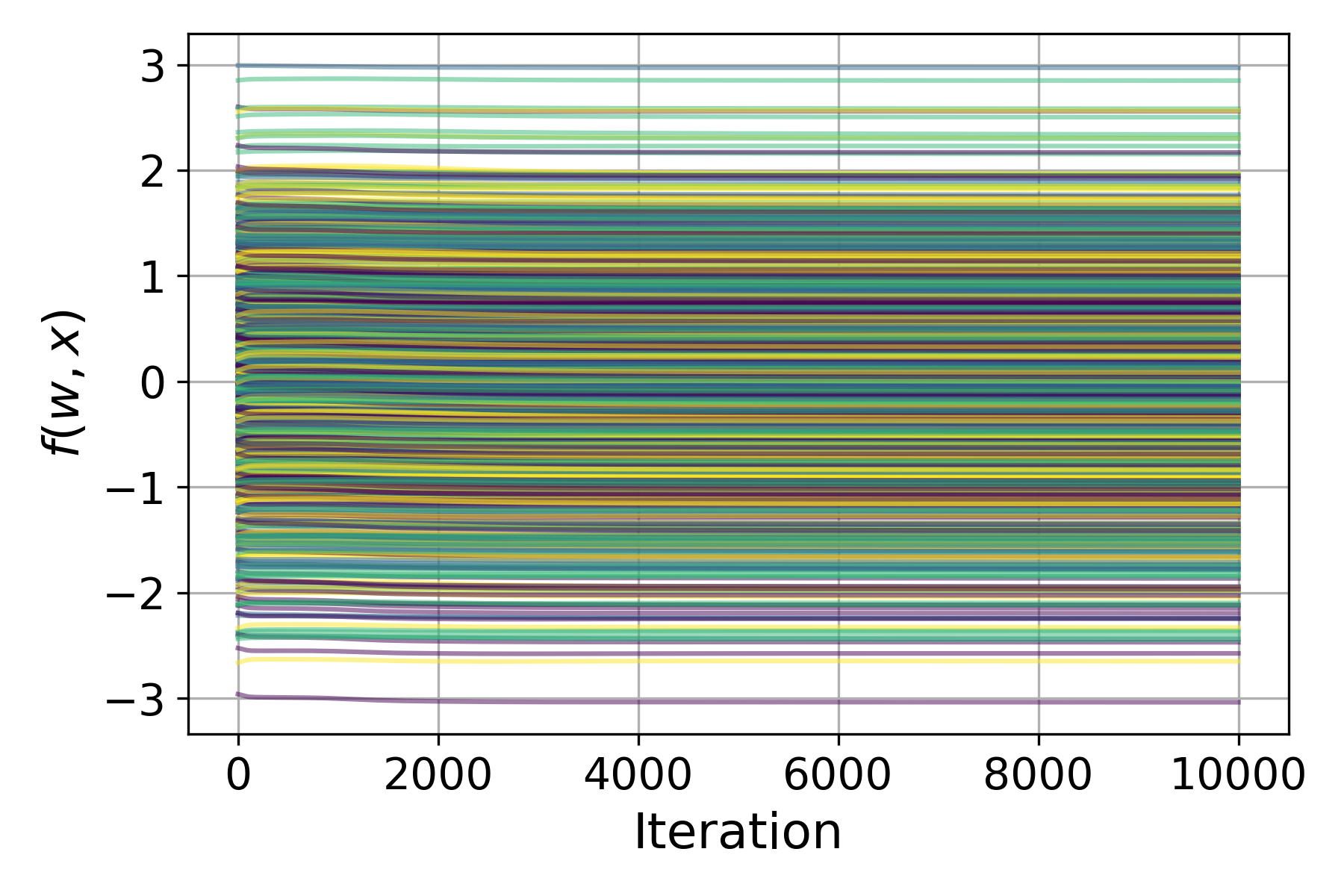}
        \caption*{Larger $\varepsilon/\tau$: no separation in practice}
    \end{minipage}

    \caption{This experiment verifies the parameter condition in \Cref{prop:contrastive_taylor}. The training outcome is governed by the ratio between the initial intra-feature distance $\varepsilon$ and the temperature $\tau$ in the SimCLR loss. When $\varepsilon^{3}/\tau$ is small, the dynamics quickly produce linearly separable features after a short time. When $\varepsilon^{3}/\tau$ is large, separation is delayed and may not occur even after $10{,}000$ iterations.}

    \label{fig:nn-linear-sep-on-off}
\end{figure}

\newlength{\colw}
\newlength{\dataheight}
\setlength{\colw}{0.31\linewidth}      
\setlength{\dataheight}{2.2cm}         

\begin{figure}[t]
\centering

\begin{minipage}[t]{\colw}
\centering
\includegraphics[height=\dataheight,keepaspectratio]{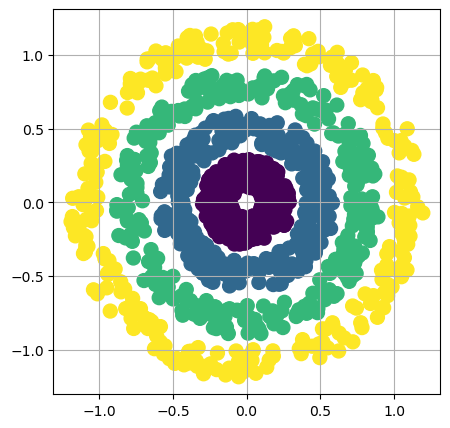}
\par\small Donuts (data)
\end{minipage}\hfill
\begin{minipage}[t]{\colw}
\centering
\includegraphics[height=\dataheight,keepaspectratio]{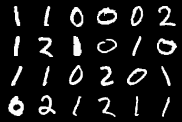}
\par\small MNIST (data)
\end{minipage}\hfill
\begin{minipage}[t]{\colw}
\centering
\includegraphics[height=\dataheight,keepaspectratio]{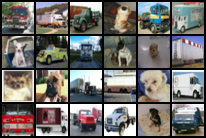}
\par\small CIFAR-10 (data)
\end{minipage}

\vspace{0.6em}

\begin{minipage}[t]{\colw}
\centering
\includegraphics[width=\linewidth]{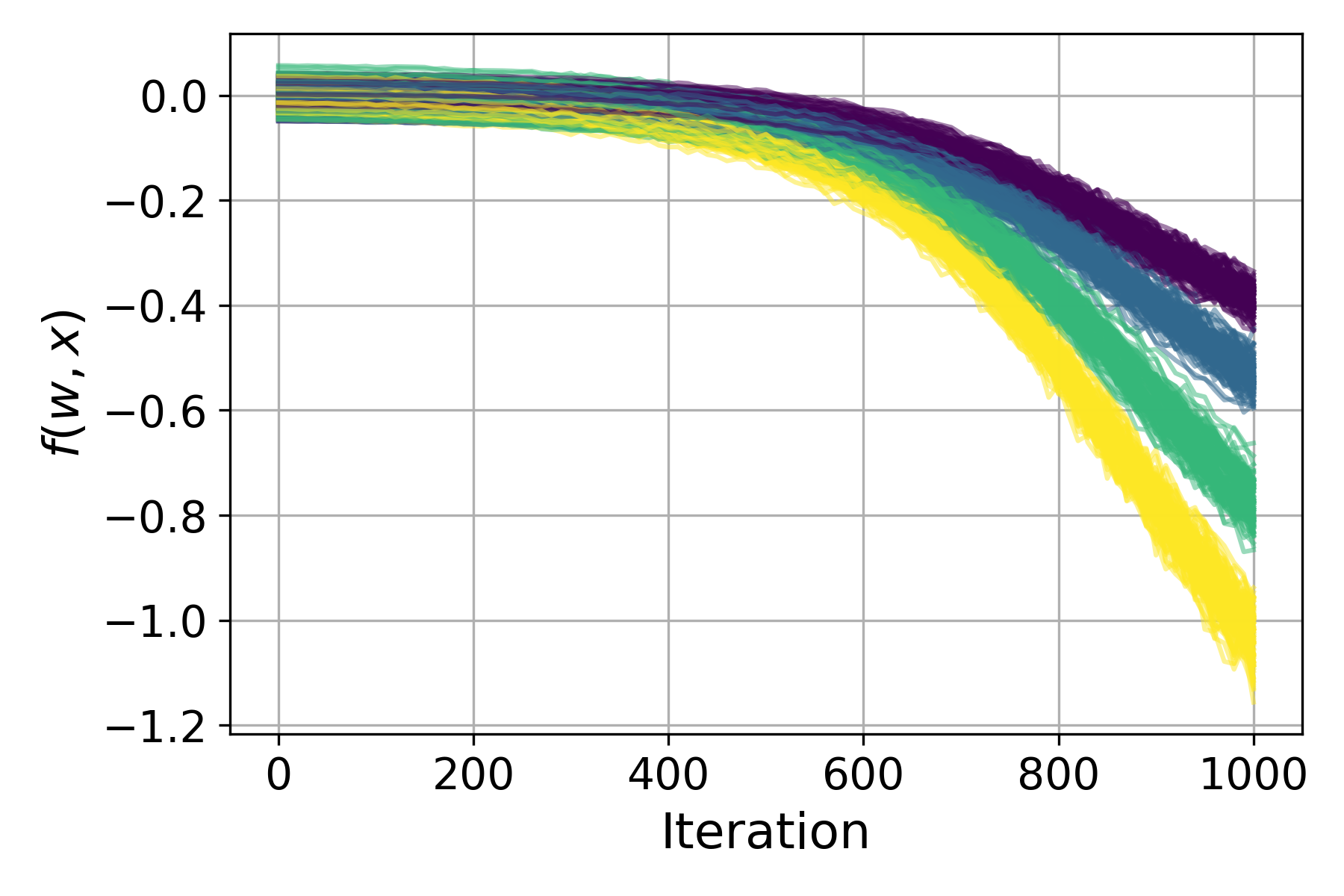}
\par\small 1D embedding (Donuts)
\end{minipage}\hfill
\begin{minipage}[t]{\colw}
\centering
\includegraphics[width=\linewidth]{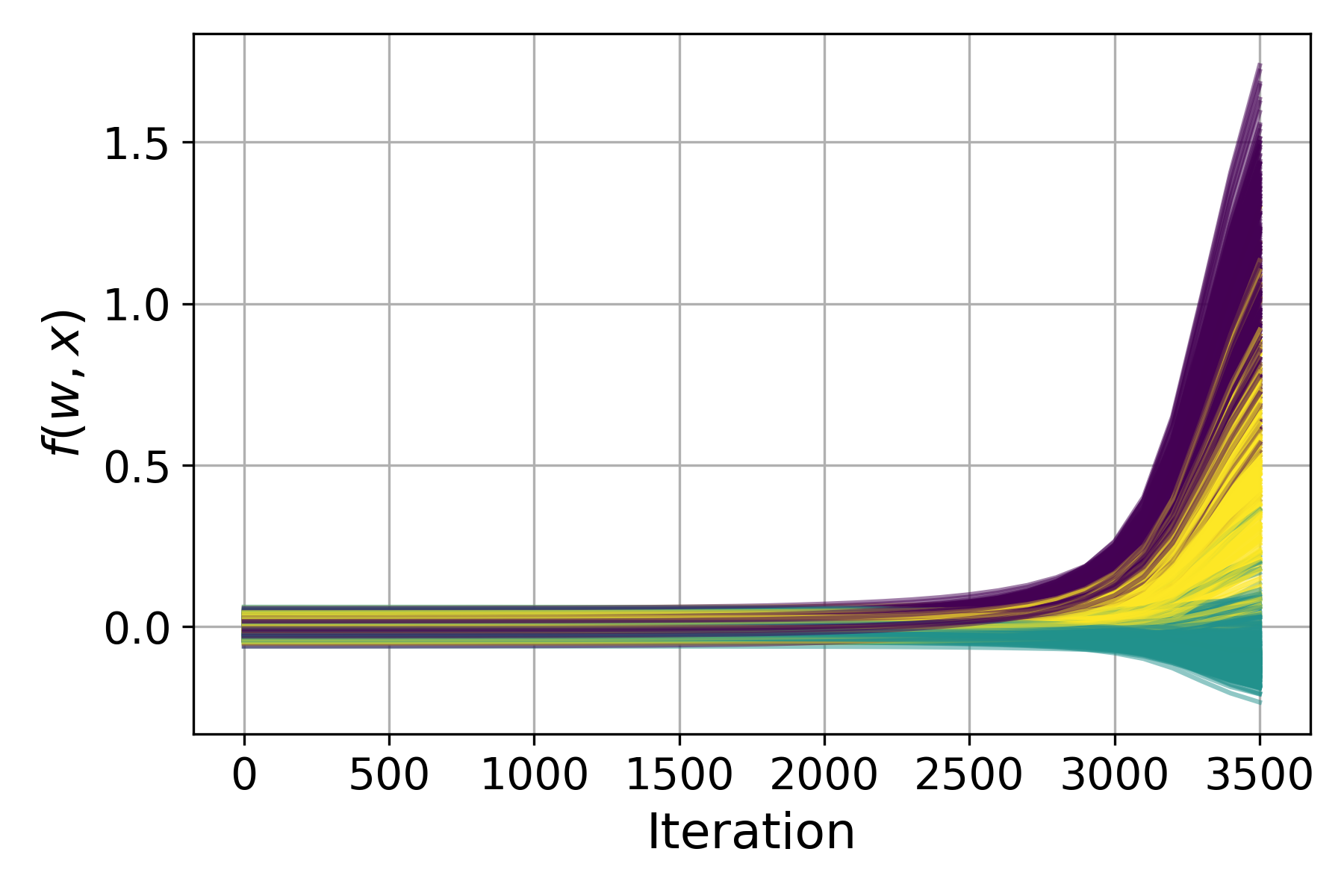}
\par\small 1D embedding (MNIST)
\end{minipage}\hfill
\begin{minipage}[t]{\colw}
\centering
\includegraphics[width=\linewidth]{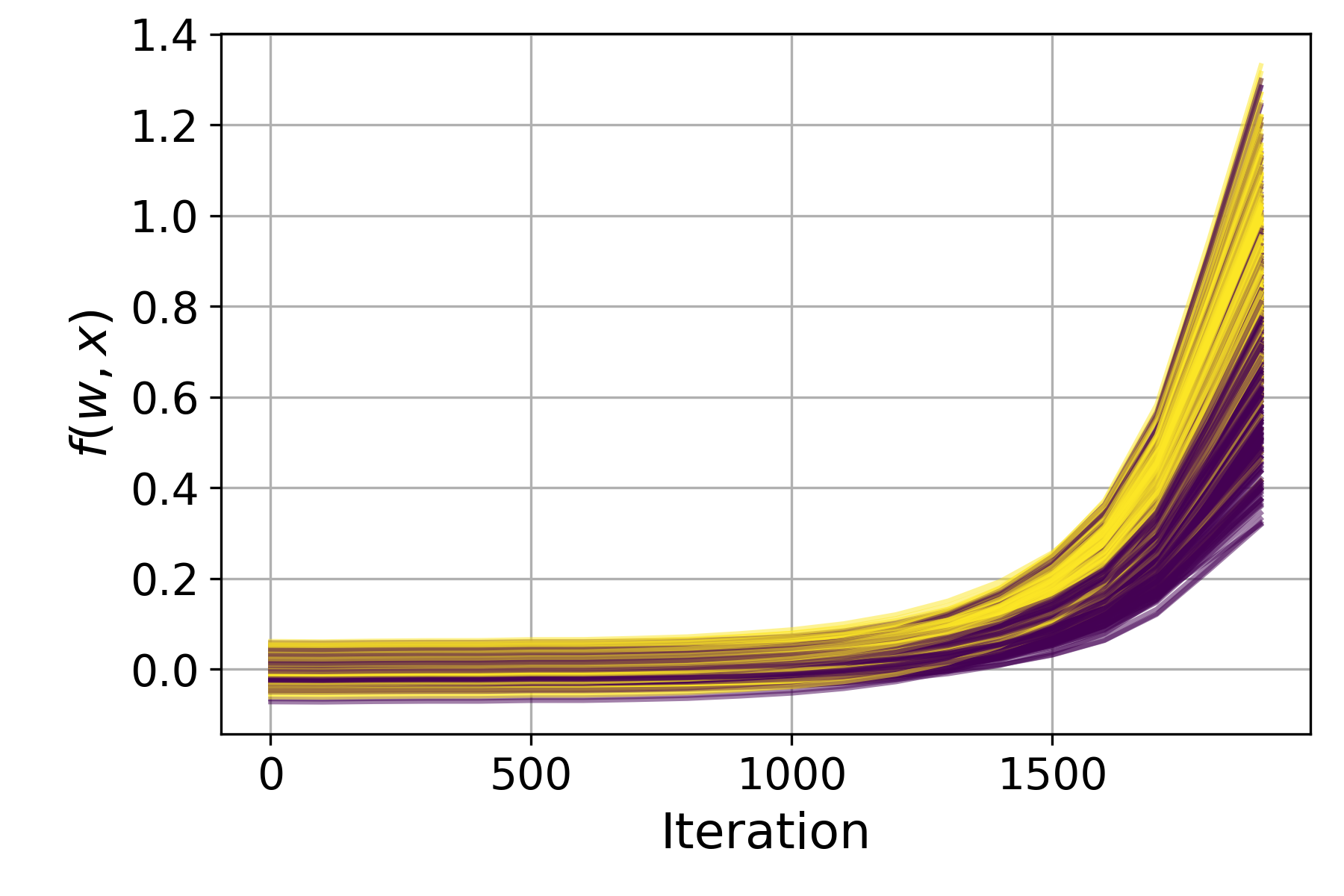}
\par\small 1D embedding (CIFAR-10)
\end{minipage}

\caption{
Top row: visualizations of nonlinear input datasets that are not linearly separable. Bottom row: corresponding one-dimensional embedding trajectories over iterations under SimCLR gradient flow, showing increasing cluster structure and eventual linear separability.
}
\label{fig:nn-linear-sep}
\end{figure}

\begin{figure}[ht!]
\centering

\begin{minipage}[t]{0.24\linewidth}
\centering
\includegraphics[width=\linewidth]{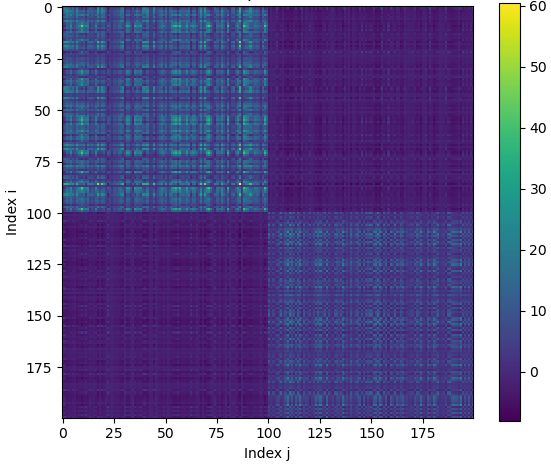}
\par\small
Linear data\\Iteration 0
\end{minipage}\hfill
\begin{minipage}[t]{0.24\linewidth}
\centering
\includegraphics[width=\linewidth]{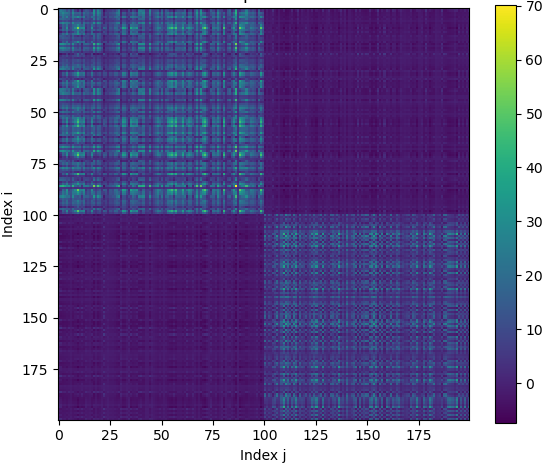}
\par\small
Linear data\\Iteration 5000
\end{minipage}\hfill
\begin{minipage}[t]{0.24\linewidth}
\centering
\includegraphics[width=\linewidth]{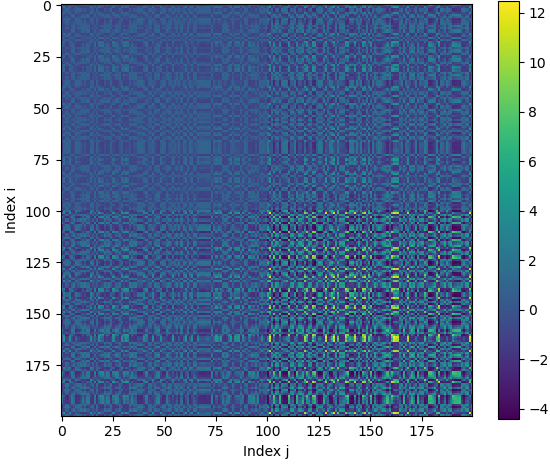}
\par\small
Nonlinear data\\Iteration 0
\end{minipage}\hfill
\begin{minipage}[t]{0.24\linewidth}
\centering
\includegraphics[width=\linewidth]{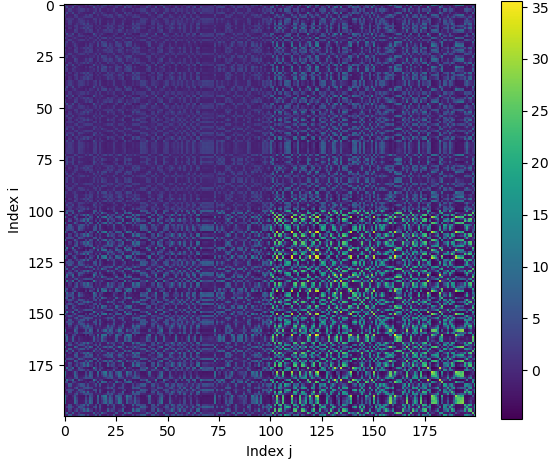}
\par\small
Nonlinear data\\Iteration 5000
\end{minipage}

\caption{
Evolution of the kernel matrix $K$ in \cref{eq:kernel-K} for two-cluster datasets.
Linear data: clear $2\times2$ block structure at $t=0$ and mild refinement by 5000 iterations.
Nonlinear donut: no block structure at $t=0$ but pronounced blocks after 5000 iterations.
}
\label{fig:kernel-evolution}
\end{figure}

\subsection{The Role of Augmentation Design}
\label{sec:augmentation_experiments}

To empirically validate the necessity of the geometric and distributional assumptions placed on the augmentation mapping $T(x)$, we construct a synthetic multi-class experiment. We demonstrate how violating the continuous augmentation overlap graph assumptions in \Cref{assum:continuous_augmentation_concentration} impacts the structural connectivity of the latent space and the resulting linear separability.

We consider a 3-class classification problem where the data is distributed across three concentric rings (radii $r \in \{0.4, 1.0, 1.6\}$). Each class consists of $50$ points drawn uniformly over its respective radial domain, with a slight structural noise (width factor of $0.01$). Our network $f_\theta$ is a 4-hidden-layer Multi-Layer Perceptron (MLP) with a hidden dimension of $50$ and LeakyReLU activations (negative slope $\alpha = 0.2$). The model is trained using Stochastic Gradient Descent (SGD) with a learning rate of $0.01$ for $5,000$ iterations. We optimize a contrastive loss in \cref{eq:cost-dis-new}, where similarity is defined via $e^{-|x-y|^2/\tau}$, with a temperature parameter $\tau = 0.1$.

To isolate the effect of the augmentation strategy, we compare four distinct augmentation mechanisms (visualized in Figure \ref{fig:aug_visuals})
\begin{itemize}
    \item \textbf{Full Rotation (Ideal):} Points are rotated by a uniformly sampled angle $\theta \sim [-\pi, \pi]$. This forms a completely connected, closed-loop augmentation graph.
    \item \textbf{Small Rotation:} Points are rotated uniformly within a narrow bounded range $\theta \sim [-\pi/4, \pi/4]$.
    \item \textbf{Gaussian Noise:} Points are corrupted by additive Gaussian noise $\epsilon \sim \mathcal{N}(0, 0.2^2 I)$.
    \item \textbf{Half-Circle Rotation:} Points are rotated uniformly, but strictly within their respective spatial half-circle (e.g., bounded within $y > 0$ or $y < 0$).
\end{itemize}
To explicitly connect these strategies to our theoretical framework, recall \Cref{assum:augmentation_graph}, which bounds the intra-class augmentation variance via the continuous Poincaré inequality. The ratio between the local feature smoothness $\eta$ and the spectral gap of the augmentation graph $\lambda_2$ acts as a geometric penalty
\[\| (\mathcal{P}_{\mathrm{mean}} - \mathcal{P}_{\mathrm{class}}) z \| \leq \sqrt{\frac{\eta}{\lambda_2}} \|(I - \mathcal{P}_{\mathrm{mean}}) z\|.\]

A critical phenomenon in contrastive learning is that the network will naturally gravitate toward learning the representational structure (the definition of a "class") that provides the path of least geometric resistance. The magnitude of the penalty ratio $\eta / \lambda_2$ determines whether the continuous Wasserstein flow can successfully compress the latent representations into the target classes:

\begin{itemize}
    \item \textbf{Full Rotation:} Because the augmentation domain covers all possible orientations, it forms a perfectly connected graph for the entire class with no boundaries. This yields a large spectral gap $\lambda_2$. Consequently, the geometric penalty ratio $\eta / \lambda_2$ is minimal. The network easily overcomes this small resistance, rapidly compressing the intra-cluster dispersion and separating the dataset into the 3 true underlying classes ($<1,000$ iterations).
    
    \item \textbf{Small Rotation ($[-\pi/4, \pi/4]$):} The restricted augmentation neighborhood means the overlap between distinct base samples is heavily bottlenecked. The graph is technically connected along the ring, but stretched thin, resulting in a small spectral gap $\lambda_2$. This drives the penalty ratio $\eta / \lambda_2$. This geometric resistance suppresses the macroscopic separation amplitude $\alpha(t)$ early in training. However, because the graph is still fundamentally connected, the continuous flow eventually overcomes this penalty. The representations ultimately separate into the 3 classes, but the convergence is delayed ($\sim 2,500$ iterations).
    
    \item \textbf{Gaussian Noise:} It is crucial to differentiate the behavior here from the small rotation. Isotropic Gaussian noise scatters points radially in 2D space, which fundamentally breaks the 1D structural overlap between neighboring points along the ring. The probability of distinct samples' augmented views meaningfully overlapping is exponentially small, meaning $\lambda_2$ functionally approaches zero. If the penalty ratio $\eta / \lambda_2$ is too massive, the attractive forces in the contrastive loss are completely overwhelmed. Unable to pull the disconnected graph together to form the 3 target classes, the separation process fails entirely. The network minimizes its loss by shattering the representations into instance-level noise.
    
    \item \textbf{Half-Circle Rotation:} The bounded 180-degree rotation physically fails to bridge the gap between base samples on opposing sides of the ring. The augmentation graph physically disconnects into isolated subgraphs. Across the entire ring, the global spectral gap is exactly zero ($\lambda_2 = 0$), making the penalty to form a single 3-class structure strictly infinite. However, within each isolated half-circle subgraph, local connectivity yields a $\lambda_2 > 0$. The network takes the path of least resistance by satisfying the continuous bound locally, fundamentally learning $6$ distinct clusters (splitting each of the 3 rings in half) rather than 3.
\end{itemize}

\begin{figure}[htbp]
    \centering
    \includegraphics[width=\textwidth]{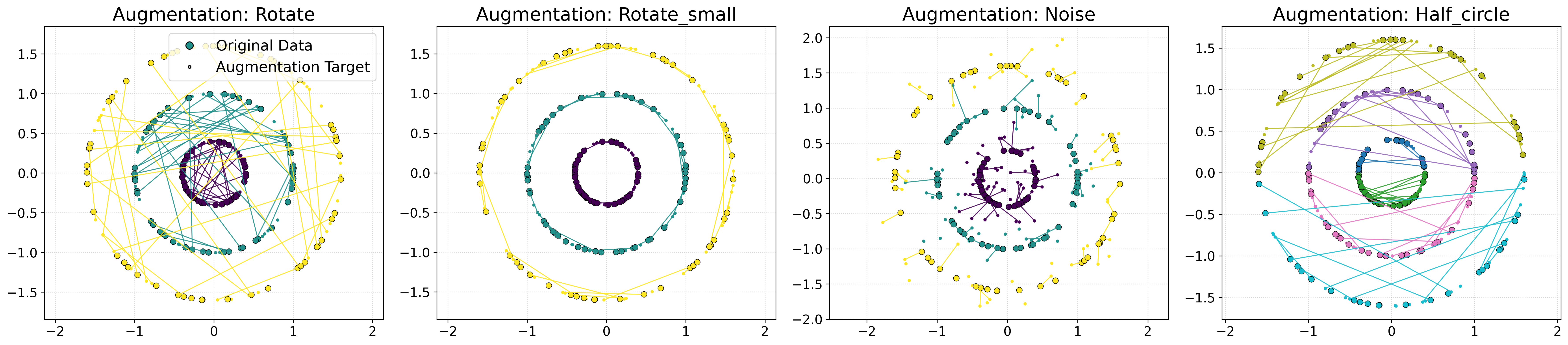}
    \caption{Spatial visualization of the augmentation strategies. Original data points are plotted as colored circles, while the augmented targets are shown in faint grey. Only the full rotation provides uniform global coverage over the underlying class structures.}
    \label{fig:aug_visuals}
\end{figure}

The choice of augmentation strictly governs the evolution of the inter-class and intra-class feature distances. Figure \ref{fig:margin_comp} tracks the continuous separation margin $\Phi$ (defined as the macroscopic separation amplitude minus the intra-cluster dispersion) over the training process. 

Under the ideal full rotation, the margin grows robustly and immediately due to the low $\eta/\lambda_2$ penalty. Under the small rotation, the model ultimately achieves the same margin, but its trajectory is heavily delayed by the high geometric resistance. Finally, when the augmentation severs the underlying graph entirely (Gaussian Noise), the massive penalty prevents the margin from ever growing, causing structural separation to fail.

\begin{figure}[htbp]
    \centering
    \includegraphics[width=0.5\textwidth]{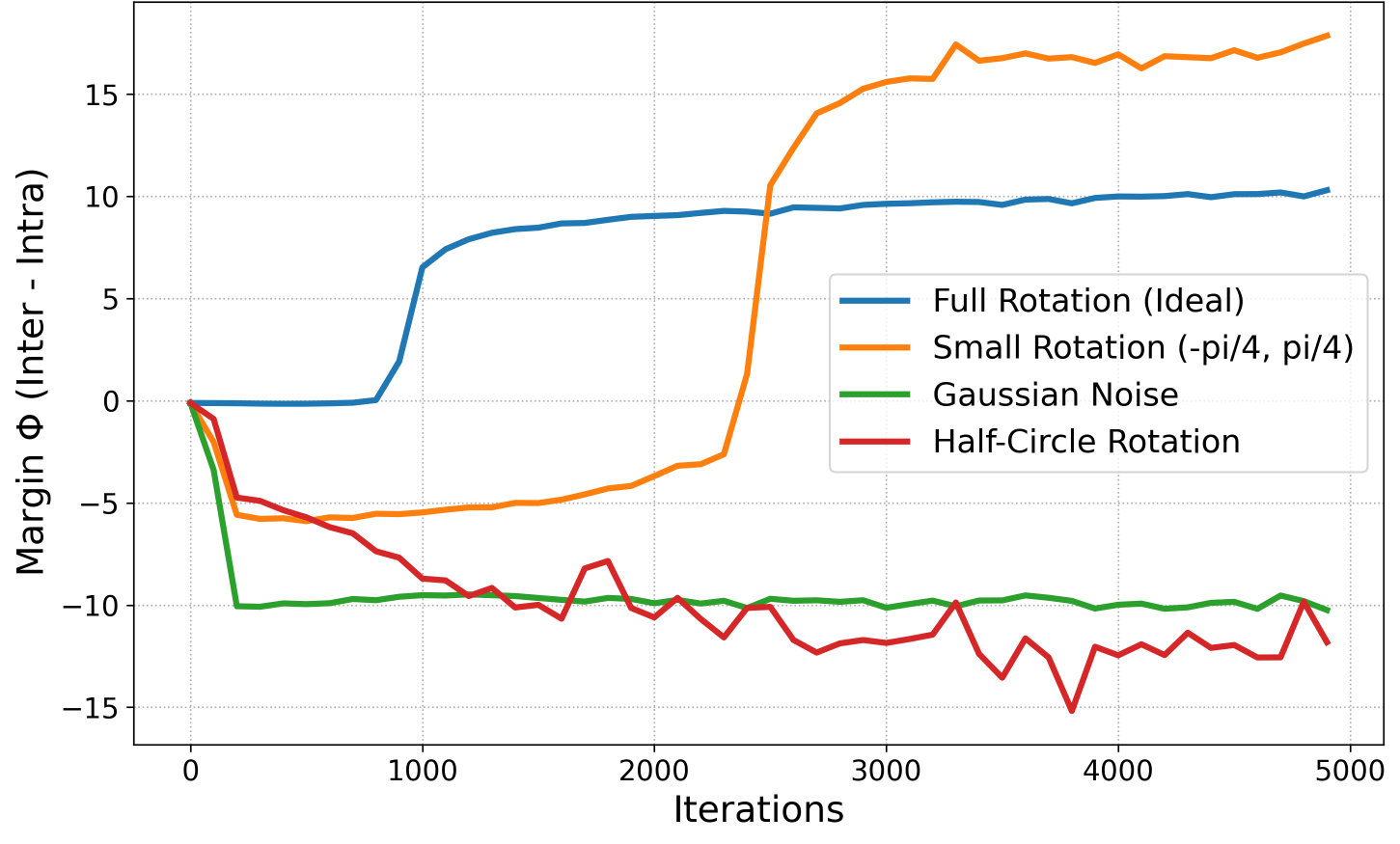}
    \caption{Evolution of the linear separability margin $\Phi$ across the augmentation strategies. Full Rotation maximizes separation rapidly, while Small Rotation exhibits a mathematically predictable delay due to a large $\eta/\lambda_2$ ratio. Severing the graph entirely (Gaussian Noise) prevents any inter-class structural separation.}
    \label{fig:margin_comp}
\end{figure}

These macroscopic margin effects are directly driven by the microscopic particle dynamics of the 1D function outputs, shown in Figure \ref{fig:particle_evol}. For the full rotation, the outputs exhibit a clear and rapid divergence into three distinct bands. For the small rotation, these three bands also form, but at a significantly later stage due to the bottlenecked flow. For the Gaussian noise, the lack of graph connectivity induces chaotic, instance-level trajectories with no coherent clustering. Finally, the half-circle rotation forces the embeddings to diverge into six distinct bands, visually confirming that when $\lambda_2 = 0$ globally, the disconnected augmentation subgraphs dictate the learned invariant topology.

\begin{figure}[htbp]
    \centering
    \includegraphics[width=\textwidth]{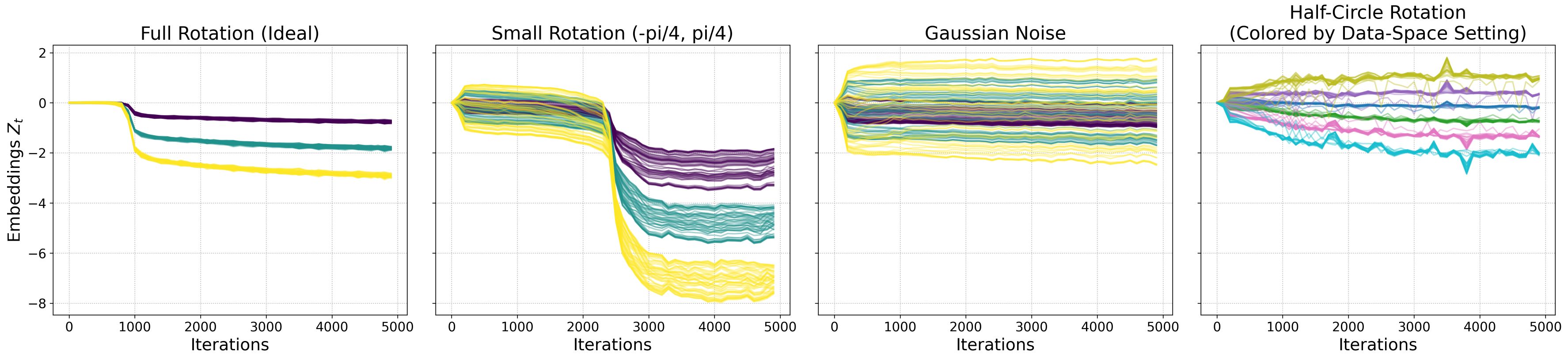}
    \caption{Particle evolution of embeddings over $5,000$ iterations. The full rotation clusters the three classes rapidly. The bottlenecked graph of the small rotation severely delays this clustering. Gaussian noise destroys the structural overlap entirely, causing instance-level shattering, while the half-circle rotation splits the dataset into six disconnected representations.}
    \label{fig:particle_evol}
\end{figure}

\subsection{Emergence of linear separability from complex datasets.}
\Cref{fig:nn-linear-sep} shows three experiments where features become linearly separable under gradient flow of the SimCLR loss, starting from datasets that are nonlinear and more complex than the well-clustered case in \Cref{fig:nn-linear-sep-on-off}. The networks are: a one-hidden-layer MLP with 100 neurons, a three-hidden-layer MLP with 100 neurons per layer, and a small CNN with two convolutional layers followed by one fully connected layer. The datasets are an artificial donut, MNIST (three classes), and CIFAR-10 (two classes). Augmentations are as follows: for the donut, a random rotation in which $T(x)$ is sampled uniformly on the circle centered at the origin with radius $\|x\|$; for MNIST, Gaussian blur, random rotation, and resizing; for CIFAR-10, Gaussian blur, random rotation, resizing, and color jittering.
In all runs, features are initialized uniformly in an $\varepsilon$-ball and then evolved by gradient flow. Each curve traces one particle and is colored by its cluster label. Across all settings, the learned features become increasingly linearly separable, indicating that the separation mechanism extends beyond the theory’s simplifying assumptions. In particular, condition (c) in \Cref{thm:lin_sep} may not hold at initialization for these complex datasets.

To connect these observations with the theorem, we examine the kernel $K$ in \cref{eq:kernel-K}. Condition (c) requires $K$ to be approximately block structured. Empirically, we observe that if this structure is absent at initialization, training first reshapes $K$ toward it, which typically adds iterations.

We visualize this transition in \Cref{fig:kernel-evolution} using two clusters so that the target pattern is a $2\times 2$ block. Panels (a) and (b) show the linearly separable case from \Cref{fig:nn-linear-sep-on-off} with two clusters. A clear block structure is already visible at iteration 0 because the data are well clustered (see \Cref{ssec:veri-lin}), and it refines only modestly by iteration 5000. Panels (c) and (d) show the donut dataset, which is not well clustered: the initial kernel shows little separation, but by iteration 5000 a sharp block structure emerges. These visuals link the empirical separation in \Cref{fig:nn-linear-sep} to the hypothesis in \Cref{thm:lin_sep}: once training sculpts $K$ into the required block form, linear separability follows.

\subsection{Empirical Validation of Coercivity and Linear Separability}\label{sec:main-numeric}

To further empirically validate our main theoretical results from \Cref{thm:lin_sep}, we evaluate our method across four distinct modalities. Across all experiments, we focus on binary clustering tasks using non-linear datasets. It is crucial to emphasize that the inherent data distributions in these domains are highly complex. Consequently, embedding such intricate, non-linear structures into a 1D space to achieve linear separation is an exceptionally challenging task. Despite this inherent difficulty, our results demonstrate the remarkable power of contrastive learning in successfully discovering these 1D representations.

\Cref{fig:main_results} illustrates the training dynamics for each domain. For every dataset, we plot the evolution of the coercivity metric $\delta(t)$, the linear separability metric $\Phi$, and the corresponding 1D particle evolution. Consistently across all four settings, regardless of the underlying architecture or data modality, we observe that $\delta(t)$ sharply increases prior to the onset of linear separation (indicated by the subsequent rise in $\Phi$). This empirical observation directly corroborates our main theorem, demonstrating that the system consistently achieves strong coercivity before the clusters become linearly separable.

The four experimental domains, visualized in \Cref{fig:four_datasets}, are structured as follows
\begin{enumerate}[(a)]
    \item Synthetic Dataset (Donut): We constructed a 2D synthetic dataset consisting of two concentric donut-shaped clusters with different radii. The neural network utilizes fully connected layers. To preserve the topological structure of the clusters, data augmentation is strictly limited to rotational shifts (angle changes).
    \item CIFAR-10: We applied a convolutional neural network architecture to the CIFAR-10 dataset. Data augmentation incorporated standard spatial and visual transformations, specifically rotation, translation, resizing, and blurring.
    \item Partial Differential Equations (PDE): We generated a physics-based dataset using the heat equation. The two clusters represent different initial input frequencies: low-frequency Fourier modes (Label 0) and high-frequency Fourier modes (Label 1). These initial fields are then subjected to heat evolution smoothing. Augmentation for this domain consists of spatial translations combined with forward heat equation steps.
    \item Text (IMDB) \cite{maas2011learning}: We implemented a Transformer architecture to classify text data, clustering IMDB movie reviews into positive and negative sentiments.
\end{enumerate}

\begin{figure}[htbp]
    \centering
    \includegraphics[width=0.97\textwidth]{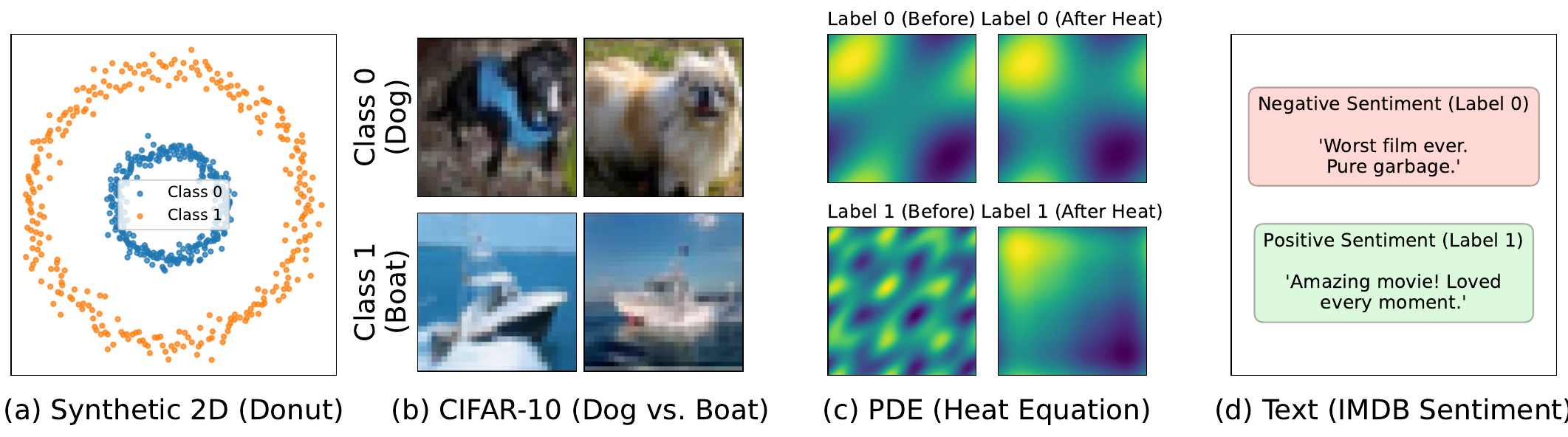}
    \caption{Visualizations of the four experimental datasets: (a) 2D synthetic concentric donuts, (b) CIFAR-10 samples for the dog and boat classes, (c) low- and high-frequency Fourier modes before and after forward heat evolution, and (d) IMDB text sentiment examples.}
    \label{fig:four_datasets}
\end{figure}

\begin{remark}[Connection to the Backward Heat Equation]
For the PDE domain in particular, the contrastive learning task involves distinguishing the original input frequencies from data that has been smoothed by forward heat evolution. Analytically, reversing the heat equation is a famously difficult inverse problem. While Theorem 11 in Section 2 of \cite{evans2022partial} establishes the backward uniqueness of the heat equation, ensuring that the initial states are theoretically distinct, practical recovery is highly unstable due to the severe loss of high-frequency information. Although the specific dataset we generated for this experiment is only moderately challenging, the underlying inverse problem remains inherently difficult. The fact that the network successfully separates these diffused states still highlights how the spectral alignment mechanism can extract and amplify the surviving signals despite this smoothing effect.
\end{remark}

Comprehensive details regarding network architectures, hyperparameters, data generation procedures (including the precise Fourier mode configurations for the PDE dataset), and augmentation strategies are provided in the Appendix.

\begin{figure}[htbp]
    \centering
    \begin{minipage}[b]{0.45\linewidth}
        \centering
        \includegraphics[width=\linewidth, height=10cm]{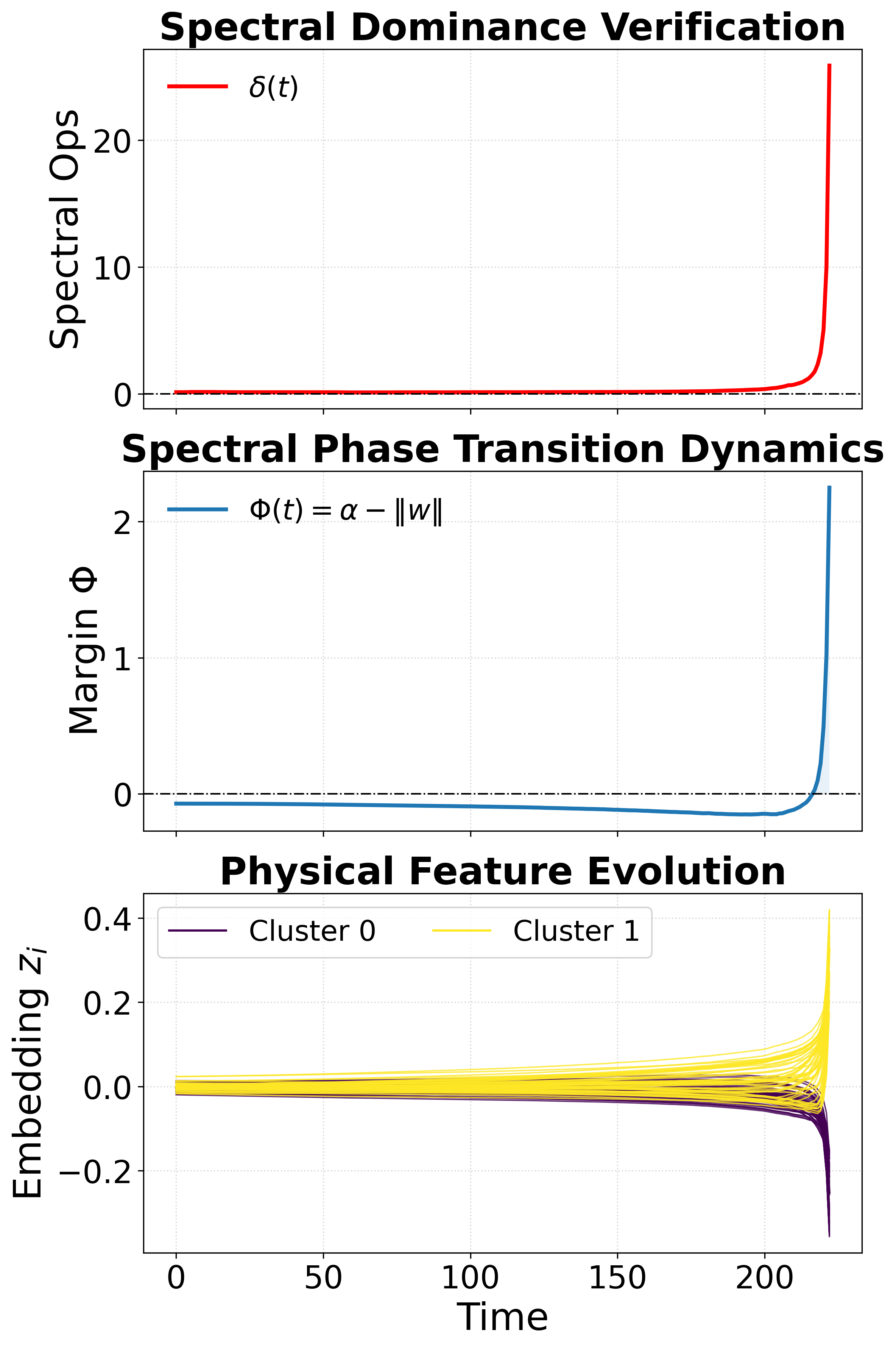}

        \vspace{-0.2cm}
        \centerline{\footnotesize (a) Synthetic Dataset (Donut)}
        \label{fig:result_donut}
    \end{minipage}
    \begin{minipage}[b]{0.45\linewidth}
        \centering
        \includegraphics[width=\linewidth, height=10cm]{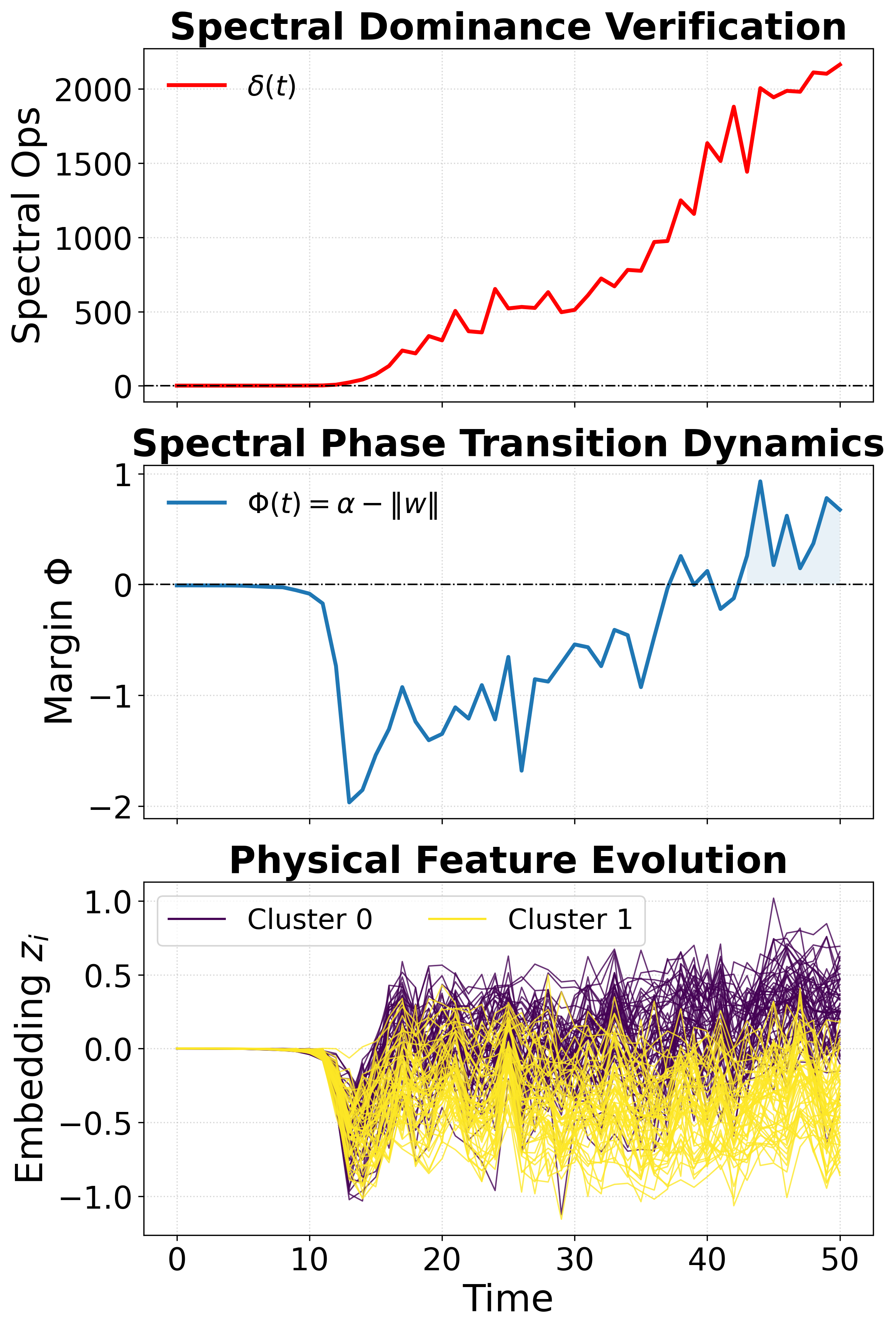}

        \vspace{-0.2cm}
        \centerline{\footnotesize (b) CIFAR-10}
        \label{fig:result_cifar}
    \end{minipage}

    \vspace{0.6em} 
    
    \begin{minipage}[b]{0.45\linewidth}
        \centering
        \includegraphics[width=\linewidth, height=10cm]{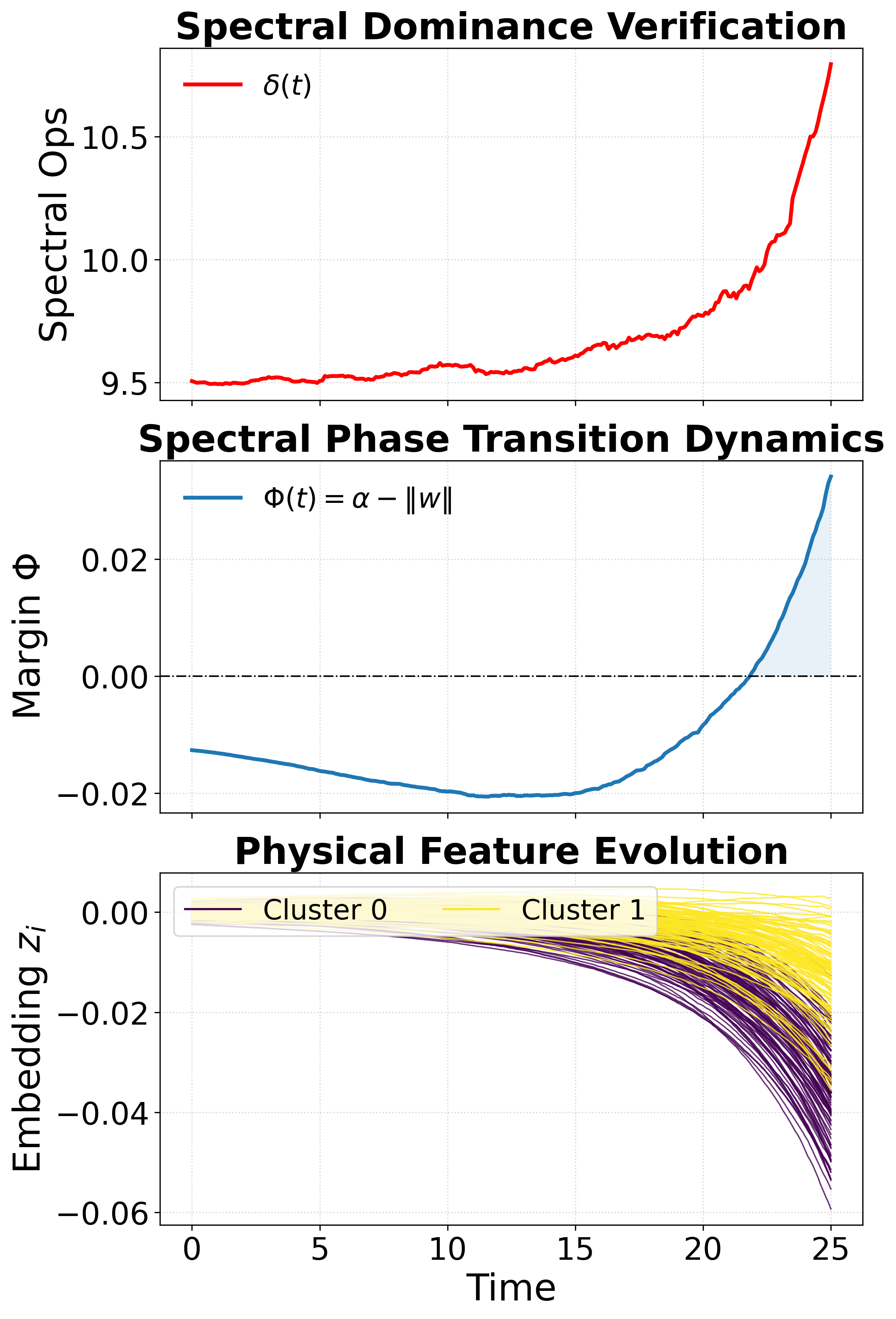}

        \vspace{-0.2cm}
        \centerline{\footnotesize (c) PDE}
        \label{fig:result_pde}
    \end{minipage}
    \begin{minipage}[b]{0.45\linewidth}
        \centering
        \includegraphics[width=\linewidth, height=10cm]{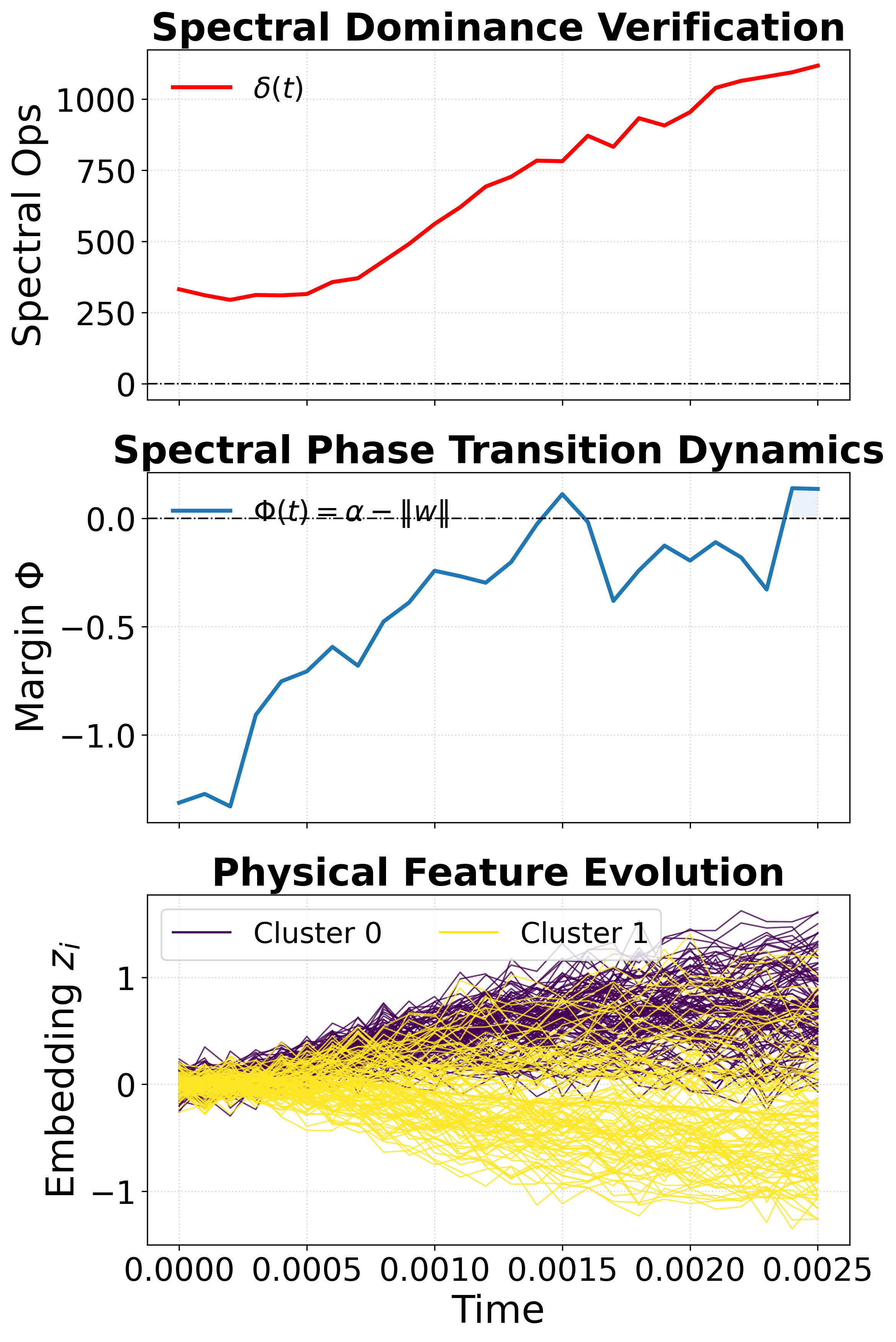}

        \vspace{-0.2cm}
        \centerline{\footnotesize (d) Text}
        \label{fig:result_text}
    \end{minipage}

    \vspace{-0.4em}
    
    \caption{Evolution of coercivity and linear separability across four domains. Each panel contains three vertical plots sharing the same x-axis, which represents time (iteration $\times$ learning rate). \textbf{Top:} The metric $\delta(t)$ quantifying coercivity. \textbf{Middle:} The metric $\Phi$ quantifying linear separability. \textbf{Bottom:} Particle evolution over time, where each particle is colored yellow or purple based on its cluster assignment. Across all domains, the results consistently show that $\delta(t)$ increases prior to the occurrence of linear separation, successfully confirming the main theorem.}
    \label{fig:main_results}
\end{figure}

\section{Conclusion and Future Work}

This paper investigates the underlying mechanisms that drive the success of contrastive learning. Crucially, we highlight that this success cannot be adequately explained by the contrastive loss function alone, which inherently contains infinitely many undesirable local minimizers. Instead, the ability to find meaningful representations depends fundamentally on the training dynamics induced by parameterized optimization. To understand this phenomenon, we analyzed the spectral quantities of the kernel matrix derived from the training dynamics to derive our main theoretical result: identifying the precise triggering condition, quantified by the coercivity metric $\delta(t)$, that must be met to guarantee linear separability.

The primary contribution of this work lies in successfully isolating and defining this exact condition. However, while our theoretical and empirical analyses confirm that reaching this condition dictates the onset of linear separation, our current framework does not capture how the neural network training inherently evolves the kernel to reach this state. Understanding the highly complex kernel dynamics that naturally drive the optimization trajectory toward this favorable triggering condition is a formidable challenge and forms the most critical direction for our future work. 

Beyond tracking the kernel evolution, other promising avenues for future research include developing a mean-field limit for these dynamics \cite{mei2018mean} and analyzing the infinite-width regime to obtain precise convergence guarantees and rates. Additionally, in many practical contrastive learning applications, the neural network is trained end-to-end, but the final projection head is discarded when extracting features. Prior works \cite{bordes2023guillotine,gui2023unraveling,wen2022mechanism} demonstrate that this practice can significantly improve downstream feature quality. Understanding how discarding this layer impacts the underlying training dynamics and linear separability remains another compelling open question for future investigation.

\appendix
\section{Appendix}

In the appendix, we present the proofs of that are missing in the main manuscript.

\subsection{Interpretation of VICReg and BYOL}
In this section, we further explore our observation that the NT-Xent loss becomes independent of the data distribution once the feature map $f$ is invariant. Consequently, the latent distribution corresponding to an invariant minimizer can be entirely unrelated to the input data. The following proposition demonstrates this.
\begin{proposition}\label{prop:illposed}
Suppose that $\mu \in \mathbb{P}(\mathbb{R}^d)$ is absolutely continuous and that the embedding map $f:\mathbb{R}^D \to \mathbb{R}^d$ is invariant under the distribution $\nu$ (satisfying \cref{eq:inv}). By applying a change of variables, we obtain the following reformulation of \cref{eq:cost-discrete}
\begin{align*}
\min_{f:\mathbb{R}^D\to\mathbb{R}^d} \mathbb{E}_{x\sim f_\#\mu} \log\left(1 + 2\mathbb{E}_{y\sim f_\#\mu} \Big[\mathds{1}_{x\neq y} \exp\bigl(\simcost_f(x,y)/\tau\bigr)\Big]\right) \\
=\min_{\rho \in \mathbb{P}(\mathbb{R}^d)} \mathbb{E}_{x\sim \rho} \log\left(1 + 2\mathbb{E}_{y\sim \rho} \Big[\mathds{1}_{x\neq y} \exp\bigl(\simcost(x,y)/\tau\bigr)\Big]\right),
\end{align*}
where $\simcost(x,y) = \simcost_{\id}(x,y) = \frac{x \cdot y}{\|x\|\|y\|}$.
\end{proposition}
The result in \Cref{prop:illposed} shows that minimizing the NT-Xent cost with respect to an embedding map, once the map is invariant, is equivalent to minimizing over the probability distribution in the latent space. This minimization is \emph{completely independent} of the input data distribution $\mu$.

We further show that two other popular methods for learning dataset invariance with deep learning models exhibit a similar phenomenon as in \Cref{prop:illposed}, namely, that the loss function itself becomes independent of the original data structure if the embedding map becomes invariant.

First, consider the VICReg~\cite{bardes2021vicreg} loss. Given a data distribution $\mu \in \mathbb{P}(\mathbb{R}^D)$ and a distribution for the perturbation functions $\nu$, VICReg minimizes
\begin{align*}
    \min_{f:\mathbb{R}^D\rightarrow\mathbb{R}^d} \mathbb{E}_{f,g\sim\nu} \mathbb{E}_{x_1,\cdots,x_n\sim\mu} &\frac{\lambda_1}{n} \sum^n_{i=1} \|f(f(x_i)) - f(g(x_i))\|^2\\
    &+ \lambda_2\Big(v(f(f(x_1)),\cdots,f(f(x_n))) + v(f(g(x_1)),\cdots,f(g(x_n)))\Big)\\
    &+ \lambda_3\Big(c(f(f(x_1)),\cdots,f(f(x_n))) + c(f(g(x_1)),\cdots,f(g(x_n)))\Big)
\end{align*}
where $\lambda_1$, $\lambda_2$, and $\lambda_3$ are hyperparameters. The first term ensures the invariance of $f$ with respect to perturbation functions from $\nu$, $v$ maintains the variance of each embedding dimension, and $c$ regularizes the covariance between pairs of embedded points towards zero. Suppose $f$ is an invariant embedding map such that $f(T(x)) = f(x)$ for all $T \sim \nu$. Then, the above minimization problem becomes
\begin{align*}
    &\min_{f:\mathbb{R}^D\rightarrow\mathbb{R}^d} \mathbb{E}_{f,g\sim\nu} \mathbb{E}_{x_1,\cdots,x_n\sim\mu} \lambda_2\Big(v(f(x_1),\cdots,f(x_n)) + v(f(x_1),\cdots,f(x_n))\Big)\\
    &\hspace{4cm}+ \lambda_3\Big(c(f(x_1),\cdots,f(x_n)) + c(f(x_1),\cdots,f(x_n))\Big)\\
    &=\min_{f:\mathbb{R}^D\rightarrow\mathbb{R}^d} \mathbb{E}_{y_1,\cdots,y_n\sim f_\#\mu} \lambda_2\Big(v(y_1,\cdots,y_n) + v(y_1,\cdots,y_n)\Big)\\
    &\hspace{4cm}+ \lambda_3\Big(c(y_1,\cdots,y_n) + c(y_1,\cdots,y_n)\Big)
\end{align*}
Similar to the result in \Cref{prop:illposed}, the invariance term vanishes. This minimization can now be expressed as a minimization over the embedded distribution
\begin{align*}
    \min_{\rho \in \mathbb{P}(\mathbb{R}^d)} \mathbb{E}_{y_1,\cdots,y_n\sim \rho} &\lambda_2\Big(v(y_1,\cdots,y_n) + v(y_1,\cdots,y_n)\Big)\\
    &+ \lambda_3\Big(c(y_1,\cdots,y_n) + c(y_1,\cdots,y_n)\Big)
\end{align*}
This shows that given an invariant map $f$, the minimization problem becomes completely independent of the input data $\mu$, thus demonstrating the same ill-posedness as the NT-Xent loss in \Cref{prop:illposed}.

Now, consider the loss function from BYOL~\cite{grill2020bootstrap}. Given a data distribution $\mu \in \mathbb{P}(\mathbb{R}^D)$ and a distribution for the perturbation functions $\nu$, the loss takes the form
\begin{align*}
    \min_{f,q} \mathbb{E}_{f,g\sim\nu} \mathbb{E}_{x\sim\mu} \|q(f(T(x))) - f(T'(x))\|^2
\end{align*}
where $q:\mathbb{R}^d \rightarrow \mathbb{R}^d$ is an auxiliary function designed to prevent $f$ from collapsing all points $x$ to a constant in $\mathbb{R}^d$. Similar to the previous case, if we assume an invariant map $f$, the above problem becomes
\begin{align*}
    &\min_{f,q} \mathbb{E}_{x\sim\mu} \|q(f(x)) - f(x)\|^2=\min_{f,q} \mathbb{E}_{y\sim f_\#\mu} \|q(y) - y\|^2
\end{align*}
where the second equality follows from a change of variables. Again, this minimization problem can be written with respect to the embedded distribution as
\begin{align*}
    \min_{\rho \in \mathbb{P}(\mathbb{R}^d), q} \mathbb{E}_{y\sim \rho} \|q(y) - y\|^2
\end{align*}
This again shows that once the invariant map is considered, the minimization problem becomes completely independent of the input data $\mu$, highlighting the ill-posedness of the cost function.

\subsection{Further analysis of the stationary points of NT-Xent loss }

Our first result provides the first order optimality conditions of the NT-Xent loss \cref{eq:cost-dis-new}. 
\begin{proposition}\label{thm:first-general-eta}
    The first optimality condition of the problem~\cref{eq:cost-dis-new} takes the form
    \begin{multline}\label{eq:psi-first}
        \int_{\mathcal{T}}\int_{\mathbb{R}^D} \bigg\langle
            \int_{\mathcal{T}}\int_{\mathbb{R}^D} \left(\frac{\Psi'(G_{T,T'}(f,x))}{\eta_f(T(x),T'(x))} + \frac{\Psi'(G_{T,T'}(f,y))}{\eta_f(T(y),T'(y))}\right) 
            \eta_f'(T(x),T'(y))\big(f(T(x)) - f(T'(y))\big)\\
        - \big(\Psi'(G_{T,T'}(f,x))\eta_f(T(x),T'(y)) + \Psi'(G_{T',T}(f,x))\eta_f(T'(x),T(y))\big) \\
         \frac{\eta_f'(T(x),T'(x))}{\eta_f^2(T(x),T'(x))} \big(f(T(x)) - f(T'(x))\big)
         d\mu(y)  d\nu(T'),
         h(T(x))   \bigg\rangle
          d\mu(x) d\nu(T) = 0
    \end{multline}
    for all $h$ such that $f+h\in\mathcal{C}$ where $\eta'_f(x,y) = \eta'(\|f(x)-f(y)\|^2/2)$, $\Psi(t) = \log(1+t)$ and $G_{T,T'}(f,x) = \frac{\mathbb{E}_{z \sim \mu} \eta_f(T(x),T'(z))}{\eta_f(T(x),T'(x))}$.

    If $f$ is invariant to the perturbation in $\nu$, then the gradient of $L$ takes the form
    \begin{equation}\label{eq:psi-first-inv}
        \nabla L(f)(x) = \int_{\mathbb{R}^D} \left({\Psi'(G_{\id,\id}(f,x))} + {\Psi'(G_{\id,\id}(f,y))}\right) \eta_f'(x,y)( f(x) - f(y)) d\mu(y).
    \end{equation}
\end{proposition}
Using the first optimality condition described in \Cref{thm:first-general-eta}, we can characterize the minimizer of the NT-Xent loss in \cref{eq:cost-dis-new}. The following theorem describes the possible local minimizers of \cref{eq:cost-dis-new}, considering the constraint set defined as $\mathcal{C}=\{f: \mathbb{R}^D \rightarrow \mathbb{S}^{d-1}\}$.

From the modified formulation~\cref{eq:cost-dis-new}, we can define a minimizer that minimizes the function $L(f)$ on a constraint set $\mathcal{C}=\{f:\mathbb{R}^D\rightarrow\mathbb{R}^d\}$. The following proposition provides insight into the minimizer of~\cref{eq:cost-dis-new}. The proof is provided in the appendix.
\begin{proposition}\label{cor:mini}
This proposition describes three different possible local minimizers of \cref{eq:cost-dis-new} that satisfy the Euler-Lagrange equation in \cref{eq:psi-first}.
    \begin{enumerate}
        \item Any map $f:\mathbb{R}^D\rightarrow \mathbb{R}^d$ that maps to a constant, such that 
        \[
            f(x) = c \in \mathbb{R}^d, \quad \forall x \in \mathcal{M}.
        \]
        \item In addition to the condition in \cref{eq:eta}, suppose the attraction and repulsion similarity functions $a:\mathbb{R}_{\geq0}\rightarrow \mathbb{R}$ and $r:\mathbb{R}_{\geq0}\rightarrow \mathbb{R}$ satisfy the following properties:
        \begin{enumerate}
            \item Each function is maximized at 0, where its value is 1.
            \item Each function satisfies $\lim_{t\rightarrow \infty} a(t) = 0$ and $\lim_{t\rightarrow \infty} t a'(t) = 0$.
        \end{enumerate}
        Let $f$ be a map invariant to $\mathcal{T}$. Consider a sequence of maps $\{f_k\}$ such that
        \begin{align*}
            f_k(x) = k f(x), \quad \forall x \in \mathcal{M}, \forall k \in \mathbb{N}.
        \end{align*}
        The limit $f_* = \lim_{k\rightarrow\infty} f_k$ satisfies the Euler-Lagrange equation~\cref{eq:psi-first}.
    \end{enumerate}
\end{proposition}

\begin{proof}[Proof of \Cref{cor:mini}]
    If $f$ is a constant function, it is trivial that it satisfies \cref{eq:psi-first}.

    Let us prove the second part of the proposition.
    From the Euler-Lagrange equation in \cref{eq:psi-first}, by plugging in $f_k$ and using the fact that $f$ is invariant to $\mathcal{T}$, the Euler-Lagrange equation can be simplified to
    \begin{equation*}
        \int_{\mathbb{R}^D} \left(\Psi'({G}({f_k},\id,x)) + \Psi'({G}({f_k},\id,y))\right)  r'_{f_k}(x,y)\langle {f_k}(x) - {f_k}(y), h(x) \rangle \, d\mu(y)
    \end{equation*}
    for any $h:\mathbb{R}^D \rightarrow \mathbb{R}^d$.
    Using the invariance of $f_k$, we have
    \begin{equation}\label{eq:Psi-EL}
        =    \int_{\mathbb{R}^D} \Bigl(\Psi'({G}({f_k},\id,x)) + \Psi'({G}({f_k},\id,y))\Bigr) \Bigl(k r'_{f_k}(x,y)\Bigr)\langle {f}(x) - {f}(y), h(x)\rangle \, d\mu(y).
    \end{equation}
    Furthermore, by the assumptions on the function $r$, 
    \begin{align*}
        k r'_{f_k}(x,y) &= k r'\left(\frac{k^2\|f(x)-f(y)\|^2}{2}\right) \rightarrow 0 ,\quad \text{as } k \rightarrow \infty\\
        \Psi'(G(f_k,\id,x)) &= \Psi'\left(\mathbb{E}_{z\sim \mu} r\left(\frac{k^2\|f(x)-f(z)\|^2}{2}\right)\right) \rightarrow \Psi'(0),\quad \text{as } k\rightarrow \infty.
    \end{align*}
    Thus, \cref{eq:Psi-EL} converges to $0$ as $k\rightarrow \infty$. This proves the theorem.    
\end{proof}

\subsection{Proof of \Cref{prop:w-gf}}

\begin{proof}
    The gradient of $\mathcal{L}$ is given by
    \begin{equation}\label{eq:grad-L}
        \nabla \mathcal{L}(w) = \frac{1}{n} \sum_{i=1}^n \sum_{k=1}^d \nabla_{z^k} L(f(w, x_i), x_i) \nabla_w f^k(w, x_i)
    \end{equation}
    where $\nabla_{z^k} L(f(w,x_i))$ is a gradient of $L$ with respect to $k$-th coordinate.
    
    For simplicity of notation, let us denote by
    \begin{align*}
        f^k_i = f^k(w,x_i),\quad L_i = L(f(w,x_i), x_i).
    \end{align*}
    Thus, \cref{eq:grad-L} can be rewritten as 
    \begin{align}\label{eq:nabla-mathcal-l}
        \nabla \mathcal{L}(w) = \frac{1}{n} \sum_{i=1}^n \sum_{k=1}^d \nabla_{z^k} L_i \nabla_w f^k_i
    \end{align}
    By the definition of the loss function in \cref{eq:loss-w}, $w(t)$ satisfies the gradient flow such that
    \begin{equation}\label{eq:w-gf}
        \dot{w}(t) = -\nabla \mathcal{L}(w).
    \end{equation}
    Thus, the solution of the above ODE converges to the local minimizer of $\mathcal{L}$ as $t$ grows.

    For each $i\in \llbracket  n \rrbracket$ and $k \in \llbracket d \rrbracket$, denote by
    \begin{align*}
        z^k_i(t) = f^k_i(t).
    \end{align*}
    Let us compute the time derivative of $z^k_i$. Using a chain rule, \cref{eq:nabla-mathcal-l} and \cref{eq:w-gf},
    \begin{equation}\label{eq:z-dot}
        \begin{aligned}
            \dot z^k_i(t) = \nabla_w f^k_i \cdot \dot w(t) &= - \nabla_w f^k_i \cdot \nabla \mathcal{L}(w) 
            = - \frac{1}{n}\sum^m_{j=1} \sum^d_{l=1} \nabla_w f^k_i \cdot \nabla_w f^l_j \, \nabla_{y^l} L_j.
        \end{aligned}
    \end{equation}
    Using \cref{eq:kernel-K}, \cref{eq:z-dot}, $\dot z_i(t)$ can be written as 
    \begin{align*}
        \dot z_i (t) = - \frac{1}{n} \sum^n_{j=1}K_{ij} \nabla L_{j}(t).
    \end{align*}
    This completes the proof.
\end{proof}

\subsection{Proof of \Cref{prop:nn-toy-map}}

\begin{proof}
Consider the gradient descent iterations in \cref{eq:z-no-nn-gd}. Suppose $f$ is invariant to the perturbation from $\nu$ at $b$-th iteration, that is we have 
$f(w^{(b)},f(x)) = f(w^{(b)},x)$ for all $x \sim \mu$ and $T \sim \nu$. We want to show that, given an invariant embedding map $f(w^{(b)},\cdot)$, it remains invariant after iteration $b$. From the gradient formulation of the loss function in \Cref{thm:first-general-eta}, we have 
\[
    \nabla L(f(w^{(b)},x),x) = - \int_{\mathbb{R}^D} \Psi'(x,y) (f(w^{(b)},x) - f(w^{(b)},y)) \, d\mu(y),
\]
where $\Psi'(x,y) = \Big(\Psi'(G(f(w^{(b)},\cdot),x)) + \Psi'(G(f(w^{(b)},\cdot),y))\Big) \eta_{f(w^{(b)},\cdot)}'(x,y)$. 

From the gradient descent formulation in \cref{eq:z-no-nn-gd}, we have
\[
    f(w^{(b+1)},f(x_i)) = f(w^{(b)},f(x_i)) - \sigma \nabla L(f(w^{(b)},f(x_i))),
\]
which gives
\[
    f(w^{(b+1)},f(x_i)) = f(w^{(b)},f(x_i)) + \sigma \int_{\mathbb{R}^D} \Psi'(x,y) (f(w^{(b)},f(x_i)) - f(w^{(b)},y)) \, d\mu(y),
\]
and since $f(w^{(b)},f(x_i)) = f(w^{(b)},x_i)$ by invariance, this simplifies to
\[
    f(w^{(b+1)},x_i) - \sigma \nabla L(f(w^{(b)},x_i),x_i) = f(w^{(b+1)},x_i),
\]
which shows that $f(w^{(b+1)},x_i)$ is invariant for all $i \in \lrbr{n}$. Therefore, the embedding map remains invariant throughout the optimization process. 

Now, consider the gradient descent iteration with a neural network in \cref{eq:z-nn-gd}. Suppose $f$ is invariant to perturbations from $\nu$ and satisfies
\begin{equation}\label{eq:prcond}
    \nabla_w f(w^{(b)},f(x)) = \nabla_w f(w^{(b)},x),\quad \forall x\sim \mu, f\sim\nu.
\end{equation}
Denote the kernel matrix function $K_{ij}$  given a perturbation function $T \sim \nu$ as
\[
    (K_{ij}(w^{(b)},f))^{kl} = (\nabla_w f^k(w^{(b)},f(x_i)))^\top(\nabla_w f^l(w^{(b)},f(x_j))).
\]
Then, we have
\begin{align*}
    f(w^{(b+1)},f(x_i)) &= f(w^{(b)},f(x_i)) - \frac{\sigma}{n} \sum_{j=1}^n K_{ij}(w^{(b)},f) \nabla L(f(w^{(b)},f(x_i)),x_i)\\
    &= f(w^{(b)},x_i) - \frac{\sigma}{n} \sum_{j=1}^n K_{ij}(w^{(b)},\id) \nabla L(f(w^{(b)},x_i),x_i)\\ 
    &= f(w^{(b+1)},x_i).
\end{align*}
Thus, if $f$ is invariant at the $b$-th iteration, it remains invariant. However, note that this result no longer holds if the condition in \cref{eq:prcond} fails, meaning that $f$ is not invariant for $b+1$-th iteration. This completes the proof.
\end{proof}

\subsection{Explicit Formula for the Neural Network Kernel Matrix}

In this section, we derive an explicit formula for the kernel matrix to better understand how the neural network kernel influences the gradient descent dynamics of $f(w(t), x)$ for data points $x \in X$. This formula provides insight into how the kernel governs the evolution of features during training.

While the main paper focuses on a simplified case where the embedding map outputs real-valued features (i.e., $f: \mathbb{R}^{MD + D} \rightarrow \mathbb{R}$), here we consider a more general setting where the embedding map outputs vector-valued features in $\mathbb{R}^d$. This generalization allows us to derive a broader formula for the kernel matrix. Specifically, we consider a one-hidden-layer fully connected neural network of the form
$f: \mathbb{R}^{MDd + D} \rightarrow \mathbb{R}^d.$

\begin{equation}\label{eq:f-nn}
f(w(t), x)= A^\top \sigma(w(t)x),
\end{equation}
where $x\in\mathbb{R}^D$ is a data sample and $A \in \mathbb{R}^{Md \times d}$ is a constant matrix defined as
\begin{equation}\label{eq:def-A}
A = \frac{1}{\sqrt{MD}} \begin{bmatrix}
\mathbf{1}_{M \times 1} & \mathbf{0}_{M \times 1} & \cdots & \mathbf{0}_{M \times 1} \\
\mathbf{0}_{M \times 1} & \mathbf{1}_{M \times 1} & \cdots & \mathbf{0}_{M \times 1} \\
\vdots & \vdots & \ddots & \vdots \\
\mathbf{0}_{M \times 1} & \mathbf{0}_{M \times 1} & \cdots & \mathbf{1}_{M \times 1}
\end{bmatrix} \in \mathbb{R}^{Md \times d},
\end{equation}
where $\mathbf{1}_{M \times 1}$ and $\mathbf{0}_{M \times 1}$ represent the $M$-dimensional vectors of ones and zeros, respectively. Additionally, $w(t) = (b^k_p(t))_{k \in \lrbr{D}, p \in \lrbr{Md}} \in \mathbb{R}^{Md \times D}$ is the weight matrix, and $\sigma$ is a differentiable activation function applied element-wise. 

Note that $A$ acts as an averaging matrix that, when multiplied by the $(Md)$-dimensional vector $\sigma(w(t)x)$, produces a $d$-dimensional vector. Furthermore, we assume that the parameters of $w(t)$ are uniformly bounded, such that there exists a constant $C$ with $|b^k_p(t)| < C$ for all $t \geq 0$, $k$, and $p$.

\begin{proposition}\label{prop:explicit-kernel}
    Given the description of the embedding map in \cref{eq:f-nn}, the kernel matrix $K^{kl}_{ij}$ defined in \cref{eq:kernel-K} can be explicitly written as
\begin{align}\label{eq:kernel-simple}
        \begin{aligned}
            K^{kl}_{ij}
                =\frac{\mathbf{1}_{k=l}}{MD} x_i^{\top} x_j \sum^{kM}_{p=(k-1)M+1} \sigma'(b_px_i) \sigma'(b_px_j).
        \end{aligned}
\end{align}
where $\mathbf{1}_{k=l}$ is an indicator function that equals $1$ if $k=l$ and $0$ otherwise.

\end{proposition}

\begin{proof}
    We describe the matrices $A \in \mathbb{R}^{Md \times d}$ and $B \in \mathbb{R}^{Md \times D}$ as follows
\begin{align*}
    A &= \begin{bmatrix}
        | & | & \cdots & | \\
        a^1 & a^2 & \cdots & a^d \\
        | & | & \cdots & |
    \end{bmatrix}
    =
    \begin{bmatrix}
        - & a_1 & - \\
        & \vdots & \\
        - & a_{Md} & -
    \end{bmatrix}
    =
    \begin{bmatrix}
        a_1^1 & \cdots & a_1^d \\
        \vdots & \cdots & \vdots \\
        a_{Md}^1 & \cdots & a_{Md}^d
    \end{bmatrix}, \\
    B(t) &= \begin{bmatrix}
        | & | & \cdots & | \\
        b^1(t) & b^2(t) & \cdots & b^D(t) \\
        | & | & \cdots & |
    \end{bmatrix}
    =
    \begin{bmatrix}
        - & b_1(t) & - \\
        & \vdots & \\
        - & b_{Md}(t) & -
    \end{bmatrix}
    =
    \begin{bmatrix}
        b_1^1(t) & \cdots & b_1^D(t) \\
        \vdots & \cdots & \vdots \\
        b_{Md}^1(t) & \cdots & b_{Md}^D(t)
    \end{bmatrix}.
\end{align*}
In this notation, $a^k$ and $b^k$ are $Md$-dimensional column vectors, $a_p$ and $b_p$ are $d$- and $D$-dimensional row vectors, and $a^k_p$ and $b^k_p$ are scalars.

    We can write $f^k$ with respect to $a^k_i$ and $b^k_i$.
    \begin{align*}
        f^k(B,x) = (a^{k})^{\top} \sigma(Bx) = \sum^{Md}_{i=1} a^{k}_i \sigma\Big( b_i x\Big)
        &= \frac{1}{\sqrt{MD}} \sum^{kM}_{i=(k-1)M+1} \sigma\Big( b_i x\Big)
    \end{align*}
    where the last equality uses the definition of a matrix $A$ in \cref{eq:def-A}.
    By differentiating with respect to $b^l_i$, we can derive explicit forms for the gradient of $f^k$ with respect to a weight matrix $B$.
    \begin{align*}
        \nabla_w f^k(B,x) &= \Bigl(a^k \odot \sigma'(Bx)\Bigr) x^{\top}\\
        &= \begin{bmatrix}
            a_1^k \sigma'(b_1x) x^1 & \cdots & a^k_1 \sigma'(b_1x) x^D\\ 
            \vdots & \ddots & \vdots\\
            a^k_M \sigma'(b_Mx) x^1 & \cdots & a^k_M \sigma'(b_Mx) x^D\\ 
        \end{bmatrix} \\
        &= \frac{1}{\sqrt{MD}} \begin{bmatrix}
            0 & \cdots & 0\\
            \vdots & \ddots & \vdots\\
            0 & \cdots & 0\\
             \sigma'(b_1x) x^1 & \cdots & \sigma'(b_1x) x^D\\ 
            \vdots & \ddots & \vdots\\
            \sigma'(b_Mx) x^1 & \cdots & \sigma'(b_Mx) x^D\\ 
            0 & \cdots & 0\\
            \vdots & \ddots & \vdots\\
            0 & \cdots & 0\\
        \end{bmatrix} \in \mathbb{R}^{Md \times D}
    \end{align*}
    where the row index of nonzero entries ranges from $(k-1)M+1$ to $kM$.

    Define an inner product such that for $h\in\mathbb{R}^{Md\times D}$
    \begin{align*}
        \langle \nabla_w f^k(B,x), h \rangle, \quad k \in\lrbr{D}.
    \end{align*}
    Now we are ready to show the explicit formula of the inner product $\langle \nabla_w f^k, \nabla_w f^l\rangle$.
    \begin{align*}
            \langle \nabla_w f^k(B,x_i), \nabla_w f^l(B,x_j) \rangle
        &= \frac{\mathbf{1}_{k=l}}{MD}  (x_i^{\top} x_j )\sum^{kM}_{p=(k-1)M+1} \sigma'(b_px_i) \sigma'(b_px_j)
    \end{align*}
    where $\mathbf{1}_{k=l}$ is an indicator function that equals $1$ if $k=l$ and $0$ otherwise.
    Therefore, the kernel matrix takes the form
    \begin{align*}
            \begin{aligned}
                (K^{kl})_{ij}
                =\frac{\mathbf{1}_{k=l}}{MD} (x_i^{\top} x_j )\sum^{kM}_{p=(k-1)M+1} \sigma'(b_px_i) \sigma'(b_px_j).
            \end{aligned}
    \end{align*}
\end{proof}

{
From \Cref{prop:explicit-kernel}, as done in NTK paper \cite{jacot2018neural}, one can consider how the kernel converges as the width of the neural network approaches infinity, i.e., as $M \to \infty$ in \cref{eq:f-nn}. The following proposition shows the formulation of the limiting kernel in the infinite-width neural network.
\begin{proposition}\label{prop:kernel-ntk}
    Suppose the weight matrix $B$ satisfies that each row vector $b_i$, for $i \in \{1, \dots, Md\}$, consists of independent and identically distributed random variables in $\mathbb{R}^D$ with a Gaussian distribution. Also, suppose the activation function is $\sigma(x) = x_+ = \max\{x, 0\}$. Then, as $M \to \infty$, the kernel matrix converges to $K^\infty\in\mathbb{R}^{d\times d}$, where
    \[
        K_{ij}^{\infty} = \frac{x_i^\top x_j}{D}
        \left[ \frac{1}{2} - \frac{1}{2\pi} \arccos\left( \frac{x_i^\top x_j}{\|x_i\| \|x_j\|} \right) \right] \pmb{I}_{d\times d},\quad \text{$\pmb{I}_{d\times d}$ is an identity matrix}.
    \]
\end{proposition}
}

\subsection{Proof of \Cref{prop:contrastive_taylor}}

\begin{proof}[Proof of \Cref{prop:contrastive_taylor}]
We decouple the continuous loss into a global repulsive term $T_{\mathrm{push}}(z)$ and a local attractive term $T_{\mathrm{pull}}(z)$. Expanding the log-sum-exp functional in $T_{\mathrm{push}}(z)$ near the origin gives
\[
T_{\mathrm{push}}(z) = - \frac{1}{\tau n m } z^\top (I - P_{\mathrm{global}})z + \mathcal{O}\left(\frac{\|z\|^4}{\tau^2(nm)^2}\right).
\]
Expanding the quadratic forms in $T_{\mathrm{pull}}(z)$ isolates the local attraction
\[
T_{\mathrm{pull}}(z) = \frac{1}{\tau nm} z^\top (I - P_{\mathrm{mean}}) z.
\]
Combining these yields
\[
\widetilde{\mathcal L}(z) = - \frac{1}{\tau nm} z^\top (P_{\mathrm{class}} - P_{\mathrm{global}}) z - \frac{1}{\tau nm} z^\top (P_{\mathrm{mean}} - P_{\mathrm{class}}) z + \mathcal{O}\left(\frac{\|z\|^4}{\tau^2(nm)^2}\right).
\]
Because $(P_{\mathrm{mean}} - P_{\mathrm{class}})$ is an orthogonal projection, the residual quadratic form is exactly $\|(P_{\mathrm{mean}} - P_{\mathrm{class}}) z\|^2$. Using the spectral gap bound derived from \Cref{assum:augmentation_graph}, this residual is bounded strictly by $\frac{\eta}{\lambda_2} \|(I - P_{\mathrm{mean}}) z\|^2$.

Taking the gradient with respect to $z$, the leading linear term evaluates directly to $- \frac{1}{\tau} H z$. The gradient of the residual form evaluates to $- \frac{2}{\tau nm} (P_{\mathrm{mean}} - P_{\mathrm{class}}) z$, whose Euclidean norm is bounded by $\frac{2}{\tau nm} \sqrt{\frac{\eta}{\lambda_2}} \|(I - P_{\mathrm{mean}}) z\|$. Absorbing constants into the asymptotic notation yields the final result.
\end{proof}

\subsection{Proof of \Cref{prop:local_ntk_window}}

\begin{proof}[Proof of \Cref{prop:local_ntk_window}]

Let $K(t)$ denote the kernel matrix from the neural network function. Under gradient flow
of the network parameters, along with the assumption (iii), the feature dynamics satisfy the NTK identity
\[
\dot z(t)
=
- K(t)\,\nabla_z \widetilde{\mathcal L}(z(t))
=
\frac{1}{\tau} K(t) H z(t)
+
R(t),
\]
where the remainder satisfies
\[
\|R(t)\|
\leq C_R C_K \left(\frac{\nu\|z\|}{nm\tau} + \frac{\|z\|^3}{(nm)^2\tau^2}\right)
\]
where we set $\nu = \sqrt{\eta/\lambda_2}$.
We now derive a differential inequality for $\|z(t)\|$.
Using

\[
\frac{d}{dt}\|z(t)\|
=
\frac{\langle z(t), \dot z(t) \rangle}{\|z(t)\|},
\]
and substituting the evolution equation gives

\[
\frac{d}{dt}\|z(t)\|
\leq
\frac{1}{\tau}\|K(t)\|_{\mathrm{op}}\,\|H\|_{\mathrm{op}}\,\|z(t)\|
+
\|R(t)\|.
\]
Using $\|K(t)\|\leq C_K$ and the operator norm $\|H\|_{\mathrm{op}} \leq \frac{2}{nm},$
we obtain
\begin{align*}
    \frac{d}{dt}\|z(t)\|
&\leq
\frac{2 C_K}{nm \tau}\|z(t)\|
+
C_R C_K \left(\frac{\nu\|z(t)\|}{nm\tau} + \frac{\|z(t)\|^3}{(nm)^2\tau^2}\right)\\
&= \left(\frac{C_K(2 + C_RC_K \nu)}{nm \tau} \right)\|z(t)\|
+
\frac{C_R C_K\|z(t)\|^3}{(nm)^2\tau^2}
\end{align*}
Solving the differential inequality yields
\begin{align*}
    \|z(t)\| 
    \leq \frac{e^{at} \|z(0)\|}{\left(1 - \frac{C_R}{nm\tau(2+C_R \nu)}(e^{2at}-1)\|z(0)\|^2\right)^{1/2}}
\end{align*}
where $a=\frac{C_K(2 + C_R\nu)}{nm\tau}$.
By setting
\begin{align*}
    T_\varepsilon
:=
\frac{nm\tau}{2C_K(2 + C_R\nu)}\log \left(\frac{1}{\varepsilon}\right),
\end{align*}
it follows that for $t \leq T_\varepsilon$, we have
\begin{align*}
    \|z(t)\| \leq \sqrt{\varepsilon}.
\end{align*}
This completes the proof.
\end{proof}

\subsection{Proof of \Cref{prop:mean_field_wgf}}
\begin{proof}[Proof of \Cref{prop:mean_field_wgf}]
Under the assumption that $w(t)$ is bounded and the feature Jacobian is uniformly bounded and Lipschitz, the discrete feature tangent kernel converges uniformly to the continuous integral operator $\mathcal{K}_t$. For any vector field $u \in L^2(\rho_t; \mathbb{R}^d)$, this operator maps purely over the spatial domain
\[(\mathcal{K}_t u)(z) = \int_{\mathbb{R}^d} K_t(z, z') u(z') \,\mathrm{d}\rho_t(z').\]
The velocity field driving the density in \cref{eq:eulerian_continuity} is exactly given by $v_t[\rho_t](z) = (\mathcal{K}_t F[\rho_t])(z)$, where $F[\rho_t](z) = -\nabla_z \frac{\delta \mathcal{J}}{\delta \rho_t}(z)$ is the intrinsic driving force.

To show that the exact velocity field $v_t[\rho_t](z)$ is globally spatially Lipschitz, we first find the spatial gradient of $v_t$,
\begin{align*}
    \|\nabla_z v_t[\rho_t](z)\| \leq \int_{\R^d} \nabla_z K_t(z,z') F[\rho_t](z') d\rho_t(z').
\end{align*}
Applying the Cauchy-Schwarz inequality for integrals to decouple the Jacobian norm and the force norm
\[\|\nabla_z v_t[\rho_t](z)\| \leq \left( \int_{\mathbb{R}^d} \|\nabla_z K_t(z,z')\|^2 \,\mathrm{d}\rho_t(z') \right)^{1/2} \left( \int_{\mathbb{R}^d} \|F[\rho_t](z')\|^2 \,\mathrm{d}\rho_t(z') \right)^{1/2}.\]
By the assumptions, the feature Jacobian is uniformly bounded and spatially Lipschitz. Because $\rho_t$ is a probability measure, $\int \mathrm{d}\rho_t(z') = 1$, it follows that $\|\nabla_z v_t[\rho_t](z)\|$ is uniformly bounded.
By the assumption, the initial feature distribution is bounded such that $\|Z_0\|_{L^2(\rho_0)} \leq \varepsilon$. Because the continuous Lagrangian flow $Z_t$ is continuous in time, there strictly exists a time window $t \in [0, T]$ such that the features remain within the expanded local neighborhood $\|Z_t\|_{L^2(\rho_0)} \leq \sqrt{\varepsilon}$. 

Within this active time window $[0, T]$, analogous to the discrete expansion in Proposition \ref{prop:contrastive_taylor}, the intrinsic driving force of the continuous population functional $\mathcal{J}[\rho_t]$ admits a rigorous Taylor expansion. Expanding the spatial gradient of the first variation around the origin yields
\[F[\rho_t](z) = \frac{1}{\tau} (\mathcal{H} z) + \mathcal{R}_t'(z),\]
where $\mathcal{H} = 2(\mathcal{P}_{\mathrm{class}} - \mathcal{P}_{\mathrm{global}})$ is the linear operator defined by the continuous projections, and $\mathcal{R}_t'$ encapsulates the higher-order residual terms and the bounded augmentation variance. Within this linearization window, $\mathcal{H}$ is strictly bounded. The operators $\mathcal{P}_{\mathrm{class}}$ and $\mathcal{P}_{\mathrm{global}}$ operate as orthogonal projections in $L^2(\rho_t)$. Since orthogonal projections have an operator norm of at most $1$, we have $\|\mathcal{H}\|_{\mathrm{op}} \leq 2(\|\mathcal{P}_{\mathrm{class}}\|_{\mathrm{op}} + \|\mathcal{P}_{\mathrm{global}}\|_{\mathrm{op}}) \leq 4$. Thus, $\mathcal{H}$ is a bounded linear operator. Furthermore, within this restricted spatial domain, the residual $\mathcal{R}_t'(z)$ evaluates over a compact neighborhood and is therefore bounded. Because both terms composing $F[\rho_t]$ are bounded, it follows that $F[\rho_t]$ is strictly bounded in $L^2(\rho_t)$. Setting the kernel constant $C_K$, we explicitly establish the globally bounded Lipschitz constant $L$
\[\sup_z \|\nabla_z v_t[\rho_t](z)\| \leq C_K \|F[\rho_t]\|_{L^2(\rho_t)} := L < \infty.\]

To establish measure-Lipschitz continuity, we leverage this same structural linearization of the contrastive loss. Because the projection operators $\mathcal{P}_{\mathrm{class}}$ and $\mathcal{P}_{\mathrm{global}}$ represent bounded linear integrations over the measure $\rho_t$, the dominant linearized force $z \mapsto \frac{1}{\tau}\mathcal{H}z$ is intrinsically measure-Lipschitz. For any two measures $\rho, \nu \in \mathcal{P}_2(\mathbb{R}^d)$ concentrated in this local neighborhood, taking the difference of their corresponding linearized forces and passing it through the bounded kernel operator $\mathcal{K}_t$ (where $\|\mathcal{K}_t\|_{\mathrm{op}} \leq C_K$) yields a strict measure-Lipschitz bound
\[\|v_t[\rho](z) - v_t[\nu](z)\| \leq M W_2(\rho, \nu),\]
where $M$ is a constant depending on $\tau$, $C_K$, and the local bounding domain.

With the spatial and measure Lipschitz bounds ($L$ and $M$) established, let $\gamma_0 \in \Gamma(\rho_0, \rho^{(n,m)}_0)$ be the initial optimal transport coupling. Let $Z_t(x)$ and $\tilde{Z}_t(y)$ denote the characteristic Lagrangian trajectories driven by the continuous velocity $v_t[\rho_t]$ and the empirical velocity $v_t[\rho^{(n,m)}_t]$, respectively. Differentiating the squared Wasserstein distance bounded by this coupling yields
\[\frac{\mathrm{d}}{\mathrm{d}t} \int \|Z_t(x) - \tilde{Z}_t(y)\|^2 \,\mathrm{d}\gamma_0 \leq 2 \int \|Z_t(x) - \tilde{Z}_t(y)\| \cdot \|v_t[\rho_t](Z_t(x)) - v_t[\rho^{(n,m)}_t](\tilde{Z}_t(y))\| \,\mathrm{d}\gamma_0.\]
Applying Young's inequality $(2ab \leq a^2 + b^2)$, splitting the velocity difference via the triangle inequality, and bounding each term using the respective Lipschitz constants $L$ and $M$ results in
\[\frac{\mathrm{d}}{\mathrm{d}t} \int \|Z_t(x) - \tilde{Z}_t(y)\|^2 \,\mathrm{d}\gamma_0 \leq (1 + 2L^2 + 2M^2) \int \|Z_t(x) - \tilde{Z}_t(y)\|^2 \,\mathrm{d}\gamma_0.\]
Setting $C = 1 + 2L^2 + 2M^2$ and applying Grönwall's inequality yields the stability bound
\[W_2^2(\rho_t, \rho^{(n,m)}_t) \leq \int \|Z_t(x) - \tilde{Z}_t(y)\|^2 \,\mathrm{d}\gamma_0 \leq e^{Ct} W_2^2(\rho_0, \rho^{(n,m)}_0).\]
Because the initial empirical density weakly converges, $W_2(\rho_0, \rho^{(n,m)}_0) \to 0$ as $n,m \to \infty$. Consequently, $W_2(\rho_t, \rho^{(n,m)}_t) \to 0$ for all $t \geq 0$, proving the limiting measure $\rho_t$ is uniquely governed by the continuous PDE in \cref{eq:eulerian_continuity}.

Finally, for $t \in [0, T]$, substituting the functional linearization $F[\rho_t] = \frac{1}{\tau} \mathcal{H} Z_t + \mathcal{R}_t'$ back into the characteristic velocity equation natively recovers the rigorous linearized Lagrangian descent
\[\dot{Z}_t(z_0) = (\mathcal{K}_t F[\rho_t])(Z_t(z_0)) = \frac{1}{\tau} (\mathcal{K}_t \mathcal{H} Z_t)(z_0) + \mathcal{R}_t(z_0).\]
Since the integral operator is bounded by $\|\mathcal{K}_t\|_{\mathrm{op}} \leq C_K$, the continuous residual $\mathcal{R}_t = \mathcal{K}_t \mathcal{R}_t'$ strictly scales identically to the discrete bound established in Proposition \ref{prop:contrastive_taylor}, completing the proof.
\end{proof}

\subsection{Proof of \Cref{thm:latent_wgf_separation}}\label{app:continuum_proof}
\begin{proof}[Proof of \Cref{thm:latent_wgf_separation}]
Consider the following orthogonal decomposition of the continuous feature map $Z_t \in L^2(\rho_0)$
\begin{equation}\label{thmpr1_cont}
    Z_t = \alpha(t) v_c + W_t + c(t)\mathbf{1},
\end{equation}
where $\alpha(t) = \langle v_c, Z_t \rangle_{L^2(\rho_0)}$ is the macroscopic separation amplitude, $W_t = \mathcal{P}_W Z_t$ is the zero-mean intra-cluster dispersion, and $c(t)\mathbf{1} = \mathcal{P}_{\mathrm{global}} Z_t$ is the uniform global shift. By definition of the orthogonal projection operators in $L^2(\rho_0)$, the functions $v_c$, $W_t$, and $\mathbf{1}$ are mutually orthogonal.

Next, we determine how the linear operator $\mathcal{H} = 2(\mathcal{P}_{\mathrm{class}} - \mathcal{P}_{\mathrm{global}})$ affects each component. Because $v_c$ lies strictly in the class-discriminative subspace, $\mathcal{P}_{\mathrm{class}} v_c = v_c$. Consequently, the operator evaluates to $\mathcal{H} v_c = 2(\mathcal{P}_{\mathrm{class}} - \mathcal{P}_{\mathrm{global}}) v_c = 2v_c$. The operator completely annihilates the global mean, yielding $\mathcal{H}(c(t)\mathbf{1}) = 0$. By definition, $W_t \in \operatorname{Range}(\mathcal{I} - \mathcal{P}_{\mathrm{class}})$, which implies $\mathcal{P}_{\mathrm{class}} W_t = 0$. Because $W_t$ inherently possesses a zero global mean, $\mathcal{P}_{\mathrm{global}} W_t = 0$. Substituting these identities perfectly annihilates the cross-coupling, yielding $\mathcal{H} W_t = 0$.

Applying $\mathcal{H}$ to \cref{thmpr1_cont} and substituting it into the linearized continuous gradient flow equation $\dot{Z}_t = \frac{1}{\tau} \mathcal{K}_t \mathcal{H} Z_t + \mathcal{R}_t$, we obtain
\begin{equation}
\dot{Z}_t = \frac{2}{\tau} \alpha(t) \mathcal{K}_t v_c + \mathcal{R}_t.
\end{equation}
We extract the time evolution of the macroscopic separation amplitude $\alpha(t)$ by taking the $L^2(\rho_0)$ inner product with the fixed linear test function $v_c$
\begin{equation*}
\dot{\alpha}(t) = \langle v_c, \dot{Z}_t \rangle_{L^2(\rho_0)} = \frac{2}{\tau} \alpha(t) \langle v_c, \mathcal{K}_t v_c \rangle_{L^2(\rho_0)} + \langle v_c, \mathcal{R}_t \rangle_{L^2(\rho_0)}.
\end{equation*}
By the spectral alignment bound $\delta$ and the residual bound $\|\mathcal{R}_t\|_{L^2(\rho_0)} \leq \tilde{R}$, 
\begin{equation*}
\dot{\alpha}(t) \geq \frac{2\delta}{\tau}  \alpha(t) - \tilde{R}.
\end{equation*}
Solving the differential inequality, we obtain
\begin{align}\label{thmpralpha_cont}
    \alpha(t) 
    \geq 
    e^{\frac{2\delta}{\tau}(t-t_*)}\left(\alpha(t_*) - \gamma\right) + \gamma,
\end{align}
where $\gamma = \frac{\tau \tilde{R}}{2\delta}$ is defined in \cref{eq:continuous_initial_condition}. By \cref{eq:continuous_initial_condition}, $\alpha(t_*) - \gamma > 0$, and thus $\alpha(t)$ increases exponentially.

Next, by orthogonality in \cref{thmpr1_cont},
\[
\|Z_t\|_{L^2(\rho_0)}^2 = \alpha(t)^2 + \|W_t\|_{L^2(\rho_0)}^2 + c(t)^2.
\]
By the continuous analogue of Proposition \ref{prop:local_ntk_window}, we have $\|Z_t\|_{L^2(\rho_0)} \le \sqrt{\varepsilon}$ on $[t_*,t_*+T_\varepsilon]$. Hence $\|W_t\|_{L^2(\rho_0)}^2 \le \varepsilon - \alpha(t)^2$.
Therefore,
\[
\Phi(t)=\alpha(t)-\|W_t\|_{L^2(\rho_0)}
\ge
\alpha(t)-\sqrt{\varepsilon-\alpha(t)^2}.
\]
It follows that $\Phi(t)>0$ whenever $\alpha(t) > \sqrt{\varepsilon/2}$. Using \cref{thmpralpha_cont}, this condition is satisfied for $t \geq T_{\mathrm{sep}}$ where 
\[
T_{\mathrm{sep}}
=
t_*+
\frac{\tau}{2\delta}
\log\!\left(
\frac{\sqrt{\varepsilon/2}-\gamma}{\alpha(t_*)-\gamma}
\right).
\]
Finally, to ensure $T_{\mathrm{sep}}\le t_*+T_\varepsilon$, it suffices that
\[
\frac{\tau}{2\delta}
\log\!\left(
\frac{\sqrt{\varepsilon/2}-\gamma}{\alpha(t_*)-\gamma}
\right)
\le
T_\varepsilon,
\]
or equivalently,
\[
\frac{\sqrt{\varepsilon/2}-\gamma}{\alpha(t_*)-\gamma}
\le
\exp\!\left(\frac{2\delta}{\tau}T_\varepsilon\right).
\]
Under this condition, we obtain $T_{\mathrm{sep}}\in[t_*,t_*+T_\varepsilon]$, and for all $t\in[T_{\mathrm{sep}},\,t_*+T_\varepsilon]$, $\Phi(t)>0.$ This completes the proof.
\end{proof}

\section{Details of Numerical Experiments and Extra Results}

In this section, we provide the comprehensive details for the numerical experiments introduced in \Cref{sec:main-numeric}. Specifically, we elaborate on the dataset generation processes, neural network architectures, training hyperparameters, and the precise formulations of the data augmentations employed across all four experimental domains: the synthetic 2D concentric rings, CIFAR-10 natural images, the continuous PDE heat equation, and the IMDB text classification task. Furthermore, we present supplementary empirical results and visualizations that provide additional support for our theoretical findings regarding the emergence of the spectral gap and its temporal precedence to linear separability.

\subsection{Synthetic Dataset: Concentric Rings (Donut)}

To isolate and observe the training dynamics in a purely geometric, low-dimensional setting, we utilize a 2D synthetic dataset consisting of two concentric, donut-shaped clusters. This environment provides a tractable baseline where the underlying topological invariants are perfectly known, allowing us to explicitly verify the theoretical relationship between spectral gap formation and linear separability.

Let the input space be $\mathbb{R}^2$. The dataset is constructed by uniformly sampling points from two distinct annular regions. The first cluster (Label 0) is defined by an inner radius $r_1$, while the second cluster (Label 1) is defined by $r_2$ such that $0 < r_1 < r_2$. Consequently, the cluster identity of any point $x \in \mathbb{R}^2$ is entirely determined by its radial norm $\|x\|_2$.

Because the semantic label depends only on the radius, the angular coordinate acts as a continuous nuisance variable. To formulate the contrastive learning task, we define our augmentation family as the special orthogonal group $SO(2)$. For a given sample $x$, an augmented view is generated via a random rotation:
$$
x \mapsto R_\theta x, \quad \text{where } R_\theta = \begin{bmatrix} \cos\theta & -\sin\theta \\ \sin\theta & \cos\theta \end{bmatrix}, \quad \theta \sim \mathcal{U}(0, 2\pi).
$$
This spatial transformation perfectly preserves the radial distance while altering the angle. The encoder $f : \mathbb{R}^2 \to \mathbb{R}$ is parameterized by a simple fully connected neural network (MLP). By enforcing invariance to $SO(2)$ transformations through the contrastive loss, the network is forced to collapse the angular variance, eventually mapping the non-linearly separable concentric rings into linearly separable clusters in the 1D embedding space.

\subsection{Image Contrastive Learning: CIFAR-10}

We use the CIFAR-10 dataset for our natural image experiments. To create a clear binary classification task, we select two highly distinct classes: "Dog" (Class 5) and "Ship" (Class 8). We randomly sample $50$ images from each class. This creates a small, challenging training set with a total of $N=100$ samples. The input space consists of RGB images $x \in \mathbb{R}^{3 \times 32 \times 32}$.

The encoder $f$ is a Convolutional Neural Network (CNN). This network processes the high-dimensional images and outputs a 1-dimensional scalar embedding. The architecture is straightforward
\begin{itemize}
    \item A 2D Convolutional layer (3 to 16 channels, $5 \times 5$ kernel) followed by a LeakyReLU activation and $2 \times 2$ Max Pooling.
    \item A second 2D Convolutional layer (16 to 32 channels, $5 \times 5$ kernel) followed by a LeakyReLU activation and $2 \times 2$ Max Pooling.
    \item A flattening step.
    \item A fully connected linear layer mapping the 2048 features down to 32, followed by a LeakyReLU activation.
    \item A final linear layer mapping the 32 features down to a single scalar value $f(x) \in \mathbb{R}$.
\end{itemize}
To create the positive pairs for contrastive learning, we apply a heavy mix of standard visual augmentations. For every training step, each image goes through a random pipeline. Specifically, we use
\begin{itemize}
    \item We randomly crop a section of the image (between 8\% and 100\% of the original area) and resize it back to $32 \times 32$ pixels using bilinear interpolation.
    \item There is a 50\% chance to flip the image horizontally. We also apply a fixed 20-degree rotation to all augmented views.
    \item There is a 50\% chance to apply color jitter. This randomly shifts the brightness, contrast, saturation, and hue. Finally, there is a 20\% chance to convert the image entirely to grayscale.
\end{itemize}

We train the model using a Contrastive Loss function with a temperature parameter of $\tau=0.1$. We optimize the network using Stochastic Gradient Descent (SGD) with a learning rate of $0.01$.

\subsection{PDE Contrastive Learning Experiments with Symmetry}

We build datasets using the heat equation to study a continuous physical setting. We look at scalar fields on the periodic domain $\Omega = [0,1)^2$. We place these fields on a uniform grid of size $40 \times 40$. Each sample is a vector in $\mathbb{R}^{1600}$, which we get by flattening the grid.

We build the starting data using periodic Fourier modes. We create two distinct clusters with 100 samples each, giving a total of $N=200$ samples. Label 0 uses low-frequency modes to make smooth, large shapes. Label 1 uses high-frequency modes to make tight, repeating patterns. We apply a forward heat evolution step to all initial fields. This process smooths out the fields and heavily reduces the high-frequency patterns. Finally, we normalize all grid values so they fall between -1 and 1.

The encoder $f$ is a Multi-Layer Perceptron (MLP). It takes the flattened 1600-dimensional input and processes it through three hidden layers. Each hidden layer has 256 neurons and uses a LeakyReLU activation function. The final layer outputs a 1-dimensional scalar embedding. 

To create positive pairs for contrastive learning, we apply a specific augmentation pipeline to the input fields. The first view is the original data point. For the second view, we apply two consecutive transformations
\begin{itemize}
    \item We apply a small, random amount of forward heat evolution using Gaussian convolution.
    \item We apply a random shift of up to 20 pixels in the horizontal and vertical directions. The edges wrap around seamlessly. 
\end{itemize}

We train the model using a contrastive loss with a temperature parameter of $\tau=0.1$. We optimize the network using Stochastic Gradient Descent with a learning rate of $0.01$. 
It shows several examples of the smooth, low-frequency fields from Label 0 next to the tighter, high-frequency patterns from Label 1.

\subsection{Text Contrastive Learning Experiment}

We use the IMDB movie review dataset \cite{maas2011learning} for our natural language experiments. We group the dataset into positive reviews (Label 0) and negative reviews (Label 1). We randomly select $100$ reviews from each sentiment class, giving us a total of $N=200$ text samples. 

We process the raw text using a simple tokenizer that only keeps alphanumeric characters and apostrophes, converting everything to lowercase. From this training data, we build a vocabulary of the $5000$ most common words, reserving special tokens for padding, unknown words, and masking. We convert each review into a sequence of token IDs, either padding it with zeros or truncating it to exactly $1024$ tokens.

The encoder $f$ is an encoder-only Transformer model. It maps the discrete tokens into a continuous space using a learned embedding layer with $d=128$ dimensions. We add learned positional embeddings to preserve the word order. The core Transformer architecture consists of two stacked encoder layers. Each layer uses 4 attention heads and a feedforward network dimension of 512, with a GELU activation function. We apply a 10\% dropout rate to both the general Transformer blocks and the attention weights specifically.

To ensure the model learns semantic features rather than relying on sequence length, we use a masked mean pooling strategy. We average the hidden states across the sequence, but we strictly ignore any padding tokens:
$$
\bar h(x) = \frac{\sum_{\ell=1}^L m_\ell(x)\,h_\ell(x)}{\sum_{\ell=1}^L m_\ell(x)}
$$
Here, $h_\ell(x)$ is the output of the Transformer at position $\ell$, and $m_\ell(x)$ is a binary mask that is zero for padding tokens. We pass this pooled representation through a small fully connected head (hidden size of 256) to produce the final scalar embedding $f(x) \in \mathbb{R}$.

To create positive pairs for contrastive learning, we apply a text-specific augmentation pipeline. We dynamically generate two new augmented views for every sample at each training step. We carefully design these augmentations to preserve the core sentiment while removing grammatical structure and varying the length. The pipeline includes
\begin{itemize}
    \item  We take a random contiguous slice of the sequence (between 85\% and 100\% of the original length) before applying padding. 
    \item  There is a 30\% chance to randomly reorder the sentences within the review.
    \item  We randomly delete or mask entire contiguous blocks of words.
    \item  We randomly drop individual tokens, but we make it 2.5 times more likely to drop common, uninformative stopwords like "the" or "and."
\end{itemize}

Crucially, we maintain a hardcoded lexicon of "protected" words. This includes negations (like "not" or "never") and strong polarity words (like "excellent" or "terrible"). We strictly prevent the augmentation pipeline from masking, dropping, or swapping these protected words, ensuring the fundamental sentiment of the review remains identical across all views.

We train the model using a contrastive loss with a temperature of $\tau=0.1$. We use the AdamW optimizer with a learning rate of $10^{-5}$ and a weight decay of $0.01$. 

\subsection{Extra numerical results}

In this section, we extend the experiment to higher-dimensional settings, where the feature distribution lies in $\mathbb{R}^2$ and $\mathbb{R}^3$ instead of $\mathbb{R}$, and present additional results to support \Cref{prop:nn-toy-map} and \Cref{thm:lin_sep}.  Furthermore, we constrain the embedding maps to lie in the set $\mathcal{C} = \{f:\mathbb{R}^D \rightarrow \mathbb{S}^{d-1}\}$, where we consider $d = 2, 3$, making the setup closer to the original NT-Xent loss in \cref{eq:cost-discrete}. All experiments in the paper are conducted using one NVIDIA V100 GPU. 

These experiments further illustrate how neural network optimization influences training dynamics. We also compare the behavior of vanilla gradient descent (assuming that $f$ is fixed at time $t=0$ as in \cref{eq:inv}) with that of gradient descent through a neural network. Despite the change in output dimension, the same phenomenon persists: neural network training guides the dynamics toward stationary points that reflect the underlying clustering structure of the data. In contrast, vanilla gradient descent, which lacks the representational capacity of the neural network, remains insensitive to the data structure and fails to produce meaningful separation. 

In \Cref{fig:vec-comparison}, we use a simple artificial dataset in $\mathbb{R}^2$ consisting of four clusters aligned along the $x$-axis, with each cluster generated from a Gaussian distribution centered at $(-1,0)$, $(0,0)$, $(1,0)$, and $(2,0)$, respectively. Points are colored according to their cluster labels, and arrows indicate the direction of the negative gradient computed at each point. The key parameters are set as follows: $\tau=0.2$, $\delta=0.3$, and $n=2{,}000$. We use stochastic gradient descent with a learning rate of $0.001$ to optimize the model. The training dynamics are examined using the contrastive loss, comparing two optimization strategies: vanilla gradient descent and gradient descent through a neural network.

For vanilla gradient descent (Row 1 in \Cref{fig:vec-comparison}), the training dynamics show that the data points gradually spread until they form a uniformly dispersed distribution on a sphere, effectively erasing the initial clustering structure. This behavior aligns with \Cref{prop:nn-toy-map}, which asserts that the gradient of the loss function is independent of the input structure.

In contrast, neural network optimization produces linearly separable feature representations early in training, with clusters becoming distinctly separated by hyperplanes, thereby confirming \Cref{thm:lin_sep}. Despite random initialization, the contrastive loss gradient effectively guides the features so that the error terms of order $\delta$ remain negligible over short intervals.

\begin{figure}[!t]
    \centering
    
    \begin{minipage}[t]{0.20\linewidth}
        \centering
        \includegraphics[width=\linewidth]{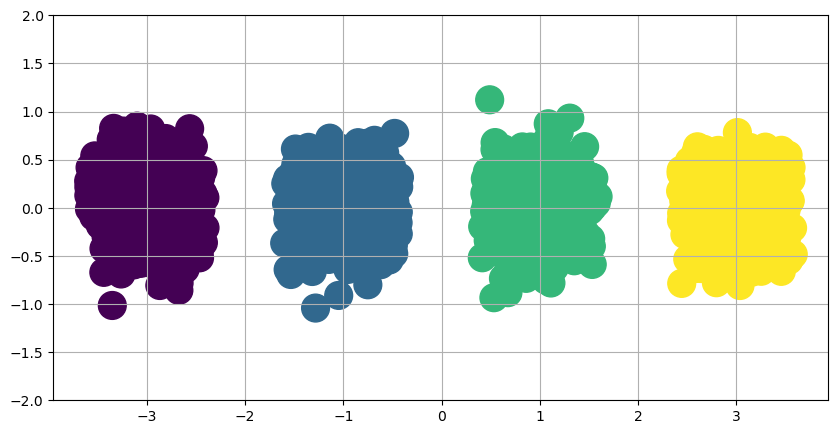}
        \caption*{(a) Input data}
    \end{minipage}\hfill
    \begin{minipage}[t]{0.75\linewidth}
        \centering

        \begin{minipage}[t]{1.0\linewidth}
            \centering
            \begin{minipage}[t]{0.19\linewidth}
                \centering
                \includegraphics[width=\linewidth]{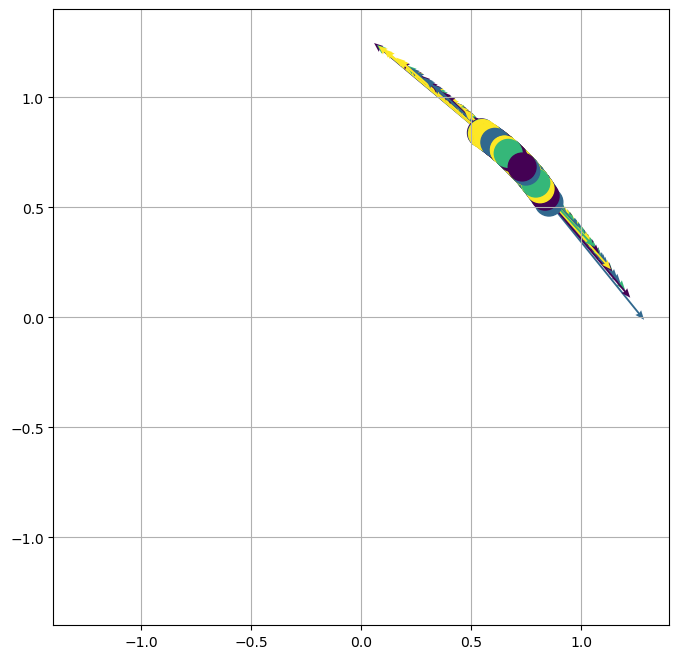}
                \caption*{(b) It: 0}
            \end{minipage}\hfill
            \begin{minipage}[t]{0.19\linewidth}
                \centering
                \includegraphics[width=\linewidth]{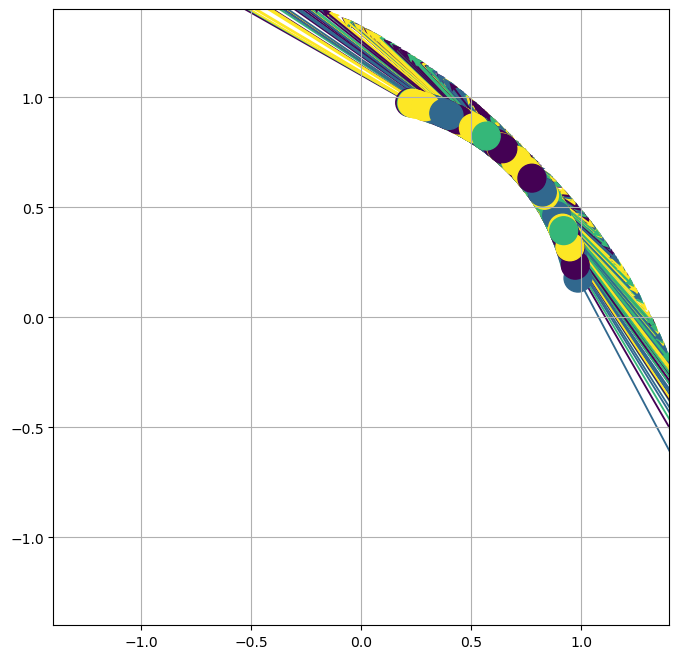}
                \caption*{(c) It: 30}
            \end{minipage}\hfill
            \begin{minipage}[t]{0.19\linewidth}
                \centering
                \includegraphics[width=\linewidth]{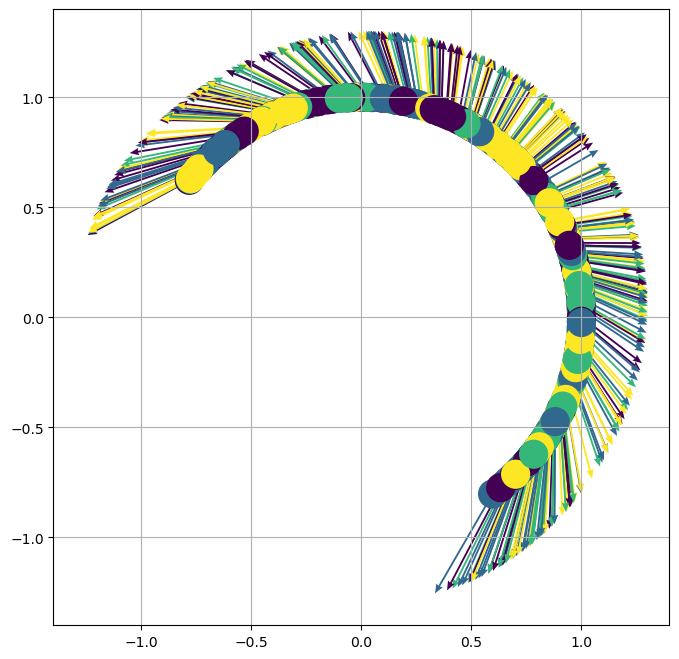}
                \caption*{(d) It: 200}
            \end{minipage}\hfill
            \begin{minipage}[t]{0.19\linewidth}
                \centering
                \includegraphics[width=\linewidth]{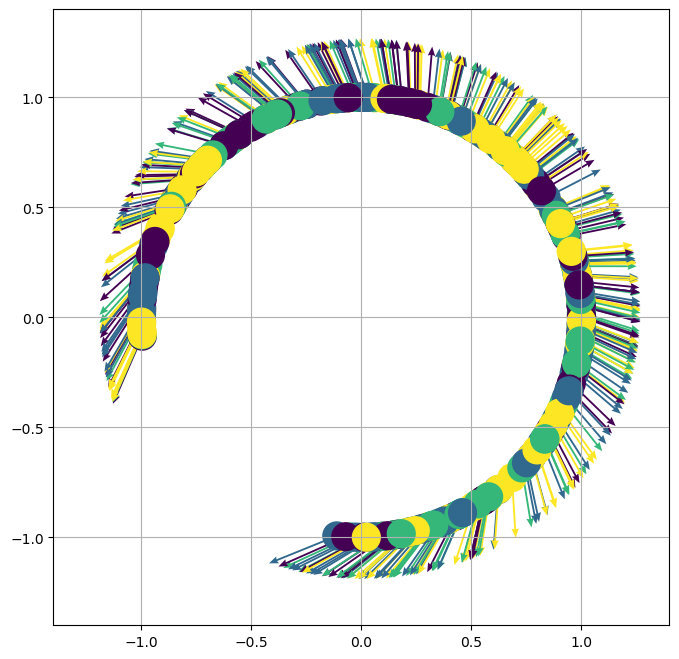}
                \caption*{(e) It: 400}
            \end{minipage}\hfill
            \begin{minipage}[t]{0.19\linewidth}
                \centering
                \includegraphics[width=\linewidth]{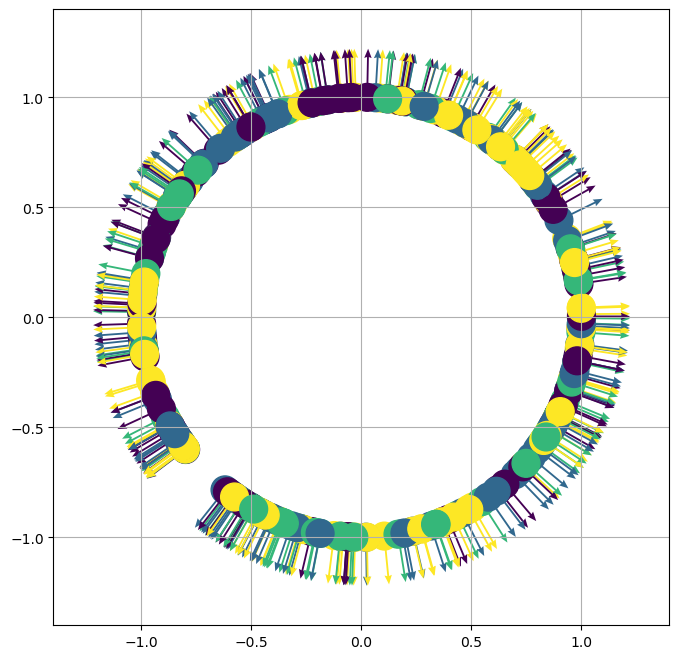}
                \caption*{(f) It: 2,000}
            \end{minipage}
        \end{minipage}

        \vspace{0.5cm}

        \begin{minipage}[t]{1.0\linewidth}
            \centering
            \begin{minipage}[t]{0.19\linewidth}
                \centering
                \includegraphics[width=\linewidth]{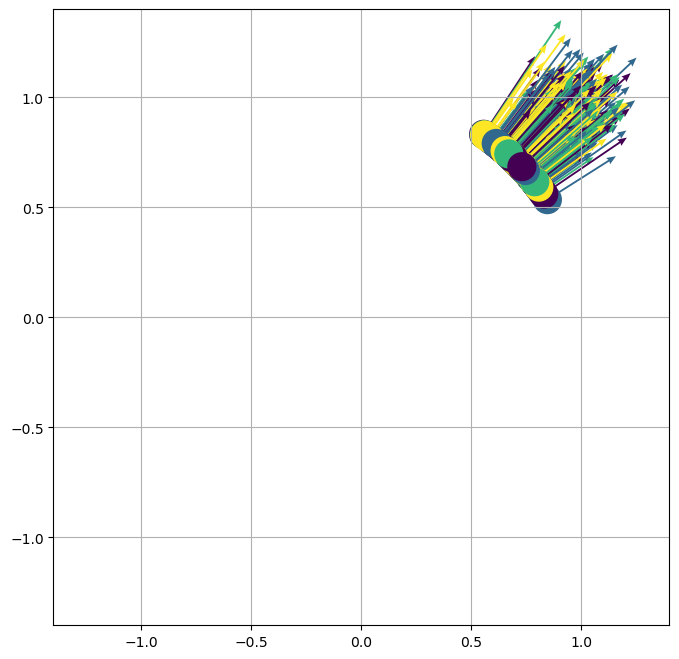}
                \caption*{(g) It: 0}
            \end{minipage}\hfill
            \begin{minipage}[t]{0.19\linewidth}
                \centering
                \includegraphics[width=\linewidth]{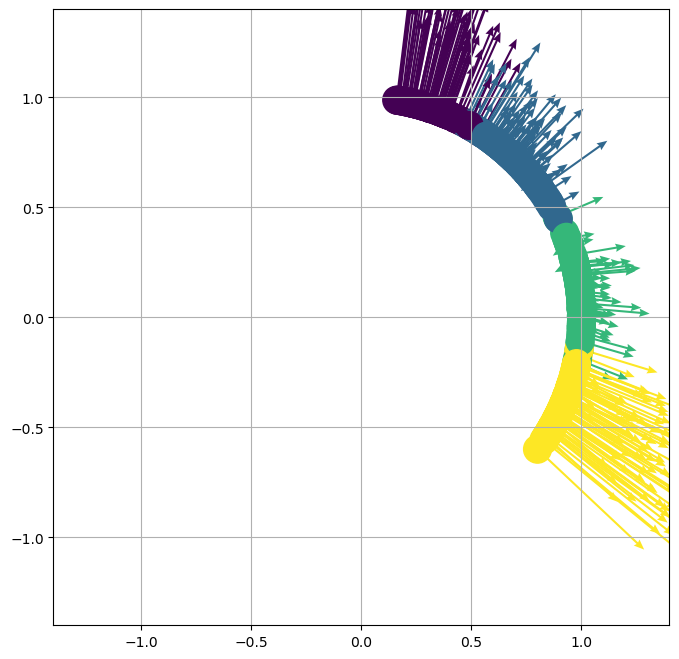}
                \caption*{(h) It: 40}
            \end{minipage}\hfill
            \begin{minipage}[t]{0.19\linewidth}
                \centering
                \includegraphics[width=\linewidth]{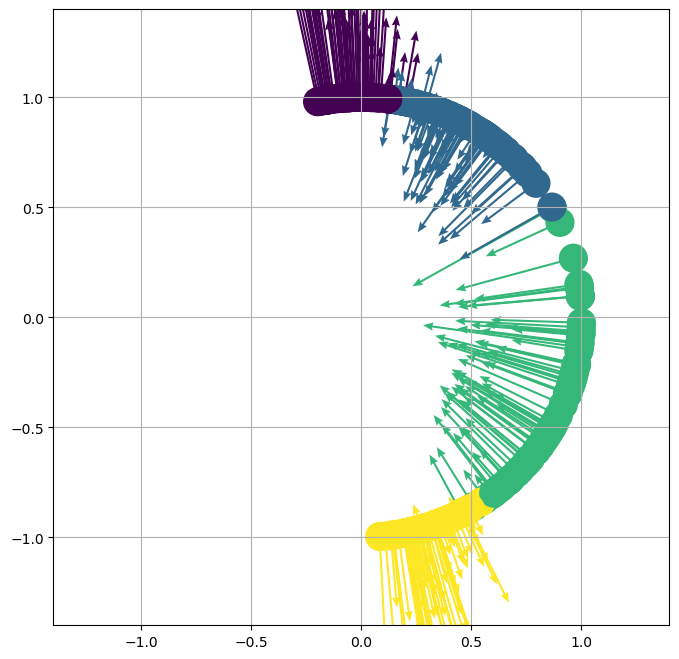}
                \caption*{(i) It: 80}
            \end{minipage}\hfill
            \begin{minipage}[t]{0.19\linewidth}
                \centering
                \includegraphics[width=\linewidth]{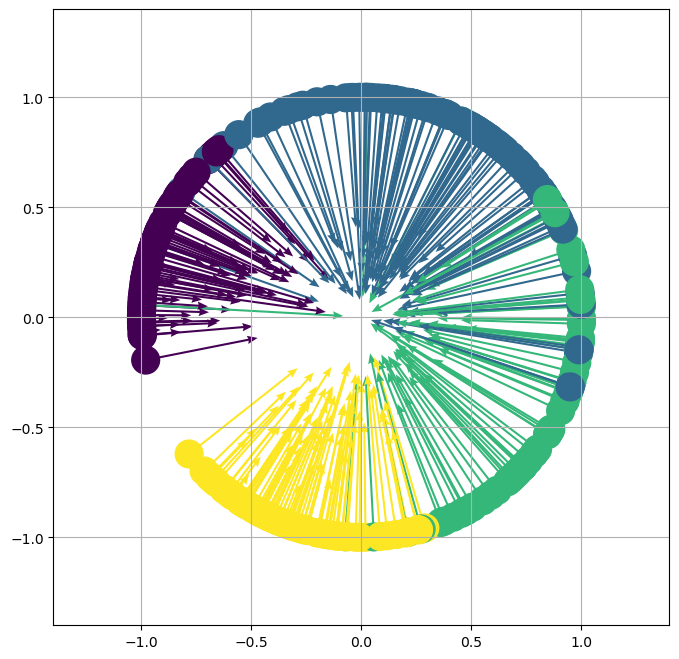}
                \caption*{(j) It: 300}
            \end{minipage}\hfill
            \begin{minipage}[t]{0.19\linewidth}
                \centering
                \includegraphics[width=\linewidth]{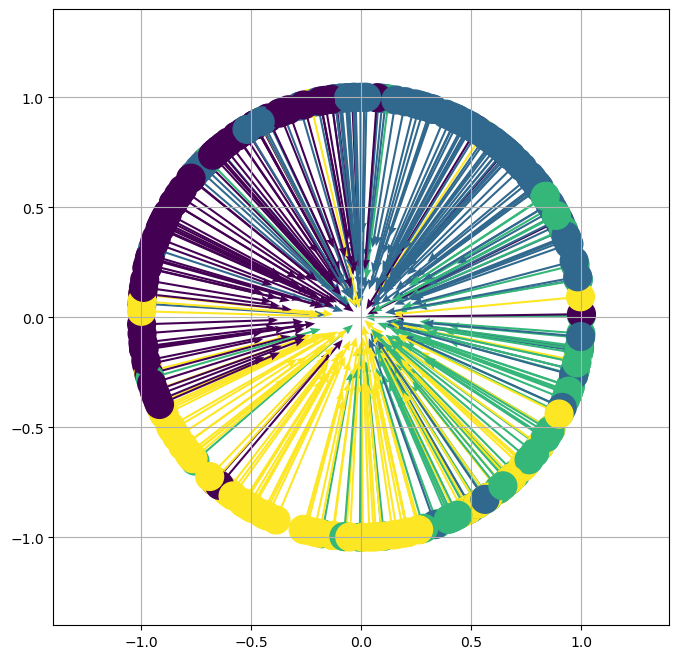}
                \caption*{(k) It: 2,000}
            \end{minipage}
        \end{minipage}

    \end{minipage}
    
    \caption{
        Comparison of the optimization process with and without neural network training. In each panel, points are colored according to their cluster assignments based on the input data, and arrows denote the negative gradient. Row 1 shows optimization using vanilla gradient descent, where the distribution eventually becomes uniformly dispersed, disregarding the clustering structure of the input data. Row 2 shows optimization with neural network training, where the clustering structure becomes linearly separable in the early iterations. This outcome is consistent with \Cref{prop:w-gf}.
    }
    \label{fig:vec-comparison}
\end{figure}

\Cref{fig:appendix-3d} compares the optimization processes with and without neural network training in both 2D and 3D, using different neural network architectures from that used in the main text and earlier figures. While the main text and preceding experiments use a simple one-hidden-layer neural network, \Cref{fig:appendix-3d} employs a deeper 4-layer fully connected neural network. This demonstrates that the same phenomenon, emergence of clustering under neural network training, persists regardless of the network architecture.

The initial data distributions are shown in panels (a) and (l). The color of each point corresponds to its respective cluster. Rows 1 and 3 show the optimization with neural network training, starting from a random initial embedding and progressively revealing the clustering structure over time. In contrast, rows 2 and 4 display the optimization with vanilla gradient descent, where the feature distribution gradually converges to a uniform arrangement, disregarding the clustering structure present in the input data.

\begin{figure}[!ht]
    \centering

    \begin{minipage}[c]{0.18\linewidth}
        \centering
        \includegraphics[width=\linewidth]{figures/data.png}
        \caption*{(a) Input data (2D)}
    \end{minipage}\hfill
    \begin{minipage}[c]{0.8\linewidth}
        \begin{minipage}[t]{0.24\linewidth}
            \centering
            \includegraphics[width=\linewidth]{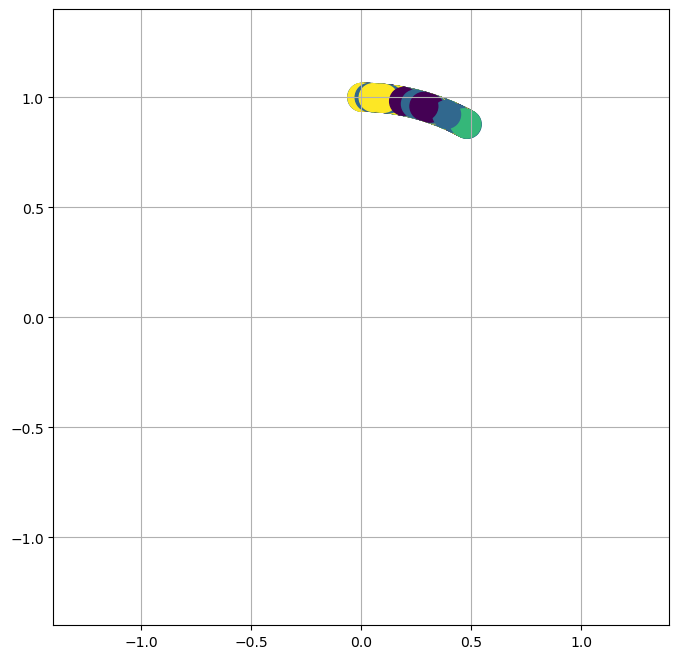}
            \caption*{(b) It: 0}
        \end{minipage}\hfill
        \begin{minipage}[t]{0.24\linewidth}
            \centering
            \includegraphics[width=\linewidth]{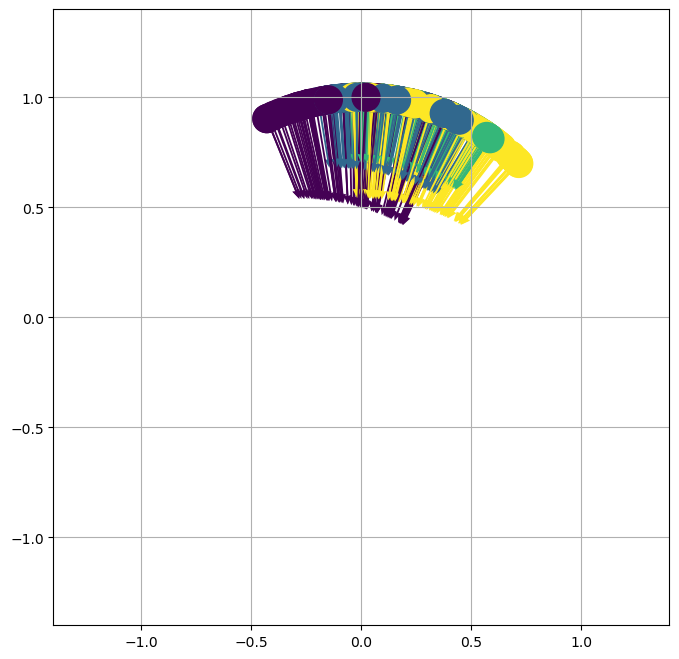}
            \caption*{(c) It: 70}
        \end{minipage}\hfill
        \begin{minipage}[t]{0.24\linewidth}
            \centering
            \includegraphics[width=\linewidth]{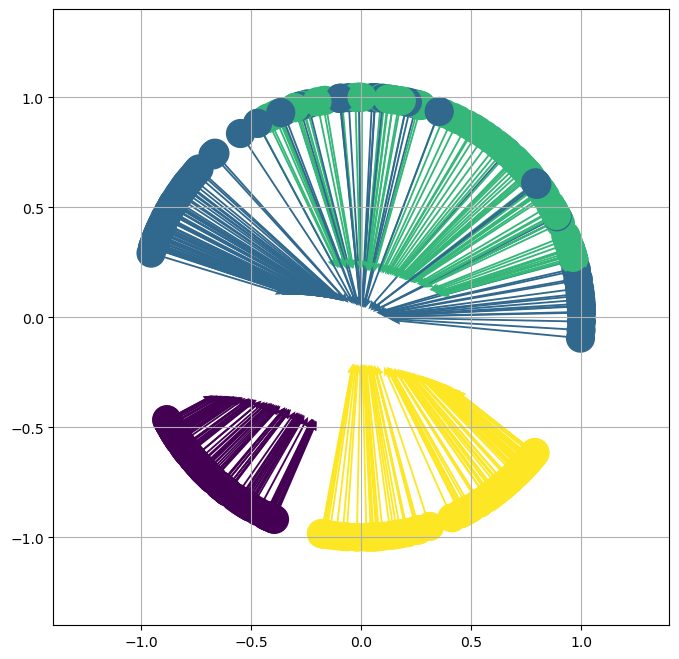}
            \caption*{(d) It: 100}
        \end{minipage}\hfill
        \begin{minipage}[t]{0.24\linewidth}
            \centering
            \includegraphics[width=\linewidth]{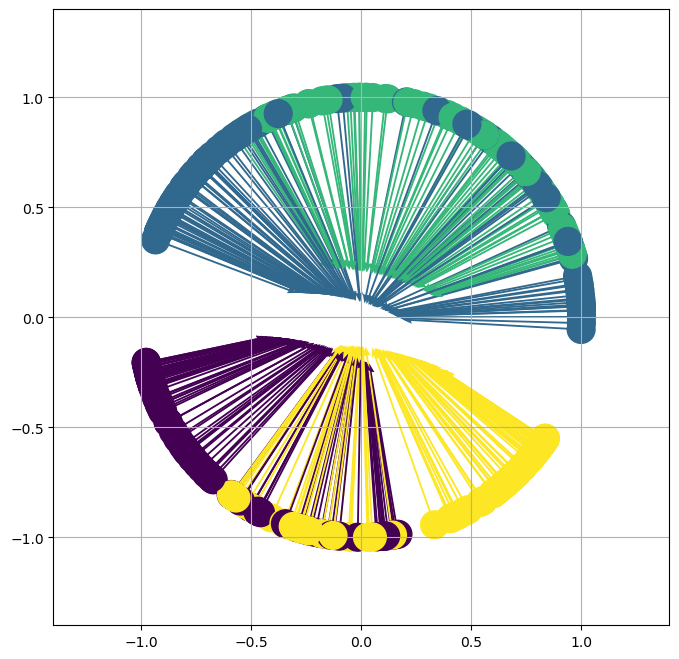}
            \caption*{(e) It: 500}
        \end{minipage}

        \vspace{0.8em}

        \begin{minipage}[t]{0.24\linewidth}
            \centering
            \includegraphics[width=\linewidth]{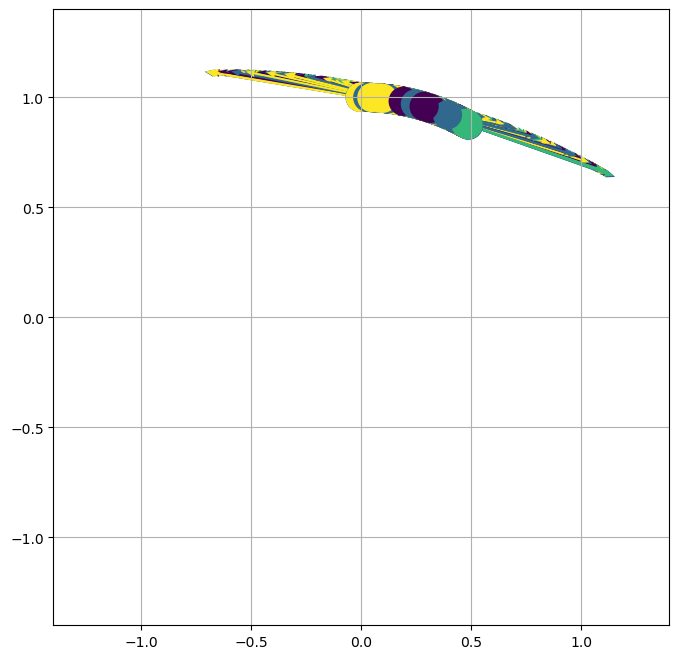}
            \caption*{(f) It: 0}
        \end{minipage}\hfill
        \begin{minipage}[t]{0.24\linewidth}
            \centering
            \includegraphics[width=\linewidth]{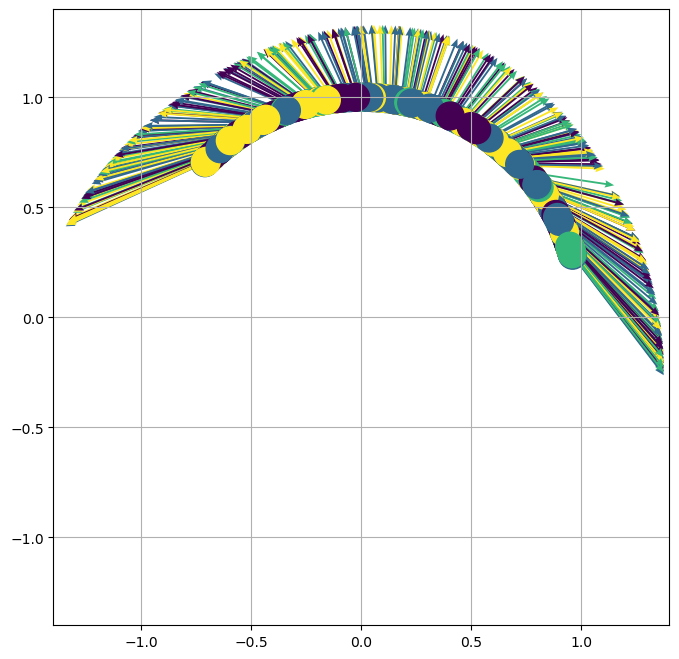}
            \caption*{(g) It: 50}
        \end{minipage}\hfill
        \begin{minipage}[t]{0.24\linewidth}
            \centering
            \includegraphics[width=\linewidth]{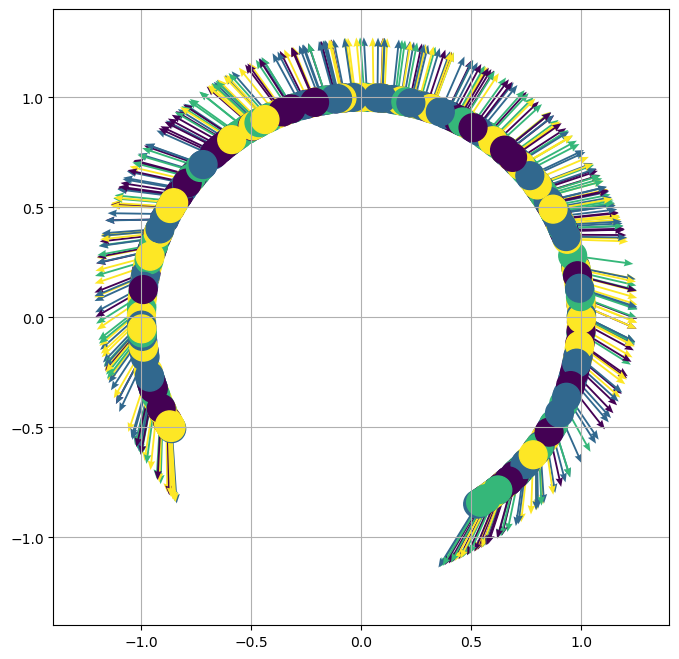}
            \caption*{(h) It: 400}
        \end{minipage}\hfill
        \begin{minipage}[t]{0.24\linewidth}
            \centering
            \includegraphics[width=\linewidth]{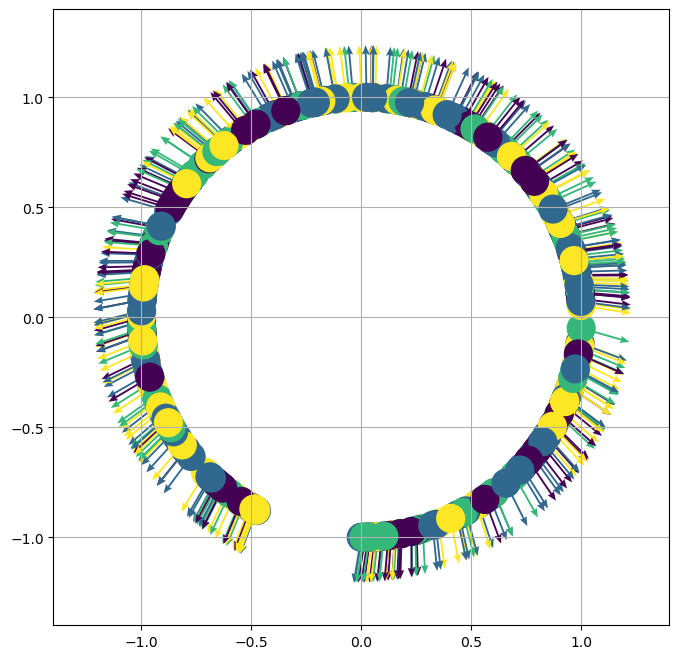}
            \caption*{(i) It: 800}
        \end{minipage}
    \end{minipage}

    \vspace{1.5em}
    \hrule
    \vspace{1.5em}

    \begin{minipage}[c]{0.18\linewidth}
        \centering
        \includegraphics[width=\linewidth]{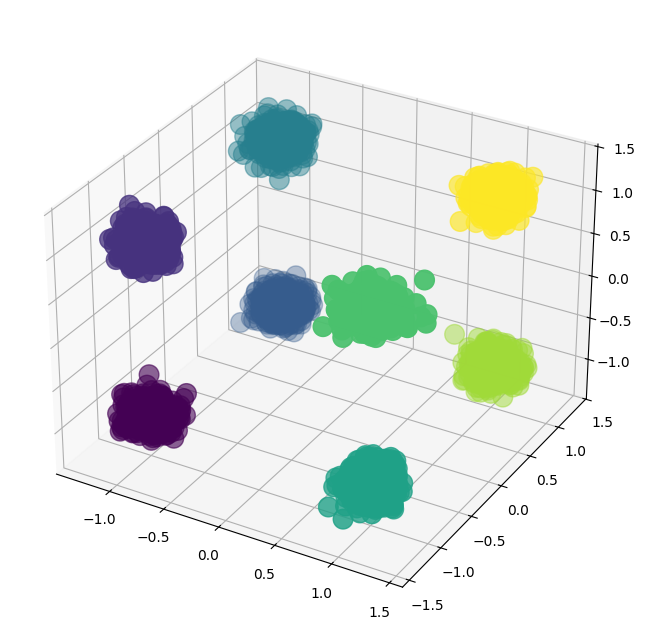}
        \caption*{(j) Input data (3D)}
    \end{minipage}\hfill
    \begin{minipage}[c]{0.8\linewidth}
        \begin{minipage}[t]{0.24\linewidth}
            \centering
            \includegraphics[width=\linewidth]{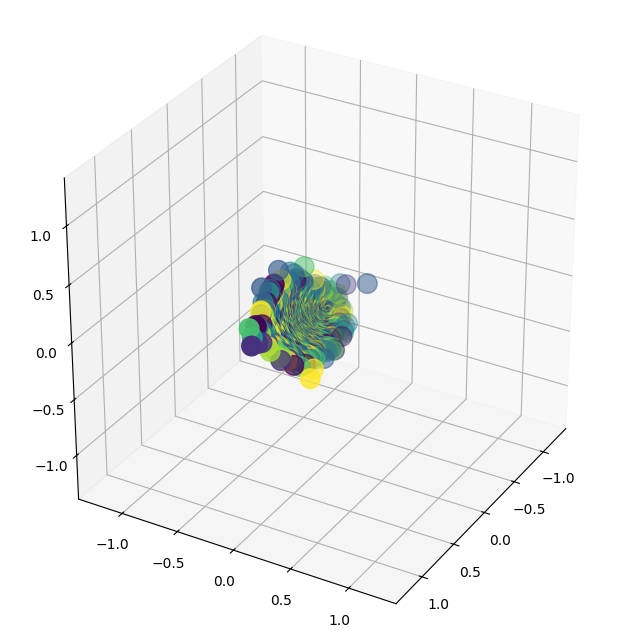}
            \caption*{(k) It: 0}
        \end{minipage}\hfill
        \begin{minipage}[t]{0.24\linewidth}
            \centering
            \includegraphics[width=\linewidth]{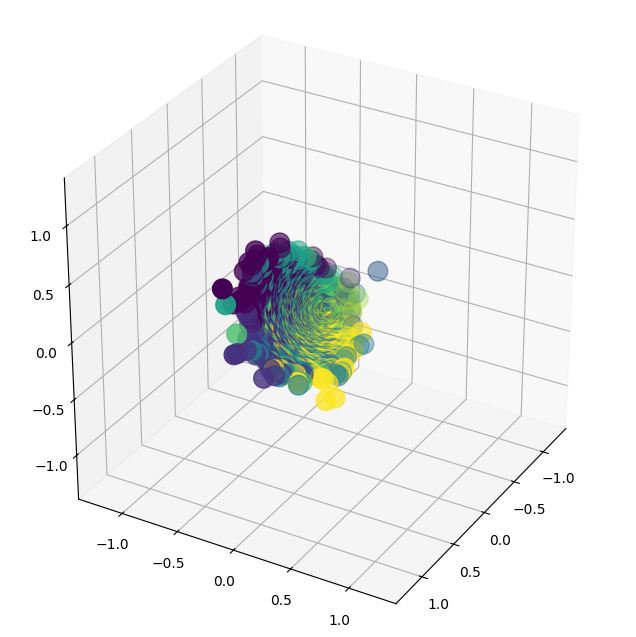}
            \caption*{(l) It: 700}
        \end{minipage}\hfill
        \begin{minipage}[t]{0.24\linewidth}
            \centering
            \includegraphics[width=\linewidth]{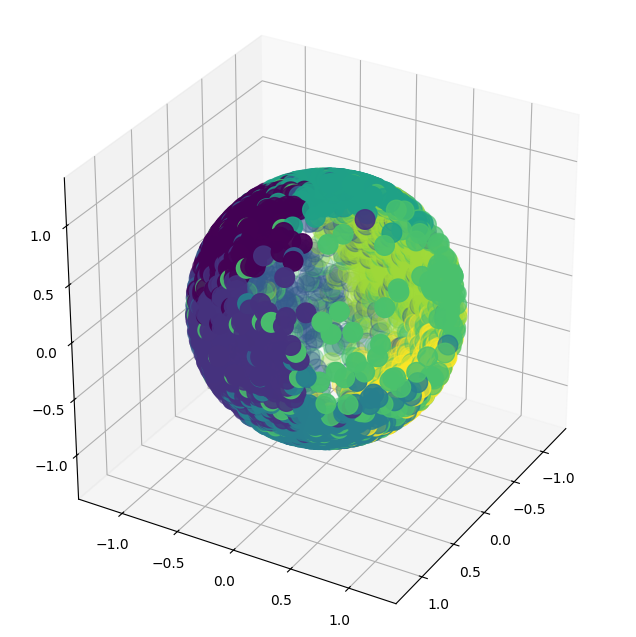}
            \caption*{(m) It: 1,500}
        \end{minipage}\hfill
        \begin{minipage}[t]{0.24\linewidth}
            \centering
            \includegraphics[width=\linewidth]{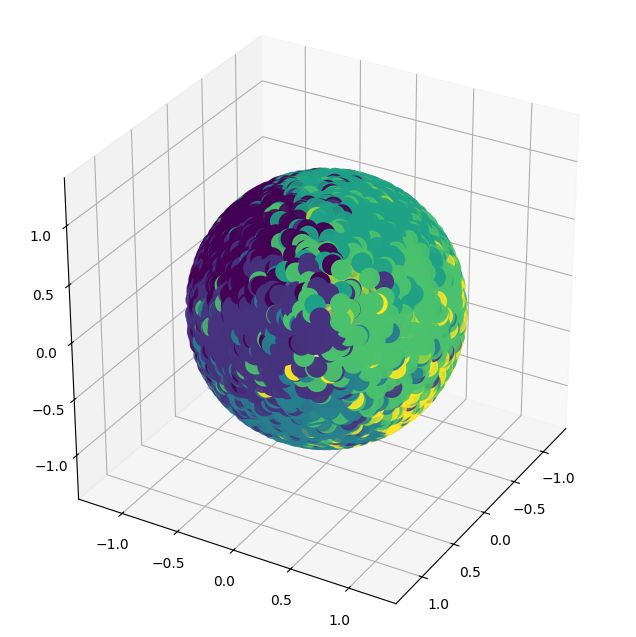}
            \caption*{(n) It: 2,000}
        \end{minipage}

        \vspace{0.8em}

        \begin{minipage}[t]{0.24\linewidth}
            \centering
            \includegraphics[width=\linewidth]{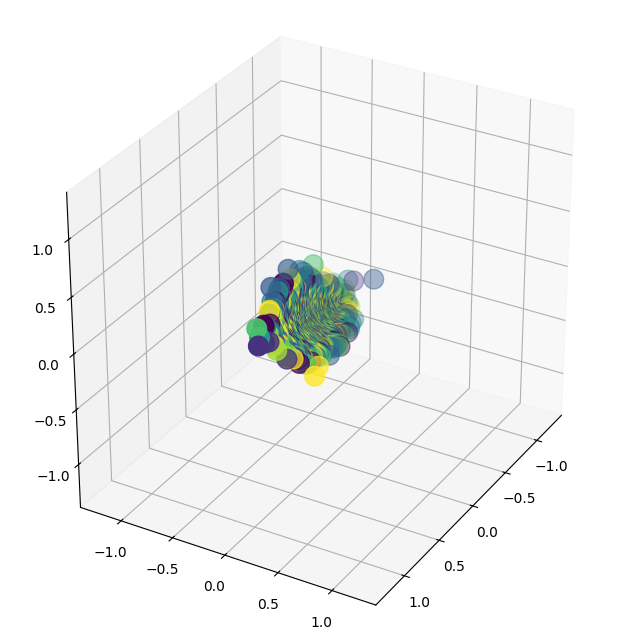}
            \caption*{(o) It: 0}
        \end{minipage}\hfill
        \begin{minipage}[t]{0.24\linewidth}
            \centering
            \includegraphics[width=\linewidth]{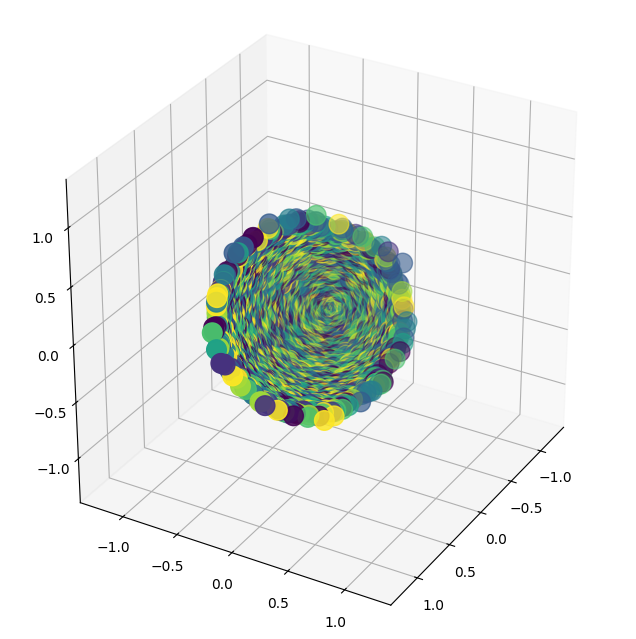}
            \caption*{(p) It: 50}
        \end{minipage}\hfill
        \begin{minipage}[t]{0.24\linewidth}
            \centering
            \includegraphics[width=\linewidth]{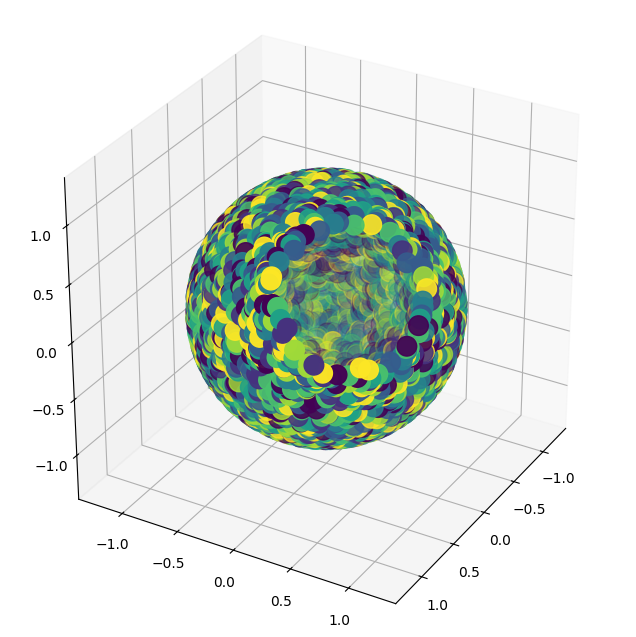}
            \caption*{(q) It: 1,000}
        \end{minipage}\hfill
        \begin{minipage}[t]{0.24\linewidth}
            \centering
            \includegraphics[width=\linewidth]{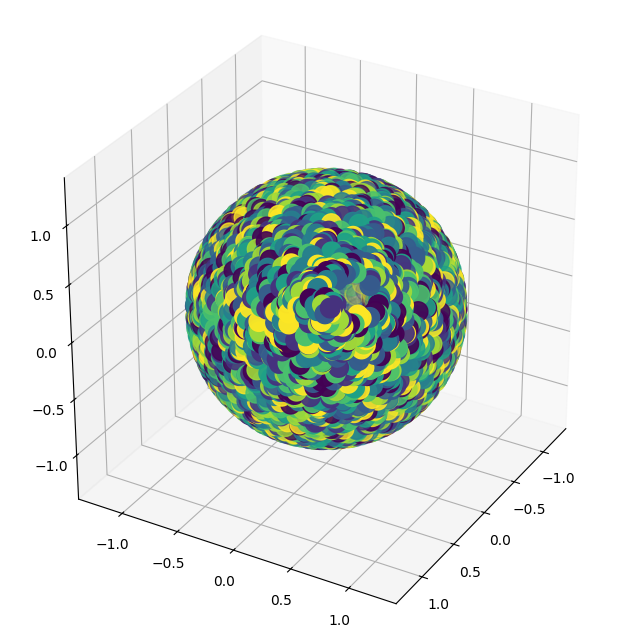}
            \caption*{(r) It: 2,000}
        \end{minipage}
    \end{minipage}

    \caption{
        This experiment compares the optimization processes with and without neural network training in 2D and 3D, with the data distribution depicted in (a) and (j). A 4-layer fully connected neural network demonstrates consistent outcomes as in \Cref{fig:vec-comparison}. Each point's color indicates its cluster. For both the 2D and 3D blocks, the top rows show optimization with neural network training, starting from a random embedding and gradually revealing the clustering structure. In contrast, the bottom rows illustrate the optimization process using vanilla gradient descent, which converges to a uniformly dispersed arrangement, disregarding the input data's clustering structure.
    }
    \label{fig:appendix-3d}
\end{figure}

\bibliographystyle{unsrt}
\bibliography{ref}

\end{document}